\documentclass{amsart}
\usepackage{amsmath,amssymb,amsthm,amsfonts,enumerate,array,color,lscape,fancyhdr,layout,pst-all}
\usepackage[english]{babel}
\usepackage[latin1]{inputenc}
\usepackage{young}

\usepackage[hcentermath]{youngtab}
\usepackage{graphicx}
\usepackage{float}
\usepackage{pstricks, mathtools}

\newcounter{numberofremark}
\newcommand\nothing[1]{}
\usepackage[active]{srcltx}
\usepackage[hyperindex,plainpages,bookmarksnumbered]{hyperref}

\newcommand{\dcl}{\DeclareMathOperator}
\dcl\cdet{cdet} \dcl\Sp{Specm} \dcl\depth{depth} \dcl\im{Im} \dcl\Span{Span} \dcl\id{id} \dcl\Ker{Ker} \dcl\Specm{Specm}
\dcl\Supp{Supp} \dcl\codim{codim} \dcl\Y{Y} \dcl\gl{\mathfrak{gl}}    \dcl\U{U} \dcl\T{T} \dcl\soc{soc} 
\dcl\qdet{qdet} \dcl\sgn{sgn} \dcl\gr{gr} \dcl\diag{diag}
\dcl\g{\mathfrak{g}} \dcl\C{\mathbb C} \dcl\dd{{\mathrm d}}
 \dcl\Tab{Tab}

\newcommand\sm{{\mathsf m}}

\newcommand\Ga{{\Gamma}}

\def\cD{\mathcal D}

\newlength\yStones
\newlength\xStones
\newlength\xxStones
\makeatletter
\def\Stones{\pst@object{Stones}}
\def\Stones@i#1{%
  \pst@killglue%
  \begingroup%
  \use@par%
  \setlength\xxStones{\xStones}%
  \expandafter\Stones@ii#1,,\@nil
  \endgroup
  \global\addtolength\xStones{0.6cm}%
  \global\addtolength\yStones{-7.5mm}}%
\def\Stones@ii#1,#2,#3\@nil{%
  \rput(\xxStones,\yStones){%
    \psframebox[framesep=0]{%
      \parbox[c][6mm][c]{11mm}{\makebox[11mm]{$#1$}}}}%
  \addtolength\xxStones{1.2cm}%
  \ifx\relax#2\relax\else\Stones@ii#2,#3\@nil\fi}
\makeatother

\def\Stone#1{\fbox{\makebox[11mm]{\strut#1}}\kern2pt}

\newtheorem{theorem}{Theorem}[section]
\newtheorem{lemma}[theorem]{Lemma}
\newtheorem{corollary}[theorem]{Corollary}
\newtheorem{proposition}[theorem]{Proposition}
\newtheorem{example}[theorem]{Example}
\newtheorem{remark}[theorem]{Remark}

\newtheorem{definition}[theorem]{Definition}

\begin{document}

\title{Classification of simple Gelfand-Tsetlin modules of $\displaystyle \mathfrak{sl}(3)$}

\author{Vyacheslav Futorny}
\address{Instituto de Matem\'atica e Estat\'istica, Universidade de S\~ao
Paulo,  S\~ao Paulo, Brasil} \email{futorny@ime.usp.br}

\author{Dimitar Grantcharov}
\address{\noindent
University of Texas at Arlington,  Arlington, TX 76019, USA} \email{grandim@uta.edu}

\author{Luis Enrique Ramirez}
\address{\noindent
Universidade Federal do ABC,  Santo Andr\'e SP, Brasil} \email{luis.enrique@ufabc.edu.br}
\date
\makeatletter
\def\@roman#1{\romannumeral#1}
\makeatother

\begin{abstract}
We provide a classification and an explicit realization of all simple Gelfand-Tsetlin modules of the complex Lie algebra $\mathfrak{sl}(3)$. The realization of these modules, including those with infinite-dimensional weight spaces, is provided via  regular and derivative Gelfand-Tsetlin tableaux.  Also, we show that all simple Gelfand-Tsetlin $\mathfrak{sl}(3)$-modules can be obtained as subquotients of localized Gelfand-Tsetlin $E_{21}$-injective modules.
\end{abstract}
\maketitle
 
\noindent\textbf{MSC 2010 Classification:} 16G99, 17B10.\\
\noindent\textbf{Keywords:} Gelfand-Tsetlin modules, Gelfand-Tsetlin bases, weight modules,
localization.

\tableofcontents
\section{Introduction}

Let ${\mathfrak g}$ be a finite-dimensional simple Lie algebra over the
complex numbers, and let ${\mathfrak h}$ be a Cartan
subalgebra of ${\mathfrak g}$. A ${\mathfrak g}$-module $M$ is called a {\it
weight module} if   $ M=\bigoplus_{\lambda\in {\mathfrak h}^*} M_\lambda$, where $M_\lambda=\{v\in M\mid hv=\lambda(h)v,\
\forall h\in {\mathfrak h}\}$. The space $M_\lambda$ is called a \emph{weight space}, the set $\{\lambda\in {\mathfrak h}^*\ |\ 
M_{\lambda}\neq 0\}$ is called the {\it weight support} of $M$ and the
dimension of $M_{\lambda}$ is called the {\it weight mltiplicity} of
$\lambda$.  If $\mathfrak h$ does not necessarily act diagonally but acts locally finitely on $M$, 
 then we say that  $M$ is a {\it generalized
weight module}.  It is an easy exercise to show that a simple generalized weight module is a weight module.

A weight module $M$ is {\it torsion free} provided that all root vectors of
 ${\mathfrak g}$ act injectively on $M$. If $M$ is a torsion free module then all weight 
 multiplicities of $M$ (finite or infinite) are equal. This invariant of $M$ is
called the {\it weight degree} of $M$. Furthermore, the weight support of a torsion free module $M$ coincides with a full coset $\lambda + Q$ of ${\mathfrak h}^*/Q$, where $Q$ is the root lattice of $\mathfrak g$ and $\lambda$ is in the weight support of $M$. On the other hand, a simple weight module
may have ``full support'' without being torsion free, in which case the 
weight multiplicities are necessarily infinite. The first 
examples of such modules were given in \cite{Fut1}. Simple modules with
full support are called {\it dense}.

A breakthrough in the  theory of weight modules with finite weight multiplicities was made 
by Fernando, \cite{Fe}, in 1990 who reduced the classification of all such simple modules to determining the simple torsion free
modules. He also showed that  the only simple Lie algebras admitting
torsion free modules are those of type $A$ or $C$. The next major breakthrough was made in 2000 by Mathieu,
\cite{M}, who classified and provided a realization of all
simple torsion free  modules of finite degree.
Previously, the case of degree $1$ was worked out in \cite{Britten Lemire 3}. Important properties of the annihilators of the torsion free modules were established in \cite{Jos92}.

The study of weight modules with infinite multiplicities is still at its initial stage. A result similar to the one of Fernando 
reduces  the classification of all such
simple modules to the classification of all simple dense modules
of simple Lie algebras. For the classical simple
Lie algebras this reduction was obtained  in \cite{Fut3} and
for all exceptional simple Lie algebras except $E_8$ in
\cite{FutOvsTsy}. Finally, in \cite{DimMathPen}, the reduction
in all cases, including all important classes of finite-dimensional Lie superalgebras, was completed.

The classification of all simple dense modules, possibly with infinite weight multiplicities, remains out of reach. The latter classification is present only in the case of $\mathfrak{sl}(2)$, where in the fundamental paper of Block, \cite{Block},  all simple  $\mathfrak{sl}(2)$-modules are classified, and in particular, it is shown that the simple dense modules have always weight degree $1$.

One natural category of weight modules is the category of Gelfand-Tsetlin modules. More precisely, this is the full subcategory of the category of generalized weight modules consisting of modules that admit a generalized eigenbasis for the \emph{Gelfand-Tsetlin subalgebra}, a maximal commutative subalgebra of the universal enveloping algebra $U(\mathfrak g)$ of $\mathfrak g$. 
 Gelfand-Tsetlin modules  were introduced in
\cite{DFO2}, \cite{DFO1}, \cite{DFO3} as an attempt to generalize the celebrated tableaux construction of Gelfand-Tsetlin bases of irreducible finite-dimensional representations of
simple classical Lie algebras, \cite{GT}, \cite{m:gtsb}, \cite{zh:cg}.

Gelfand-Tsetlin subalgebras have applications that extend beyond the study of Gelfand-Tsetlin modules. For example, these subalgebras were related to the solutions of the
Euler equation in \cite{FM}, and to the subalgebras of $U(\mathfrak g)$ of
maximal Gelfand-Kirillov dimension in \cite{Vi}. Gelfand-Tsetlin subalgebras  were studied in \cite{KW-1}, \cite{KW-2} in
connection with classical mechanics, and also in \cite{Gr1}, \cite{Gr2} in
connection with general hypergeometric functions on the Lie group
$GL(n, \mathbb C)$.

A general theory of Gelfand-Tsetlin modules for a class of Galois algebras (for a definition see \cite{FO1}) was developed in \cite{FO2}. The results for these Galois algebras can be applied to the universal enveloping algebras of  $\mathfrak{sl}(n)$ and  $\mathfrak{gl}(n)$, and provide structural properties of the corresponding simple Gelfand-Tsetlin modules. In the \emph{generic} case the characters of the Gelfand-Tsetlin subalgebra parametrize such simple modules. However, in the nongeneric case, i.e. in the \emph{singular} case and $n>2$, we may have more than one isomorphism class of simple Gelfand-Tsetlin modules with a fixed character of the Gelfand-Tsetlin subalgebra. The theory of singular Gelfand-Tsetlin modules was initiated in  \cite{FGR3} where 
$1$-singular modules were constructed and studied in detail. Immediately after the construction of the 1-singular modules, there was an abundance of successful attempts to construct simple Gelfand-Tsetlin modules with a given singular character.  For more details, we refer the reader to the following papers
 \cite{EMV},  \cite{FGR3}, \cite{FGR16a}, 
\cite{FGR17},  \cite{FGRZ},  \cite{FGRZ1}, 
\cite{FK17},  \cite{FRZ16}, \cite{Har17}, \cite{RZ17}, \cite{V17a}, \cite{V17b}, \cite{Za17}.  In poarticular, a  classification of the simple Gelfand-Tsetlin modules was recently announced in 
\cite{KTWWY} and \cite{W}. 

A classification of the simple $1$-singular Gelfand-Tsetlin modules was obtained in \cite{FGR17} and leads to the classification of all simple Gelfand-Tsetlin modules of the Lie algebra $\mathfrak{sl}(3)$ (and of $\mathfrak{gl}(3)$). The latter classification is the main purpose of the present paper and it is provided via very explicit tableaux construction.

Our classification result relies on various old  results on Gelfand-Tsetlin $\mathfrak{sl}(3)$-modules obtained in  \cite{Britten Futorny Lemire}, \cite{Fut1}, \cite{Fut2}, \cite{Fut3}, \cite{Fut4}, combined with newer results from \cite{FGR1} and  \cite{Ramirez Thesis}. We remark that  some technical statements in the paper on the properties of Gelfand-Tsetlin modules can be simplified using the  theory recently developed, for example, in \cite{FGR17}, \cite{FGRZ1}, \cite{RZ17}. However,  for the sake of completeness and for reader's convenience, we opted to keep the original manuscript containing detailed and explicit proofs. 

The structure of the paper is as follows.
In Section $2$  we set up the notation and state  basic
definitions and  results needed in the rest of the paper. In Section $3$  we prove some general results about the Gelfand-Tsetlin modules of $\mathfrak{gl}(n)$. Section $4$  is devoted to the description of certain ``easier to study'' classes of Gelfand-Tsetlin modules of $\mathfrak{gl}(n)$, namely finite-dimensional modules, generic modules, and $1$-singular modules. In Section 5 we collect some important definitions and preliminary results that relate  Gelfand-Tsetlin modules to their Gelfand-Tsetlin character. In Section $6$  we prove the main results about existence and uniqueness of simple Gelfand-Tsetlin modules of $\mathfrak{sl}(3)$. The explicit description of all simple Gelfand-Tsetlin modules for $\mathfrak{sl}(3)$ is included in  Section $7$  . Finally, in Section $8$, we study localization functors on the category of Gelfand-Tsetlin $\mathfrak{sl}(3)$-modules and prove that any simple module in this category can be obtain from an $E_{21}$-injective module using a $E_{21}$-localization functors. 

\medskip

\noindent{\bf Acknowledgements.} D.G. gratefully acknowledges the
hospitality and excellent working conditions at the S\~ao Paulo
University where part of this work was completed. V.F. is
supported in part by the CNPq grant (304467/2017-0) and Fapesp grant (2018/23690-6). D.G. is supported in part by Simons Collaboration Grant 358245 and by Fapesp grant (2014/09310-5). L.E.R. was supported by Fapesp grant (2018/17955-7).

\section{Preliminaries} \label{Section: Preliminaries}

The ground field will be ${\mathbb C}$.  In the first part of the paper we fix an integer $n\geq 2$.  For $a \in {\mathbb Z}$, we write $\mathbb Z_{\geq a}$ for the set of all integers $m$ such that $m \geq a$. Similarly, we define $\mathbb Z_{< a}$, etc. For a Lie algebra ${\mathfrak a}$ by $U(\mathfrak a)$ we denote the universal enveloping algebra of ${\mathfrak a}$. For a commutative ring $R$, ${\rm Specm}\, R$ will stand for  the set of maximal ideals of $R$.  

By $\gl(n)$ we denote the general linear Lie algebra consisting of all $n\times n$ complex matrices, and by $\{E_{i,j}\ |\ 1\leq i,j \leq n\}$  - the standard basis of $\gl(n)$ of elementary matrices. We fix the standard Cartan subalgebra  of $\gl(n)$, the standard triangular decomposition and the corresponding basis of simple roots of $\gl(n)$.  The weights of $\gl(n)$ will be written as $n$-tuples $(\lambda_1,...,\lambda_n)$ through the identification ${\mathfrak h}^* \to {\mathbb C}^n$. 

The Lie subalgebra  ${\mathfrak g}=\mathfrak{sl}(n)$ of $\mathfrak{gl}(n)$ is generated by $\{E_{i,i+1},\ E_{i+1,i}\ |\ 1\leq i \leq
n-1\}$. The standard Cartan subalgebra of ${\mathfrak g}$  will be denoted by $\mathfrak h$, i.e. $${\mathfrak
h}=\hbox{$\Span$}\{h_i=E_{ii}-E_{i+1,i+1}\ |\  i=1,\dots, n-1\}.$$ Let
$\epsilon_i$ denote the projection of a $n\times n$
matrix onto its $(i,i)^{th}$ entry. Then a basis of simple roots
of the root system $\Delta$ of ${\mathfrak g}$ is given by
$\pi=\{\alpha_i=\epsilon_i-\epsilon_{i+1}\ |\  i=1, \dots, n-1\}$ and
the corresponding positive roots are
$\Delta^+=\{\epsilon_i-\epsilon_j=\alpha_i+\cdots+\alpha_{j-1}\mid
i<j\}$.

\subsection{Index of notations}\label{subsection: Index of notations}

\begin{itemize}
\item \S  \ref{Subsection: Gelfand-Tsetlin modules of gl(n)}. $c_{mk}$; $i(c_{mk})$; $\Gamma$; $\gamma_{mk}(l)$; $M(\chi)$; $Supp_{GT}(M)$; $M_{\lambda}$; $\Supp (M)$; $\Gamma(\pi)$; $\mathcal{GT}_{\pi}(n)$.
\item \S \ref{subsection: Gelfand-Tsetlin modules for sl(n)}.  $\mathcal{GT}_{\pi}(n)$.
\item \S \ref{subsection:Gelfand-Tsetlin tableaux}  $T(v)$; $T_{n}(R)$; $\chi_{v}$.
\item \S \ref{subsection: Gelfand-Tsetlin formulas for finite-dimensional modules} $\delta^{ij}$;  $Q_{n}$; $\mathcal{B}(T(v))$; $\mathcal{GT}_{T(v)}(n)$.
\item \S \ref{subsection: Generic modules}. $V(T(v))$; $\tilde{S}_{m}$; $\Phi_{\ell m}$; $\varepsilon_{rs}$; $e_{t}^{(+)}(w)$; $e_{t+1}^{(-)}(w)$; $e_{r,s}(w)$; $\Omega(T(w))$; $\Omega^{+}(T(w))$; $\mathcal{N}(T(w))$; $\mathcal{I}(T(w))$.
\item \S \ref{subsection: Singular Gelfand-Tsetlin modules}. $\mathcal{F}$; $T_{n}(\mathbb{C})_{\rm gen}$; $\mathcal{H}_{ij}^{k}$; $\Phi_{\ell m}(k,t)$; $\mathcal{H}$; $\overline{\mathcal{H}}$; $\mathcal{F}_{ij}$, $V(T(\bar{v}))$; $\mathcal{V}_{\rm gen}$; $\mathcal{S}$; $\mathcal{D}^{\bar{v}}$; $\mathcal{B}(T(\bar{v}))$; $\Tab(w)$. 
\item \S \ref{Section: Gelfand-Tsetlin modules of sl(3)}. $g_{\lambda}(x,y)$; $C(\mathfrak{h})$.
\item \S \ref{subsection: Structure of singular sl(3)-modules V(T(v))}. $\Lambda^{+}(Tab(w))$; $\mathcal{A}(\Tab(w))$; $\mathcal{N}(\Tab(w))$; $\mathcal{A}(\Tab(w))$; $\mathcal{C}(w)$; $\mathcal{R}(w)$; $\mathcal{A}(\Tab(w))$; $\mathcal{N}_{(r,s,p)}(\Tab(w))$; $\mathcal{N}^{(1)}(\Tab(w))$; $\mathcal{N}^{(2)}(\Tab(w))$; $\widehat{\mathcal{A}}(\Tab(w))$; $\widehat{\mathcal{N}}(\Tab(w))$; $\Tab(\mathcal{B})$.
\item \S \ref{subsection: Realizations of all simple singular Gelfand-Tsetlin modules for sl(3)}. $M({\mathcal B})$; $L({\mathcal B})$.
\item \S \ref{Section: Localization functors for Gelfand-Tsetlin modules} $\cD_{ij}$; $\cD^x_{ij}$; $\mathcal{QD}_{ij}$.
\item \S \ref{subsection: Localization functors on sl(3)-case}. $\Theta_{x}(u)$; $\Phi_\alpha^x$.
\item \S \ref{subsection: simple Gelfand-Tsetlin modules and localization functors}. $L_{i}^{(Gj)}$; $L_{i}^{(Cj)}$.
\end{itemize}

\section{Gelfand-Tsetlin modules of $\mathfrak{gl}(n)$ and $\mathfrak{sl}(n)$}\label{Section: Gelfand-Tsetlin modules of gl(n) and sl(n)}

\subsection{Gelfand-Tsetlin modules of $\mathfrak{gl}(n)$}\label{Subsection: Gelfand-Tsetlin modules of gl(n)} Let  for $m\leqslant n$, $\mathfrak{gl}_{m}$ be the Lie subalgebra
of $\gl (n)$ spanned by $\{ E_{ij}\,|\, i,j=1,\ldots,m \}$. We have the following chain
$$\gl_1\subset \gl_2\subset \ldots \subset \gl_n.$$
It induces  the chain $U_1\subset$ $U_2\subset$ $\ldots$ $\subset
U_n$ for the universal enveloping algebras  $U_{m}=U(\gl_{m})$, $1\leq m\leq n$. Let
$Z_{m}$ be the center of $U_{m}$. Then $Z_m$ is the polynomial
algebra in the $m$ variables $\{ c_{mk}\,|\,k=1,\ldots,m \}$,
\begin{equation}\label{Equation: c_mk}
c_{mk } \ = \ \displaystyle {\sum_{(i_1,\ldots,i_k)\in \{
1,\ldots,m \}^k}} E_{i_1 i_2}E_{i_2 i_3}\ldots E_{i_k i_1}.
\end{equation}  The
\emph{(standard) } \index{defs}{Gelfand-Tsetlin subalgebra}
${\Ga}$ in $U$ (\cite{DFO3, GT}) is generated by $\bigcup\limits_{i=1}^{n}Z_i$.

 The algebra ${\Ga}$ is a polynomial
algebra in $\displaystyle \frac{n(n+1)}{2}$ variables $\{
c_{ij}\,|\, 1\leqslant j\leqslant i\leqslant n \}$.
 For
$i=1,\ldots, n$ denote by $S_i$ the $i$-th symmetric group and set
$G=S_n\times \cdots \times S_1$. Let $\Lambda$ be the polynomial
algebra in variables $\{l_{ij}\,|$ $1\leqslant j\leqslant
i\leqslant n \}$.

Let $\imath:{\Ga}{\longrightarrow}$ $\Lambda$ be the embedding
given by $\imath(c_{mk})=\gamma_{mk}(l)$ where 
\begin{equation}\label{Equation: gamma mk}
\gamma_{mk}(l) \  = \ \sum_{i=1}^m
(l_{mi}+m-1)^k \prod_{j\ne i}^{m} \left( 1 -
\frac{1}{l_{mi}-l_{mj}} \right).
\end{equation} 
The image of $\imath$
coincides with the subalgebra of $G-$invariant polynomials  in
$\Lambda$which we identify with $\Ga$, see \cite{zh:cg} for more details.

\begin{remark}\label{Remark: Cartan inside Gamma} Note that $\Gamma$ contains the standard Cartan subalgebra of $\gl(n)$ spanned by $E_{ii}$, $i=1,\ldots, n$. Indeed, $c_{m1}=\sum\limits_{i=1}^{m}E_{ii}$ for each $1\leq m\leq n$. Therefore, $E_{ii}$ belong to $\Gamma$ for each $1\leq i\leq n$.
\end{remark}
\begin{remark} We should note that the polynomials $\gamma_{mk}(l)$ are symmetric of degree $k$ in  variables $l_{m1}, \ldots, l_{mm}$, and $\{\gamma_{m1}(l), \ldots, \gamma_{mm}(l)\}$ generate the algebra of $S_{m}$--invariant polynomials in the variables $l_{m1},\ldots l_{mm}$ (see \cite{zh:cg}).
\end{remark}

\begin{example}\label{Example: explicit gamma's n leq 4}
Tthe polynomials $\gamma_{mk}(l)$ for $m\leq 4$ are listed below.
\begin{align*}
\gamma_{11}(l)=&l_{11}.\\
\gamma_{21}(l)=&(l_{21}+l_{22})+1.\\
\gamma_{22}(l)=&(l_{21}^{2}+l_{22}^{2})+(l_{21}+l_{22}).\\
\gamma_{31}(l)=&(l_{31}+l_{32}+l_{33})+3.\\
\gamma_{32}(l)=&(l_{31}^{2}+l_{32}^{2}+l_{33}^{2})+2(l_{31}+l_{32}+l_{33})+1.\\
\gamma_{33}(l)=&(l_{31}^{3}+l_{32}^{3}+l_{33}^{3})+4(l_{31}^{2}+l_{32}^{2}+l_{33}^{2})-(l_{31}l_{32}+l_{31}l_{33}+l_{32}l_{33})-6\\
& +(l_{31}+l_{32}+l_{33}).\\
\gamma_{41}(l)=&(l_1 + l_2 + l_3 + l_4) + 6.\\
\gamma_{42}(l)=&(l_1^2 + l_2^2 + l_3^2 +l_4^2)+ 3(l_1+ l_2+ l_3+l_4) + 4.\\
\gamma_{43}(l)=& (l_1^3+ l_2^3 + l_3^3 + l_4^3)  -( l_1 l_2 +l_1 l_3 + l_1 l_4 + l_2 l_3 + l_2 l_4+l_3 l_4 ) +\\
&+6 (l_1^2+l_2^2 + l_3^2 + l_4^2) + 3(l_1 + l_2+l_3+l_4)   - 19.\\
\gamma_{44}(l)=& (l_{1}^4+ l_2^4+ l_3^4+ l_4^4) + 9(l_{1}^3+ l_2^3 + l_3^3+ l_4^3)+ 21 (l_1^2+ l_2^2+ l_3^2+ l_4^2)-\\
&( l_1^2 l_2 + l_1^2 l_3 + l_1^2 l_4  + l_1 l_2^2  + l_1 l_3^2 + l_1 l_4^2 + l_2^2 l_3 + l_2^2 l_4+ l_2 l_3^2  + l_2 l_4^2 + l_3^2 l_4+ l_3 l_4^2)\\
&- 10 (l_1 l_2+ l_1 l_4+ l_2 l_3+ l_3 l_4+ l_1 l_3+ l_2 l_4) - 19(l_1+l_2+ l_3 + l_4) - 120.
\end{align*}
\end{example}

\begin{definition}\label{Definition: GT-modules} A finitely generated $U$-module
$M$ is called a \emph{Gelfand-Tsetlin module (relative to
$\Ga$)} provided that the restriction $M|_{\Ga}$ is a direct sum
of $\Ga$-modules:
\begin{equation}\label{Equation: GT-module}
M|_{\Ga}=\bigoplus_{\sm\in\Sp\Ga}M(\sm),
\end{equation}
where $$M(\sm)=\{v\in M\ |\  \sm^{k}v=0 \text{ for some }k\geq 0\}.$$
\end{definition}

\begin{definition}
An algebra homomorphism $\chi:\Gamma\rightarrow \mathbb{C}$ will be called \emph{Gelfand-Tsetlin character}.
\end{definition}

\begin{remark}
For each $\mathbf m\in\Sp\Ga$ we have associated a character $\chi_{\mathbf m}:\Ga\rightarrow\Ga/\mathbf m\sim\mathbb{C}$. In the same way, for each non-zero character  $\chi:\Ga\rightarrow\mathbb{C}$,  $Ker(\chi)$ is a maximal ideal of $\Ga$.  So, we have a natural identification between  characters of $\Ga$ and elements of $\Sp\Ga$.
So, using Gelfand-Tsetlin characters, a Gelfand-Tsetlin module (with respect to $\Gamma$) $M$ can be decomposed as $M=\bigoplus_{\chi\in\Gamma^{*}}M(\chi)$, where
$$M(\chi)=\{v\in M\ |\ \text{ for each }\gamma\in\Gamma\text{ , }\exists k\in\mathbb{Z}_{\geq 0} \text{ such that } (\gamma-\chi(\gamma))^{k}v=0 \}.$$
\end{remark}
\begin{definition}
Given a Gelfand-Tsetlin module $M$, the {\emph Gelfand-Tsetlin support} of $M$ is the set
$$\Supp_{GT}(M):=\{\chi\in\Ga^{*}\ |\ M(\chi)\neq 0\}.$$
\end{definition}

\begin{definition}\label{Definition: weight modules} A finitely generated $U$-module
$M$ is called a \emph{weight module (relative to
$\mathfrak h$)} provided that the restriction $M|_{\mathfrak h}$ is a direct sum
of simple $\mathfrak h$-modules:
\begin{equation}\label{Equation: GT-module}
M =\bigoplus_{\lambda\in {\mathfrak h}^*}M_{\lambda},
\end{equation}
where $$M_{\lambda} :=\{v\in M\ |\  hv=\lambda(h)v \text{ for all } h \in \mathfrak{h}\}.$$
The \emph{weight support} (or simply, the \emph{support}) of $M$ is 
$$\Supp (M):=\{\lambda \in \mathfrak{h}^{*}\ |\ M_{\lambda}\neq 0\}.$$
\end{definition}

\begin{remark}
Any simple Gelfand-Tsetlin module $M$ over $\gl(n)$ is a weight module with respect to the standard Cartan subalgebra $\mathfrak h$ spanned by $E_{ii}$, $i=1,\ldots, n$, see Remark \ref{Remark: Cartan inside Gamma}.  Moreover, $\Ga$ is diagonalizable on any finite-dimensional simple module.  On the other hand, a simple weight module $M$ need not to be Gelfand-Tsetlin, unless the $\mathfrak h$-weight multiplicities of $M$ are finite. The latter is true since in this case $\Gamma$ has a common eigenvector in every non-zero weight space. In particular, every highest weight module or, more general, every module from the category $\mathcal O$ is Gelfand-Tsetlin. 
\end{remark}

The definition of a Gelfand-Tsetlin module depends on the choice of the Gelfand-Tsetlin subalgebra $\Gamma$. One can easily define a family  of
Gelfand-Tsetlin subalgebras of $\gl(n)$ as follows. Let $\pi=\{\beta_1, \ldots, \beta_{n-1}\}$ be a base of the root system
of $\gl(n)$, where, $\beta_k = \epsilon_{i_k} -  \epsilon_{i_{k+1}}$, $k=1,...,n-1$.
Let  $\gl_{k}$ be the subalgebra of  $\gl(n)$ spanned by $E_{ij}$, $i,j \in \{ i_1,...,i_k\}$. In particular,  $\gl_{k} \simeq \gl (k)$ and $\beta_1, \ldots, \beta_k$ are the simple roots of $\gl_{k}$. Then we have a
chain of embeddings
$$\gl_1\subset \ldots \subset \gl_n.$$ 
Let $Z_i$ be the center of
$U(\gl_i)$ and $\Gamma(\pi)$ is the subalgebra generated by $Z_i$,
$i=1, \ldots, n$. We will call $\Gamma(\pi)$ a \emph{Gelfand-Tsetlin
subalgebra associated with $\pi$}.

Each  subalgebra $\Gamma(\pi)$ gives rise to a category of
Gelfand-Tsetlin modules which we denote by $\mathcal{GT}_{\pi}(n)$.  Let  $\pi$
and $\pi'$ be different bases of the root system. Then $\pi$
and $\pi'$ are
conjugate under the action of the Weyl group of $\gl(n)$. Hence  $\Gamma(\pi)$ and $\Gamma(\pi')$ are also conjugate which leads to an equivalence of the categories $\mathcal{GT}_{\pi}(n)$ and $\mathcal{GT}_{\pi'}(n)$.

\begin{example} Two bases $\pi$ and $\pi'$  may define the same Gelfand-Tsetlin subalgebra. Indeed, take for example the bases $\pi = \{ \epsilon_1 - \epsilon_2, \epsilon_2 - \epsilon_3\}$ and  $\pi' = \{ \epsilon_2 - \epsilon_1, \epsilon_1 - \epsilon_3\}$ of root system of $\gl(3)$. Then  $\Gamma(\pi) =  \Gamma(\pi')$. One easily checks that  $\gl(3)$ has three distinct Gelfand-Tsetlin
subalgebras and they are parameterized of the $\gl_2$-part of the chain.
\end{example}

\subsection{Gelfand-Tsetlin modules of $\mathfrak{sl}(n)$}\label{subsection: Gelfand-Tsetlin modules for sl(n)}

Let $\Ga$ be a
Gelfand-Tsetlin subalgebra of $\gl(n)$. Consider the natural projection
$\tau: \gl(n)\rightarrow \mathfrak{sl}(n)$, $\tau(A) = A - \frac{{\rm tr}(A) }{n}I_n$, which extends to an epimorphism
$\bar{\tau}: U(\gl(n))\rightarrow U(\mathfrak{sl}(n))$. Then the image $\bar{\tau}(\Ga)$
of $\Ga$ is called the (standard) Gelfand-Tsetlin subalgebra of $\mathfrak{sl}(n)$. It is a maximal commutative subalgebra of $U(\mathfrak{sl}(n))$ isomorphic to a polynomial ring in
$\frac{n(n+1)}{2}-1$ generators. With a small abuse of notation, by $\mathcal{GT}(n)$ we denote the category of all Gelfand-Tsetlin $\mathfrak{sl}(n)$-modules relative to $\Ga$.

\section{Families of Gelfand-Tsetlin modules for $\mathfrak{gl}(n)$}\label{Section: Families of Gelfand-Tsetlin modules for gl(n)}

\subsection{Gelfand-Tsetlin tableaux}\label{subsection:Gelfand-Tsetlin tableaux}

The simple finite dimensional modules are the first natural examples of Gelfand-Tsetlin modules. In this case, an eigenbasis for  the action of the generators of $\Gamma$ (\ref{Equation: c_mk}) is given by the so-called Gelfand-Tsetlin tableaux, following the original work of Gelfand and Tsetlin. In particular for every simple finite dimensional module $M$,  $\dim(M(\chi)) = 1$ whenever $\chi \in \Supp_{GT}(M): $. In order to describe  the  Gelfand Tsetlin tableaux,  we first fix some notation.

\begin{definition}\label{Definition: GT tableau} Fix a vector  
$v=(v_{ij})_{i\leq j}^n\in \mathbb{C}^{\frac{n(n+1)}{2}}$.
\begin{itemize}
\item[(i)] By $T(v)$ we will denote the following array with complex entries $\{v_{ij}\}$ 
\begin{center}

\Stone{\mbox{ \scriptsize {$v_{n1}$}}}\Stone{\mbox{ \scriptsize {$v_{n2}$}}}\hspace{1cm} $\cdots$ \hspace{1cm} \Stone{\mbox{ \scriptsize {$v_{n,n-1}$}}}\Stone{\mbox{ \scriptsize {$v_{nn}$}}}\\[0.2pt]
\Stone{\mbox{ \scriptsize {$v_{n-1,1}$}}}\hspace{1.5cm} $\cdots$ \hspace{1.5cm} \Stone{\mbox{ \tiny {$v_{n-1,n-1}$}}}\\[0.3cm]
\hspace{0.2cm}$\cdots$ \hspace{0.8cm} $\cdots$ \hspace{0.8cm} $\cdots$\\[0.3cm]
\Stone{\mbox{ \scriptsize {$v_{21}$}}}\Stone{\mbox{ \scriptsize {$v_{22}$}}}\\[0.2pt]
\Stone{\mbox{ \scriptsize {$v_{11}$}}}\\
\medskip
\end{center}
Such an array will be called a \emph{Gelfand-Tsetlin tableau} of height $n$. 
\item[(ii)] Throughout the paper, for any ring $R$, $T_n(R)$ will stand for the space of  the   Gelfand-Tsetlin
tableaux of height $n$ with entries in $R$. We will identify  $ T_{n}(\mathbb{C})$ with the set  $\mathbb{C}^{\frac{n(n+1)}{2}}$ in the following way:
 to  $$
v=(v_{n1},...,v_{nn}|v_{n-1,1},...,v_{n-1,n-1}| \cdots|v_{21}, v_{22}|v_{11})\in  \mathbb{C}^{\frac{n(n+1)}{2}}
$$ 
 we associate a tableau   $T(v)\in T_{n}(\mathbb{C})$ as above.
\end{itemize}
\end{definition}

\begin{remark}\label{Remark: correspondence between characters and tableaux}
 There is a natural correspondence between the set $\Gamma^*$ of characters $\chi : \Gamma \to \mathbb{C}$ and the set of Gelfand-Tsetlin tableaux of  height  $n$. In fact, to obtain a Gelfand-Tsetlin tableau $T(l)$ from a character $\chi$ we find a solution $(l_{ij})$ of the system of equations
$$\{\gamma_{mk}(l) =\ \chi(c_{mk})\}_{1\leq k\leq m\leq n}$$
Conversely, for every Gelfand-Tsetlin tableau $T(v)$ with entries $\{ v_{ij}\;  | \; 1\leq j \leq i \leq n\}$, we associate $\chi_{v} \in \Gamma^*$ by defining  $\chi_{v}(c_{mk})=\gamma_{mk}(v)$. It is clear that each tableau defines such a character uniquely. On the other hand, a
tableau is defined by a character up to a permutation of the rows, i.e. an element of $S_n \times S_{n-1} \times \cdots \times S_1$.
\end{remark}


\subsection{Gelfand-Tsetlin formulas for finite-dimensional modules}\label{subsection: Gelfand-Tsetlin formulas for finite-dimensional modules}

In this subsection we recall the classical result of I. Gelfand and M. Tsetlin, \cite{GT}.

\begin{definition}
A Gelfand-Tsetlin tableau of height $n$ is called \emph{standard} if
$v_{ki}-v_{k-1,i}\in\mathbb{Z}_{\geq 0}$ and $v_{k-1,i}-v_{k,i+1}\in\mathbb{Z}_{> 0}$ for all $1\leq i\leq k\leq n-1$.
\end{definition}
Note that, for the sake of convenience, the second condition in the definition above is slightly different from the original condition in \cite{GT}.

\begin{theorem}[Gelfand-Tsetlin, \cite{GT}]\label{Theorem: GT theorem}
Let $L(\lambda)$ be the simple finite dimensional module over $\mathfrak{gl}(n)$ of highest weight $\lambda=(\lambda_{1},\ldots,\lambda_{n})$. Then there exist a basis of $L(\lambda)$ parameterized by the set of all standard tableaux $T(v) = T(v_{ij})$ with fixed top row $v_{nj}=\lambda_j-j+1$, $j=1,\ldots,n$. Moreover,  the action of the generators of $\mathfrak{gl}(n)$ on $L(\lambda)$ is given by the \emph{Gelfand-Tsetlin formulas}:

$$E_{k,k+1}(T(v))=-\sum_{i=1}^{k}\left(\frac{\prod_{j=1}^{k+1}(v_{ki}-v_{k+1,j})}{\prod_{j\neq i}^{k}(v_{ki}-v_{kj})}\right)T(v+\delta^{ki}),$$

$$E_{k+1,k}(T(v))=\sum_{i=1}^{k}\left(\frac{\prod_{j=1}^{k-1}(v_{ki}-v_{k-1,j})}{\prod_{j\neq i}^{k}(v_{ki}-v_{kj})}\right)T(v-\delta^{ki}),$$

$$E_{kk}(T(v))=\left(k-1+\sum_{i=1}^{k}v_{ki}-\sum_{i=1}^{k-1}v_{k-1,i}\right)T(v),$$
where $\delta^{ij} \in T_{n}({\mathbb Z})$ is defined by  $(\delta^{ij})_{ij}=1$ and all other $(\delta^{ij})_{k\ell}$ are zero. If the new  tableau $T(v\pm\delta^{ki})$ is not standard, then the corresponding summand of $E_{k,k+1}(T(v))$ or $E_{k+1,k}(T(v))$ is zero by definition. 
\end{theorem}
We call the formulas above  the  \emph{Gelfand-Tsetlin formulas} for $\mathfrak{gl}(n)$.

 Set $e_{kk}(v):=\left(k-1+\sum\limits_{i=1}^{k}v_{ki}-\sum\limits_{i=1}^{k-1}v_{k-1,i}\right)$. We note that  $E_{kk}(T(v))=e_{kk}(v)T(v)$ is well defined for any Gelfand-Tsetlin tableau $T(v)$. .
 
 \begin{definition}\label{Definition: weight of a tableau} Let $T(v)$ be a Gelfand-Tsetlin tableau. Then we call $(e_{11}(v),\ldots,e_{nn}(v))$  (respectively, $(e_{11}(v)-e_{22}(v),\ldots,e_{n-1,n-1}(v)-e_{nn}(v))$) the \emph{ $\mathfrak{gl}(n)$-weight} (respectively,  the \emph{ $\mathfrak{sl}(n)$-weight}) of the tableau $T(v)$.
\end{definition}

The formulas for the  action of the generators of $\gl (n)$ in the theorem above imply that the standard tableaux $T(v)$ form an eigenbasis for the action of the standard Cartan subalgebra $\mathfrak h$. The following result shows that such basis is an eigenbasis for the Gelfand-Tsetlin subalgebra. 

\begin{theorem}[Zhelobenko, \cite{zh:cg}]\label{Theorem: action of Gamma f.d.}Let $L(\lambda)$ be the simple finite dimensional module over $\mathfrak{gl}(n)$ of highest weight $\lambda=(\lambda_{1},\ldots,\lambda_{n})$, with basis as described in Theorem \ref{Theorem: GT theorem}. The action of the generators $c_{rs}$ of $\Gamma$ (see (\ref{Equation: c_mk})) is given by 
\begin{equation}\label{Equation: action of Gamma on FD modules}
c_{rs}(T(v))=\gamma_{rs}(v)T(v),
\end{equation}
where $\gamma_{rs}(v)$ are the symmetric polynomials defined in (\ref{Equation: gamma mk}).
\end{theorem}

As a direct consequence of Theorem \ref{Theorem: GT theorem} and Theorem \ref{Theorem: action of Gamma f.d.}, any simple finite dimensional $\mathfrak{gl}(n)$-module is a Gelfand-Tsetlin module with one-dimensional Gelfand-Tsetlin subspaces.
\noindent

\begin{remark}
Whenever we refer to finite dimensional $\mathfrak{sl}(n)$-modules we will use the same vector space and the Gelfand-Tselin formulas for generators $E_{r,r+1}$ or $E_{r+1,r}$, for the action of the generators of a Cartan subalgebra $\{h_{1},\ldots, h_{n-1}\}$ we define $h_{i}(T(v)):=(E_{ii}-E_{i+1,i+1})(T(v))$. We also fix the action of the central element $E_{11}+\ldots+E_{nn}$ as zero.
\end{remark}

\begin{example}\label{Example: FD module sl(3)}
Let us to denote by $M$ the simple highest weight $\mathfrak{gl}(3)$-module with highest weight $(1,0,-1)$. $M$ is a finite dimensional module of dimension $8$. The tableaux realization guaranteed by Theorem \ref{Theorem: GT theorem} consist of a vector space spanned by the set of all standard tableaux of height $3$ with top row $(1,-1,-3)$.\\

\begin{center}
 \hspace{1.6cm}\Stone{$1$}\Stone{$-1$}\Stone{$-3$}\hspace{1cm}\Stone{$1$}\Stone{$-1$}\Stone{$-3$}\\[0.2pt]
 $T_1$= \hspace{0.7cm}\Stone{$1$}\Stone{$-1$} \hspace{1.4cm} $T_2$= \Stone{$1$}\Stone{$-1$}\\[0.2pt]
 \hspace{1.3cm}\Stone{$1$}\hspace{3.8cm}\Stone{$0$}\\
\end{center}
\bigskip
\begin{center}
 \hspace{1.6cm}\Stone{$1$}\Stone{$-1$}\Stone{$-3$}\hspace{1cm}\Stone{$1$}\Stone{$-1$}\Stone{$-3$}\\[0.2pt]
 $T_3$= \hspace{0.7cm}\Stone{$1$}\Stone{$-2$} \hspace{1.4cm} $T_4$= \Stone{$1$}\Stone{$-2$}\\[0.2pt]
 \hspace{1.3cm}\Stone{$1$}\hspace{3.8cm}\Stone{$0$}\\
\end{center}
\bigskip
\begin{center}
 \hspace{1.6cm}\Stone{$1$}\Stone{$-1$}\Stone{$-3$}\hspace{1cm}\Stone{$1$}\Stone{$-1$}\Stone{$-3$}\\[0.2pt]
 $T_5$= \hspace{0.7cm}\Stone{$0$}\Stone{$-2$} \hspace{1.4cm} $T_6$= \Stone{$0$}\Stone{$-2$}\\[0.2pt]
 \hspace{1.3cm}\Stone{$-1$}\hspace{3.8cm}\Stone{$0$}\\
\end{center}
\bigskip
\begin{center}
 \hspace{1.6cm}\Stone{$1$}\Stone{$-1$}\Stone{$-3$}\hspace{1cm}\Stone{$1$}\Stone{$-1$}\Stone{$-3$}\\[0.2pt]
 $T_7$= \hspace{0.7cm}\Stone{$1$}\Stone{$-2$} \hspace{1.4cm} $T_8$= \Stone{$0$}\Stone{$-1$}\\[0.2pt]
 \hspace{1.3cm}\Stone{$-1$}\hspace{3.8cm}\Stone{$0$}\\
\end{center}
\bigskip

By Theorem \ref{Theorem: GT theorem}, the module $M$ is  isomorphic to $\Span_{\mathbb{C}}\{T_{i}\ |\ i=1,\ldots,8\}$ endowed with the action of $\mathfrak{gl}(3)$ given by the Gelfand-Tsetlin formulas.\\

Since the action of $E_{11}+E_{22}+E_{33}$ is fixed to  be trivial and $\mathfrak{h}=\Span_{\mathbb{C}}\{h_{1}=E_{11}-E_{22},\ h_{2}=E_{22}-E_{33}\}$, $M$ becomes an $\mathfrak{sl}(3)$-module with weight support
$$\Supp(M)=\{(1,1),(-1,2),(2,-1),(0,0),(-2,1),(1,-2),(-1,-1)\}.$$
Then as an $\mathfrak{sl}(3)$-module, $M$ is isomorphic to $L(1,1)$ (the simple finite dimensional $\mathfrak{sl}(3)$-module of highest weight $(1,1)$). The following picture shows the weights lattice of the $\mathfrak{sl}(3)$-module $M$. Note that $M_{(0,0)}$ is $2$-dimensional with basis $\{T_{4},\ T_{8}\}$.\\

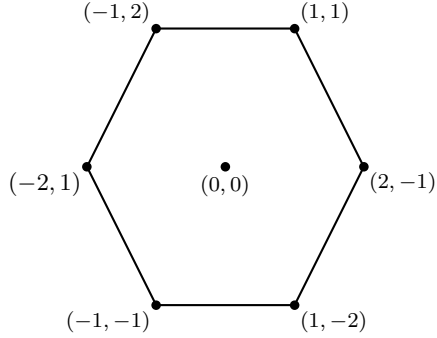
\begin{center}
\psset{unit=0.23cm}
\begin{pspicture}(9,-6)(23,15)

\psline(21,-2)(25,6)
\psline(21,-2)(13,-2)
\psline(13,-2)(9,6)
\psline(9,6)(13,14)
\psline(13,14)(21,14)
\psline(21,14)(25,6)

\psdots(17,6)
  \psdots(9,6)(25,6)
  \psdots(13,-2)(21,-2)
    \psdots(13,14)(21,14)
  \uput{3pt}[dl](9,6){{\small $(-2,1)$}}
    \uput{3pt}[d](17,6){{\footnotesize $(0,0)$}}
      \uput{3pt}[dr](25,6){{\footnotesize $(2,-1)$}}
        \uput{3pt}[dl](13,-2){{\footnotesize $(-1,-1)$}}
          \uput{3pt}[dr](21,-2){{\footnotesize $(1,-2)$}}
            \uput{3pt}[ul](13,14){{\footnotesize $(-1,2)$}}
              \uput{3pt}[ur](21,14){{\footnotesize $(1,1)$}}
             
  \end{pspicture}
\end{center}
In particular, the basis elements of $M_{(0,0)}$ can not be distinguish by the action of the Cartan subalgebra. However, using $\Gamma$ the module decomposes as a direct sum of $1$-dimensional $\Gamma$-submodules.
\end{example}

The following theorem will give us information about the dimension of Gelfand-Tsetlin subspaces for simple Gelfand-Tsetlin modules and the possible number of non-isomorphic Gelfand-Tsetlin modules with a given Gelfand-Tsetlin character in its support.

\begin{theorem}[\cite{FO2}, Theorem 6.1; \cite{O}]\label{Theorem: finiteness-for-gl-n}
Let $U=U(\gl(n))$,
$\Ga\subset U$ the Gelfand-Tsetlin subalgebra, $\sm\in\Sp\Ga$. Set $Q_n=1!2!\ldots (n-1)!$.  
\begin{itemize}
\item[(i)] For a Gelfand-Tsetlin module
$M$, such that $M(\sm)\ne 0$ and $M$ is generated by some $x\in
M(\sm)$ (in particular for an simple module), one has $$\dim_{\mathbb{C}}
M(\sm)\leq Q_n.$$
\item[(ii)] The number of  isomorphism classes of simple Gelfand-Tsetlin modules $N$ such that $N(\sm)\ne 0$ is always nonzero and does not exceed  $Q_n$.
\end{itemize}
\end{theorem}

The theorem above shows that elements of $\Sp\Ga$ classify the simple Gelfand-Tsetlin $\gl(n)$-modules (and, hence,  $\mathfrak{sl}(n)$-modules) up to some finiteness and up to an isomorphism of Gelfand-Tsetlin modules which contains two different Gelfand-Tsetlin characters.

In \cite{Mazorchuk}, Gelfand-Tsetlin modules with tableaux realization and action given by the Gelfand-Tsetlin formulas are studied, but such modules $V$ satisfy $\dim(V(\chi))\leq 1$ for all $\chi\in \Gamma^{*}$. In what follows we will consider a more general definition of tableaux realization, which will allow to consider certain classes of modules with $\dim(V(\chi))$ greater than $1$.

For any Gelfand-Tsetlin tableau $T(v)\in T_n(\mathbb C)$ we  consider the set
\begin{equation}\label{Equation: B(T(v))} \mathcal{B}(T(v)):=\{T(v+z)\; | \; z\in T_{n}(\mathbb Z), z_{nk}=0, \, 1\leq k\leq  n \}.
\end{equation}

 If an indecomposable Gelfand-Tsetlin module $V$ has a tableaux realization and $T(v)$ is one of the basis tableaux then it has a basis which is a subset of $\mathcal{B}(T(v))$. On the other hand, we might have a module with a basis  consisting of a subset of tableau from $\mathcal{B}(T(v))$ but without  a tableaux realization. This may happen, for example, when $V$ has a Gelfand-Tsetlin character of multiplicity more than $1$.  For this reason we extend the notion of modules with a tableaux realization.
\begin{definition}\label{Definition: Generalized tableaux realization}
We say that a Gelfand-Tsetlin module $M$ admits a \emph{generalized tableaux realization} with respect to a Gelfand-Tsetlin subalgebra $\Gamma$ if $M$ has a basis  $\mathcal{B}_M$  labelled by a subset  of $\mathcal{B}(T(v))$ for some tableau $T(v)$, such that every $m_{T(w)} \in \mathcal{B}_M$, $T(w) \in \mathcal{B}(T(v))$, is a generalized eigenvector of $c_{rs}$ of eigenvalue $\gamma_{rs} (w)$, for all $r,s$. We will denote by $\mathcal{GT}_{T(v)}(n)$ the full subcategory of the category of Gelfand-Tsetlin modules $\mathcal{GT}(n)$ which consists of  modules with
a generalized tableaux realization with respect to $\Gamma$ whose basis contains $T(v)$.
\end{definition}

  The subcategory  $\mathcal{GT}_{T(v)}(n)$ is closed under the operations of taking submodules and quotients. Moreover, as our main result will imply, simple modules of the categories  $\mathcal{GT}_{T(v)}(3)$ for all $T(v)$ exhaust all simple Gelfand-Tsetlin modules for $\mathfrak{sl}(3)$. 
 
 \medskip
\noindent  {\bf Conjecture}: Simple modules of the categories  $\mathcal{GT}_{T(v)}(n)$ for all $T(v)$ exhaust all simple Gelfand-Tsetlin modules for $\mathfrak{sl}(n)$.

\begin{remark}
There are modules in $\mathcal{GT}(n)$ that do not belong to $\mathcal{GT}_{T(v)}(n)$ for any $T(v)$. For example,  consider
$n=2$ and a simple weight module $V(\lambda, \gamma)$ with a weight $\lambda\in \mathbb C$ in its weight support and on which the Casimir element $c_{22}$ acts as a multiplication by $\gamma\in \mathbb C$, where  $\gamma\neq (\lambda + k)^2+2(\lambda + 2k)$ for any integer $k$.
 Then $V(\lambda, \gamma)$ has a non-split self-extension which remains a weight module but on which $c_{22}$ does not act semisimply. This self-extension is an indecomposable Gelfand-Tsetlin module that does not admit a generalized tableaux realization. 
\end{remark}

\begin{definition}\label{Definition: blocks}
We will call the subcategory $\mathcal{GT}_{T(v)}(n)$  \emph{ the block of $\mathcal{GT}(n)$ generated by $T(v)$}. 
\end{definition}

From now on, whenever $n$ is clear from the context, we will write $\mathcal{GT}_{T(v)}$ instead of $\mathcal{GT}_{T(v)}(n)$.


\subsection{Generic modules}\label{subsection: Generic modules}

Observing that the coefficients in the Gelfand-Tsetlin formulas in Theorem \ref{Theorem: GT theorem} are rational functions on the entries of the tableaux, Y. Drozd, V. Futorny and S. Ovsienko \cite{DFO3} extended the Gelfand-Tsetlin construction to more general modules.  In the case when all denominators are nonzero for all possible integral shifts, one can use the same formulas and  define a new class of infinite dimensional $\gl(n)$-modules:  {\it generic}  Gelfand-Tsetlin modules (cf. \cite{DFO3}, section 2.3.).

\begin{definition}\label{Definition: generic tableau}
A Gelfand-Tsetlin tableau $T(v)$ (equivalently, $v \in  T_{n}(\mathbb{C})$) is called \emph{generic} if
$v_{ki}-v_{kj}\notin\mathbb{Z}$ for all $1\leq i\neq j \leq k \leq n-1$. A  Gelfand-Tsetlin character  $\chi_v$ associated to a generic tableau $T(v)$ (see Remark \ref{Remark: correspondence between characters and tableaux}) will be called a \emph{generic} Gelfand-Tsetlin character.
\end{definition}
Recall  that ${\mathcal B}(T(v))=\{T(v+z)\ |\ z\in T_{n-1}(\mathbb{Z})\}$ for any Gelfand-Tableau $T(v)$.

\begin{theorem}[\cite{DFO3}, Section 2.3]\label{Theorem: Generic GT modules}
Let $T(v)$ be a generic  Gelfand-Tsetlin tableau of height $n$. 
\begin{itemize}
\item[(i)] The vector space $V(T(v)) = \Span_{\mathbb{C}} {\mathcal B}(T(v))$ has  a structure of a $\mathfrak{gl}(n)$-module with action of the generators of $\mathfrak{gl}(n)$ given by the Gelfand-Tsetlin formulas.
\item[(ii)] The action of the generators of $\Gamma$ on the basis elements of ${\mathcal B}(T(v))$ is given by (\ref{Equation: action of Gamma on FD modules}). 
\item[(iii)] The $\mathfrak{gl}(n)$-module $V(T(v))$ is a Gelfand-Tsetlin module with Gelfand-Tsetlin multiplicities equal to $1$. 
\end{itemize}
\end{theorem}

The module  $V(T(v))$  constructed in  Theorem \ref{Theorem: Generic GT modules} will be extensively used in future and will be referred as \emph{the generic Gelfand-Tsetlin  module associated to $T(v)$}.  In general $V(T(v))$ need not to be simple.  Because $\Gamma$ has simple spectrum on $V(T(v))$ for $T(w)$ in $\mathcal{B}(T(v))$ we may define the \emph{simple $U$-module in $V(T(v))$ containing $T(w)$} to be the simple subquotient of $V(T(v))$ containing $T(w)$.

\begin{remark}\label{Remark: Gamma separates generic tableaux}
By Theorem \ref{Theorem: Generic GT modules}(iii), given two different tableaux $T(w)$ and $T(w')$ in $\mathcal{B}(T(v))$, there exists an element of $\Gamma$ that has different eigenvalues for $T(w)$ and $T(w')$. Whenever we say that {\it $\Gamma$ ``separates" tableaux of $V(T(v))$} we will refer to this property. 
\end{remark}

\subsubsection{Gelfand-Tsetlin formulas in terms of permutations}\label{subsubsection: GT formulas in terms of permutations}
\medskip

In this subsection we will rewrite the Gelfand-Tsetlin formulas in terms of permutations. These formulas will be very useful when verifying certain identities for the action of $\mathfrak g$ on   $V(T(v))$.

Let $\widetilde{S}_m$ denotes the subset of $S_m$ consisting of the transpositions $(1,i)$, $i=1,...,m$. For $\ell < m $, set $\Phi_{\ell m} =\widetilde{S}_{m-1}\times\cdots\times\widetilde{S}_{\ell} $. For $\ell > m$ we set $\Phi_{\ell m} = \Phi_{m \ell}$. Finally we define, $\Phi_{\ell \ell} = \{ \mbox{Id}\}$.  Every $\sigma$ in $\Phi_{\ell m }$ will be written as an $|\ell- m|$-tuple  of transpositions and by $\sigma[t]$ we will denote the $t$-th component of the tuple. 

\begin{remark}
Recall that in order to have well defined action of $\sigma\in\Phi_{\ell m }$ on $ T_{n-1}(\mathbb{C})$, for $w\in T_{n-1}(\mathbb{C})$ and $\sigma\in\Phi_{\ell m }$ on $w$ we set  
$$
\sigma(w):=(w_{n-1,\sigma^{-1}[n-1](1)},\ldots,w_{n-1,\sigma^{-1}[n-1](1)}|\ldots|w_{1,\sigma^{-1}[1](1)}).$$
\end{remark}

\begin{definition} \label{Definition: varepsilon_rs}
Let $1 \leq r < s \leq n$. Define
$$
\varepsilon_{rs}:=\delta^{r,1}+\delta^{r+1,1}+\ldots+\delta^{s-1,1} \in T_{n}(\mathbb{Z}).
$$
Furthermore, define $\varepsilon_{rr}=0$ and  $\varepsilon_{sr}=- \varepsilon_{rs}$.
\end{definition}

\begin{definition}
For each generic vector $w$ and any $1\leq t\leq n-1$ define
\begin{align*}
e_{t}^{(+)}(w):=\frac{\prod_{j\neq 1}^{t+1}(w_{t1}-w_{t+1,j})}{\prod_{j\neq 1}^{t}(w_{t1}-w_{tj})} & \ \ \ \ ; \ \ \ \
e_{t+1}^{(-)}(w) :=\frac{\prod_{j\neq 1}^{t-1}(w_{t1}-w_{t-1,j})}{\prod_{j\neq 1}^{t}(w_{t1}-w_{tj})} \\
e_{k,k+1}(w):=- \frac{\prod_{j=1}^{k+1}(w_{k1}-w_{k+1,j})}{\prod_{j\neq 1}^{k}(w_{k1}-w_{kj})} & \ \ \ \ ; \ \ \ \ e_{k+1,k}(w):= \frac{\prod_{j=1}^{k-1}(w_{k1}-w_{k-1,j})}{\prod_{j\neq 1}^{k}(w_{k1}-w_{kj})}.
\end{align*}
\end{definition}

\begin{lemma}\label{Lemma: E_mk for m>k} For each $m>k$ the action of $E_{mk}$ is given by the expression:
{\footnotesize $$E_{mk}(T(v))=\sum_{\sigma\in \Phi_{mk }}e_{k+1,k}(\sigma(v))\left(\prod_{j=k+2}^{m}\left(e_{j}^{(-)}(\sigma(v))\right)\right)T\left(v+\sum_{i=k}^{m-1}\sigma(\varepsilon_{i+1,i})\right).$$}
\end{lemma}
\begin{proof}
The case $k=m+1$ follows from the Gelfand-Tsetlin formulas. The general case follows by induction on $m-k$ using the relation $$E_{m,k+1}(E_{k+1,k}T(v))-E_{k+1,k}(E_{m,k+1}T(v))=E_{m,k}T(v)$$ for any generic vector $v$. \end{proof}

\begin{lemma}\label{Lemma: E_mk for m<k} For each $r<s$ the action of $E_{rs}$ is given by the expression:
{\footnotesize $$E_{rs}(T(v))=\sum_{\sigma\in \Phi_{rs }}\left(\displaystyle\prod_{j=r}^{s-2}e_{j}^{(+)}(\sigma(v))\right)e_{s-1,s}(\sigma(v))T\left(v+\sum_{i=r}^{s-1}\sigma(\varepsilon_{i,i+1})\right).$$}
\end{lemma}
\begin{proof}
The case $s=r+1$ follows from the Gelfand-Tsetlin formulas. The general case follows by induction in $s-r$ using the relation $$E_{r,r+1}(E_{r+1,s}T(v))-E_{r+1,s}(E_{r,r+1}T(v))=E_{rs}T(v)$$ for any generic vector $v$.
\end{proof}

\begin{definition}\label{Definition: e_rs(v)}
For each generic vector $w\in T_{n}(\mathbb{C})$ and any $1\leq r, s\leq n$ we define
$$
e_{rs}(w):=
\begin{cases}
\left(\displaystyle\prod_{j=r}^{s-2}e_{j}^{(+)}(w)\right)e_{s-1,s}(w), & \text { if }\ \ \ r<s\\
e_{s+1,s}(w)\left(\displaystyle\prod_{j=s+2}^{r}e_{j-2}^{(-)}\right), &  \text { if }\ \ \ r>s\\
r-1+
\sum_{i=1}^{r}w_{ri}-\sum_{i=1}^{r-1}w_{r-1,i}, &  \text { if }\ \ \ r=s.
\end{cases}
$$
\end{definition}

\begin{proposition}\label{Proposition: Action of E_ij in terms of e_ij(v)} Let $v\in T_{n}(\mathbb{C})$ be any generic vector and $z\in T_{n-1}(\mathbb{Z})$. The Gelfand-Tsetlin formulas for the $U$-module $V(T(v))$ can be written as follows:
$$
E_{\ell m} (T(v+z))= \sum_{\sigma \in \Phi_{\ell m}} e_{\ell m} (\sigma (v+z)) T (v + z+\sigma (\varepsilon_{\ell m})) .$$
\end{proposition}
\begin{proof}

Follows from Lemmas \ref{Lemma: E_mk for m>k} and \ref{Lemma: E_mk for m<k} and the fact that  $\sum_{i=\ell}^{m-1}\sigma(\varepsilon_{i+1,i})=\sigma(\varepsilon_{m,\ell})$ when $m>\ell$ and $\sum_{i=m}^{\ell-1}\sigma(\varepsilon_{i,i+1})=\sigma(\varepsilon_{m,\ell})$ when $m<\ell$.
\end{proof}

\begin{example}\label{Example: GT formulas in terms of permutations for gl(3)}
Let us write explicitly the functions $e_{rs}(w)$ from Definition \ref{Definition: e_rs(v)} and the Gelfand-Tsetlin formulas in Proposition \ref{Proposition: Action of E_ij in terms of e_ij(v)} in the case of $\mathfrak{gl}(3)$. Let $v\in T_{3}(\mathbb{C})$ be a generic vector, $z\in T_{2}(\mathbb{Z})$ and $w=v+z$. Set also $\tau$ to be the permutation that interchanges the entries in positions $(2,1)$ and $(2,2)$. Considering
\begin{center}
\begin{tabular}{|>{$}l<{$}||>{$}l<{$}|}\hline
\varepsilon_{11}=(0,0,0)&  e_{11}(w)=w_{11}\\\hline
\varepsilon_{22}=(0,0,0)&  e_{22}(w)=w_{21}+w_{22}-w_{11}+1\\\hline
\varepsilon_{33}=(0,0,0)&  e_{33}(w)=w_{31}+w_{32}+w_{33}-w_{21}-w_{22}+2\\\hline
\varepsilon_{12}=(0,0,1)&e_{12}(w)=-(w_{11}-w_{21})(w_{11}-w_{22})\\ \hline
\varepsilon_{21}=(0,0,-1)& e_{21}(w)=1\\ \hline
\varepsilon_{23}=(1,0,0)& e_{23}(w)=-\frac{(w_{21}-w_{31})(w_{21}-w_{32})(w_{21}-w_{33})}{w_{21}-w_{22}}\\ \hline
\varepsilon_{32}=(-1,0,0)& e_{32}(w)=\frac{w_{21}-w_{11}}{w_{21}-w_{22}}\\ \hline
\varepsilon_{31}=(-1,0,-1)& e_{31}(w)=\frac{1}{w_{21}-w_{22}}\\ \hline
\varepsilon_{13}=(1,0,1)& e_{13}(w)=-\frac{(w_{21}-w_{31})(w_{21}-w_{32})(w_{21}-w_{33})(w_{11}-w_{22})}{w_{21}-w_{22}}\\ \hline
\end{tabular}
\end{center}

\medskip
The action of $\mathfrak{gl}(3)$ on any tableau is given by:
\begin{align*}
E_{11}(T(w)) &= e_{11}(w)T(w)&\\
E_{22}(T(w)) &= e_{22}(w)T(w)&\\
E_{33}(T(w)) &= e_{33}(w)T(w)&\\
E_{12}(T(w)) &= e_{12}(w)T(w+\varepsilon_{12})&\\
E_{21}(T(w)) &= e_{21}(w)T(w+\varepsilon_{21})&\\
E_{32}(T(w)) &= e_{32}(w)T(w+\varepsilon_{32})+e_{32}(\tau(w))T(w+\tau(\varepsilon_{32}))&\\
E_{23}(T(w))&=e_{23}(w)T(w+\varepsilon_{23})+e_{23}(\tau(w))T(w+\tau(\varepsilon_{23}))&\\
E_{13}(T(w))&=e_{13}(w)T(w+\varepsilon_{13}) +e_{13}(\tau(w))T(w+\tau(\varepsilon_{13}))&\\
E_{31}(T(w))&=e_{31}(w)T(w+\varepsilon_{31})+e_{31}(\tau(w))T(w+\tau(\varepsilon_{31})).&
\end{align*}
\end{example}

The explicit description of all simple generic modules for $\mathfrak{gl}(3)$ was obtained first in \cite{Ramirez}. The classification of simple generic modules $\mathfrak{gl}(n)$ was completed in \cite{FGR2}. Let us discuss briefly the main results in the last classification. 

\begin{definition}\label{Definition: Omega +} Let $T(v)$ be a fixed Gelfand-Tsetlin tableau. For any $T(w)\in \mathcal{B}(T(v))$, and for any $1<r\leq n$, $1\leq s\leq r$ and $1\leq u \leq r-1$ we define:
$$\Omega(T(w)):=\{(r,s,u)\ |\ w_{rs}-w_{r-1,u}\in \mathbb{Z}\}$$
$$\Omega^{+}(T(w)):=\{(r,s,u)\ |\ w_{rs}-w_{r-1,u}\in \mathbb{Z}_{\geq 0}\}$$
\end{definition}

A basis for the simple subquotients of $V(T(v))$ is provided in the following theorem.

\begin{theorem}[\cite{FGR2}, Theorems 6.8 and 6.14]\label{Theorem: basis for Irr generic modules}
Let $T(v)$ be a fixed generic Gelfand-Tsetlin tableau and $T(w)$ in $\mathcal{B}(T(v))$.
\begin{itemize}
\item[(i)]  The module $U\cdot T(w)$ has a basis of tableaux
$$\mathcal{N}(T(w))=\{T(w')\in\mathcal{B}(T(w))\ |\ \Omega^{+}(T(w))\subseteq\Omega^{+}(T(w'))\}.$$ 
\item[(ii)]  The simple module  containing $T(w)$ has a basis of tableaux
$$\mathcal{I}(T(w))=\{T(w')\in\mathcal{B}(T(w))\ |\ \Omega^{+}(T(w))=\Omega^{+}(T(w'))\}.$$
\end{itemize}

The action of $\mathfrak{gl}(n)$ on $T(w')\in \mathcal{N}(T(w))$ is given by the Gelfand-Tsetlin formulas. The action of $\gl(n)$ on $T(w') \in\mathcal{I}(T(w))$ is given by the Gelfand-Tsetlin formulas with the convention that all tableau $T(w'\pm\delta^{ki})$ for which $\Omega^{+}(T(w'\pm \delta^{ki}))\neq \Omega^{+}(T(w))$ are omitted in the sums for $E_{k,k+1}(T(w'))$ and $E_{k+1,k}(T(w'))$.
\end{theorem}

\begin{corollary}
Let $T(v)$ be a generic Gelfand-Tsetlin tableau. The module $V(T(v))$ is simple if and only if $\Omega(T(v))=\emptyset$.
\end{corollary}

\begin{example}\label{Example: explicit basis for irr gen module}
Consider $a,b,c\in\mathbb{C}$ such that $\{a-b,a-c,b-c\}\bigcap\mathbb{Z}=\emptyset$ and $v=(a,b,c | a,b+2 |a)$, then
\begin{center}
 \hspace{1.5cm} \Stone{a}\Stone{b}\Stone{c}\\[0.2pt]
 $T(v)$=\hspace{0.5cm} \Stone{a}\Stone{b+2}\hspace{1.2cm}\\[0.2pt]
 \hspace{1.3cm} \Stone{a}\\
\end{center}
then $\Omega(T(v))=\{(3,1,1),(3,2,2),(2,1,1)\}$, $\Omega^{+}(T(v))=\{(3,1,1), (2,1,1)\}$. So, by Theorem \ref{Theorem: basis for Irr generic modules}, the simple subquotient of $V(T(v))$ containing $T(v)$ has a basis
$$\mathcal{I}(T(v))=\{T(v+(m,n,k)))\ |\ m\leq 0, k\leq m, \text{ and } n> -2\}.$$
\end{example}

\begin{example}\label{Example: generic Verma module} (See also \cite{Mazorchuk}, Section 4.3) Let $a_{1},\ldots,a_{n}$ be complex numbers such that $a_{i}-a_{j}\notin \mathbb{Z}$ for any $i\neq j$. Denote by $T(v)$ the Gelfand-Tsetlin tableau of height $n$ with entries $v_{ij}$, such that $v_{rs}=a_{s}$ for $1\leq s\leq r\leq n$. The tableau $T(v)$ is a generic Gelfand-Tsetlin tableau and by Theorem \ref{Theorem: basis for Irr generic modules} a basis for an simple $\mathfrak{gl}(n)$-module containing $T(v)$ has a basis
\begin{align*}
\mathcal{I}(T(v))=&\{T(v+z)\ |\ z_{rs}-z_{r-1,s}\in\mathbb{Z}_{\geq 0} \text{ for any } r,s\}.
\end{align*}
Moreover, we can easily check that $\Span_{\mathbb{C}}\mathcal{I}(T(v))$ is a submodule of $V(T(v))$ isomorphic to the simple Verma module $M(a_{1},a_{2}+1,\ldots,a_{n}+n-1)$.
\end{example}

Using the description of simple subquotients of $V(T(v))$, we will also be able to describe the Loewy series decomposition for $V(T(v))$. We will use the convention that the first module in the list is the socle  of $V(T(v))$.
\begin{theorem}\label{Theorem: Loewy decomposition generic case}
Let $T(v)$ be a generic tableau and set $t:=|\Omega(T(v))|$. The Loewy series decomposition of the Gelfand-Tsetlin module $V(T(v))$ is given by 
$$D_{t},\ D_{t-1},\ \cdots,D_{0}$$
where, $D_{i}={\rm Span}_{\mathbb{C}}\{T(w)\in\mathcal{B}(T(v))\ |\ |\Omega^{+}(T(w))|=i\}$ and $0\leq i\leq t$.
If $D_{i}=\emptyset$ for some $1<i<t$ we omit this term in the Loewy decomposition.
\end{theorem}
\begin{proof}
Let us show first that $D_{t}$ is a simple submodule of $V(T(v))$. By Theorem \ref{Theorem: basis for Irr generic modules}(i), if $|\Omega^{+}(T(w))|=t$, the module generated by $T(w)$ is simple and hence, equal to $Span_{\mathbb{C}}\{T(w')\ |\ |\Omega^{+}(T(w'))|=t\}$. That is, $V(T(v))$ has a unique simple submodule $M$, namely $M=D_{t}$.\\
Set $M_{t+1}:=V(T(v))$ and define $M_{i}:=M_{i+1}/D_{i}$. Note that 
$$M_{i+1}={\rm Span}_{\mathbb{C}}\{T(w)\in\mathcal{B}(T(v))\ |\ |\Omega^{+}(T(w))|\leq i\}$$
So, by Theorem \ref{Theorem: basis for Irr generic modules}(ii) any element basis of $D_{i}$ is a basis element of a simple submodule of $M_{i+1}$ and, then $D_{i}$ is the sum of all simple submodules of $M_{i+1}$. 
\end{proof}

\begin{remark}\label{Remark: Loewy series for generic modules}
We will often apply Theorem \ref{Theorem: Loewy decomposition generic case} in the following way.
If $V_j$, $j \in J$, are all non-isomorphic simple subquotients  of $V(T(v))$, and $T(w_j) \in V_j$, then 
the modules Loewy series components of $V(T(v))$ are precisely:
$$D_{i}=D_{i1}\oplus\ldots\oplus D_{i,r_{i}},$$
where $\{D_{ij}\}_{j=1}^{r_{i}}$ is the set of all $V_j$ such that  $|\Omega^{+}(w_j)|=i$. 
\end{remark}

Although Theorem \ref{Theorem: basis for Irr generic modules} gives a nice relation between the category $\mathcal{GT}_{T(v)}$ (see Definition \ref{Definition: Generalized tableaux realization}) and $\Omega(T(v))$, it is not true that $\mathcal{GT}_{T(v)}$  is completely determined by $\Omega(T(v))$ - see the next example.
\begin{example}  Consider the  tableaux $T(v)$ and $T(v')$ such that $\Omega(T(v))=\Omega(T(v'))$ but $\mathcal{GT}_{T(v)}$ is not equivalent to $\mathcal{GT}_{T(v')}$. Take
\begin{center}
 \hspace{1.6cm}\Stone{$a$}\Stone{$a$}\Stone{$a$}\hspace{1cm}\Stone{$a$}\Stone{$a+1$}\Stone{$a+2$}\\[0.2pt]
 $T(v)$= \hspace{0.4cm}\Stone{$a$}\Stone{$y$} \hspace{0.7cm} $T(v')$= \hspace{0.2cm}  \Stone{$a$}\Stone{$y$}\\[0.2pt]
 \hspace{1.3cm}\Stone{$z$}\hspace{3.8cm}\Stone{$z$}\\
\end{center}
Then  $\Omega(T(v))=\Omega(T(v'))$. The Loewy series of $T(v)$ is $D_3,D_0$, however, the Loewy series of $T(v')$ is $D_3,D_2,D_1,D_0$.

\end{example}

\subsection{Singular Gelfand-Tsetlin modules}\label{subsection: Singular Gelfand-Tsetlin modules}
The construction of simple finite-dimensional modules and  of generic modules presented in Sections \ref{subsection: Gelfand-Tsetlin formulas for finite-dimensional modules} and \ref{subsection: Generic modules} have one common feature - an explicit basis parameterized by a set of Gelfand-Tsetlin tableaux. In the finite-dimensional case all the entries of the tableaux $T(v)$ in the basis satisfy  $v_{ki}-v_{kj}\in\mathbb{Z}$, while  in  the generic case they satisfy  $v_{ki}-v_{kj}\notin\mathbb{Z}$ for any $k\neq n$. We may consider the standard and generic tableaux as the two extreme cases of the singular Gelfand-Tsetlin tableaux where the latter are defined below. 

\begin{definition}\label{Definition: singular and 1-singular vector}
A vector $v\in T_{n}(\mathbb{C})$ will be called \emph{singular} if there exist $1\leq s<t\leq r\leq n-1$ such that $v_{rs}-v_{rt}\in\mathbb{Z}$. The vector $v$ will be called \emph{$1$-singular} if there exist $k,i,j$ with $1\leq i<j\leq k\leq n-1$ such that $v_{ki}-v_{kj}\in\mathbb{Z}$ and $v_{rs}-v_{rt}\notin\mathbb{Z}$ for all $(r,s,t) \neq (k,i,j)$, $r\neq n$. If $v$ is $1$-singular, the tableau $T(v)$ will be called \emph{$1$-singular tableau}. A Gelfand-Tsetlin character $\chi_v$  associated to a singular (respectively, $1$-singular) tableau $T(v)$ (see Remark \ref{Remark: correspondence between characters and tableaux}) will be called a \emph{singular} (respectively, \emph{$1$-singular}) character.
\end{definition}

\subsubsection{Construction of $1$-singular Gelfand-Tsetlin modules}\label{subsection: Construction of $1$-singular Gelfand-Tsetlin modules} 
In \cite{FGR3} an explicit construction of modules with a generalized tableaux realization (see Definition \ref{Definition: Generalized tableaux realization}) associated with any $1$-singular Gelfand-Tsetlin tableau was provided. In this section we provide the main details of this construction.

Set $v$ a vector of variables with $\frac{n(n+1)}{2}$ entries indexed by $(r,s)$ such that $1\leq s\leq r\leq n$. By ${\mathcal F}$ we will denote the space of rational functions on $v_{ij}$, $1\leq j \leq i \leq n$, with poles on the hyperplanes $v_{rs} - v_{rt} =0$. Note that $V(T(v))$ is defined for all generic $v$ and that $V(T(v)) = V(T(v'))$ whenever $v - \sigma(v') \in T_{n-1}({\mathbb Z})$ for some $\sigma\in G$.  Thus $V(T(v))$ is defined for elements $v$ in the (generic) complex torus $T= T_{n}({\mathbb C}) / T_{n-1}({\mathbb Z})$. Denote by $T_{n}({\mathbb C})_{\rm gen}$ the set of all generic vectors $v$ in $T_{n}({\mathbb C})$ such that $V(T(v))$ is simple, equivalently, $v_{rs}-v_{r-1,t}\notin\mathbb{Z}$ for any $r,s,t$.
Until the end of this section we fix $(i,j,k)$ such that  $1\leq i < j \leq k\leq n-1$. 

By ${\mathcal H}$ we denote the hyperplane $v_{ki} - v_{kj} = 0$ in $ T_{n}(\mathbb{C})$, also by $\tau \in S_{n-1} \times\cdots\times S_{1}$ we denote the transposition on the $k$th row interchanging the $i$th and  $j$th entries. $\overline{\mathcal H}$ stands for the subset of all $w$ in $ T_{n}(\mathbb{C})$ such that $w_{tr} \neq w_{ts}$ for all triples $(t,r,s)$ except for $(t,r,s) = (k,i,j)$. Finally, by ${\mathcal F}_{ij}$ denote the subspace of ${\mathcal F}$ consisting of all functions that are smooth on $\overline{\mathcal H}$.

Let us fix $\bar{v}$ in ${\mathcal H}$ such that $\bar{v}_{ki} = \bar{v}_{kj}$ and all other differences $v_{mr} - v_{ms}$ are noninteger.  In other words, $\bar{v} \in {\mathcal H}$ and $\bar{v} +  T_{n-1}(\mathbb{Z}) \subset \overline{\mathcal H}$. 

 \begin{remark}
For any generic vector $w$ we can choose a representative of the class $w+ T_{n-1}(\mathbb{Z})$ of $w$ in $T=  T_{n}(\mathbb{C})_{\rm gen} /  T_{n-1}(\mathbb{Z})$ as ``close'' as possible to $\bar{v}$ as follows. Let $m_{rs}:=\lfloor {\rm Re}(\bar{v}_{rs}-w_{rs})\rfloor$ (the integer part of the real part of $\bar{v}_{rs}-w_{rs}$),   $m$ be the vector in $ T_{n-1}(\mathbb{Z})$ with components $m_{rs}$, and  $\bar{v}[w] :=w+m$. Then $\mathcal{S} = \{\bar{v} [ w ] \; | \; w\in T_{n}(\mathbb{C})_{\rm gen}\}$ is a set of representatives of $T$.
\end{remark}

Our goal is construct a module $V(T(\bar{v}))$ with Gelfand-Tsetlin support $\{\chi_{\bar{v}+m}\ |\  m \in T_{n-1}(\mathbb{Z})\}$. We will refer to this module as the $1$-{\it singular universal tableaux Gelfand-Tsetlin module associated with $\bar{v}$}, or simply as the \emph{universal module}.

 We formally introduce the complex vector space $V(T(\bar{v}))$ as the one  spanned by vectors $\{T(\bar{v}+z), {\mathcal D}T({\bar{v}} + z)  \ |\ z \in  T_{n-1}(\mathbb{Z})\}$ subject to the relations $T(\bar{v} + z) - T(\bar{v} +\tau(z)) = 0$ and  ${\mathcal D}T({\bar{v}} + z) + {\mathcal D} T({\bar{v}} + \tau(z)) = 0$. We will refer to $T(u)$ as  {\it the regular Gelfand-Tsetlin tableau} associated with $u$ and to ${\mathcal D} T(u)$ as {\it the derivative Gelfand-Tsetlin tableau} associated with $u$.  
\begin{remark} \label{Remark: The set of all singular tableaux is not a basis}
Although $\{ T(\bar{v} + z), \,\mathcal{D} T(\bar{v} + z) \; | \; z \in  T_{n-1}(\mathbb{Z})\}$ is not a basis,  we  have the following natural basis of $V(T(\bar{v}))$:
$$
\mathcal{B}(T(\bar{v}))=\{ T(\bar{v} + z), \mathcal{D} T(\bar{v} + w) \; | \; z_{ki} \leq z_{kj}, w_{ki} > w_{kj}\}.
$$
\end{remark}

Set  ${\mathcal V}_{\rm gen} = \bigoplus_{v \in S} V(T(v))$ and  ${\mathcal V}' = V(T(\bar{v})) \oplus {\mathcal V}_{\rm gen}$. Then ${\mathcal F} \otimes {\mathcal V}_{\rm gen}$  is a  $\mathfrak{gl} (n)$-module  with the  trivial action on ${\mathcal F}$. We next define a $\mathfrak{gl}(n)$-module structure on $V(T(\bar{v}))$.\\
 
The {\it evaluation map} $\mbox{ev}(\bar{v}) : {\mathcal F}_{ij} \otimes {\mathcal V}' \to  {\mathcal V}'$ is the linear map defined by\\ $
\mbox{ev}(\bar{v})(f T(v+z))=f(\bar{v}) T(\bar{v}+z),\ 
\mbox{ev}(\bar{v})(f {\mathcal D} T(\bar{v}+z))= f(\bar{v}) {\mathcal D} T(\bar{v}+z)$. Furthermore, $\mathcal{D}^{\bar{v}}:  {\mathcal F}_{ij} \otimes V(T(v)) \to  V(T(\bar{v}))$ will denote the linear map defined by \\ $\mathcal{D}^{\bar{v}} (f T(v+z)) = \mathcal{D}^{\bar{v}} (f) T(\bar{v}+z) +   f(\bar{v}) \mathcal{D} T(\bar{v}+z)
$, where $\mathcal{D}^{\bar{v}}(f) = \frac{1}{2}\left(\frac{\partial f}{\partial v_{ki}}-\frac{\partial f}{\partial v_{kj}}\right)(\bar{v})$, $z \in  T_{n-1}(\mathbb{Z})$, $f\in {\mathcal F}_{ij}$ and $v \in {\mathcal S}$. We may think of $\mathcal{D}^{\bar{v}}$ as the map 
$$\mathcal{D}^{\bar{v}} \otimes \mbox{ev}(\bar{v})  + \mbox{ev}(\bar{v})  \otimes \mathcal{D}^{\bar{v}}.$$
This map extends to a linear map  ${\mathcal F}_{ij} \otimes {\mathcal V}_{\rm gen} \to V (T(\bar{v}))$ which we will also denote by $\mathcal{D}^{\bar{v}}$.

\begin{theorem}[\cite{FGR3} Theorems 4.9 and 5.6]\label{Theorem: 1-singular modules}
$V(T(\bar{v}))$ has structure of Gelfand-Tsetlin module over $\mathfrak{gl}(n)$ with action of the generators of $\gl(n)$ given by
\begin{align}
E_{rs}(T(\bar{v} + z))=& \mathcal{D}^{\bar{v}}((v_{ki} - v_{kj})E_{rs}(T(v + z))),\label{Equation: action of gl(n) on regular tableaux}\\
E_{rs}(\mathcal{D} ( T(\bar{v} + w)))=&\mathcal{D}^{\bar{v}} ( E_{rs}(T(v + w))),\label{Equation: action of gl(n) on derivative tableaux}
\end{align}

and action of the generators of $\Gamma$ given by 
\begin{align}
c_{rs}(T(\bar{v}+z))=&\mathcal{D}^{\bar{v}}((v_{ki} - v_{kj})c_{rs}(T(v+z))),\label{Equation: action of Gamma on regular tableaux}
\\
c_{rs}(\mathcal{D}^{\bar{v}}(T(v+w)))=& \mathcal{D}^{\bar{v}}(c_{rs}(T(v+w))),\label{Equation: action of Gamma on derivative tableaux}
\end{align}
where $v$ is a generic vector in the set of representatives $\mathcal{S}$, and $z, w \in  T_{n-1}(\mathbb{Z})$ with $w \neq \tau(w)$. 
\end{theorem}

 \begin{remark}
In the case of $\mathfrak{gl}(3)$ we can give the following interpretation of the basis elements of the module $V(T(\bar{v}))$. Let $T(v)$ be a generic tableau such that $V(T(v))$ is simple,  and let $T(\bar{v})$ be  such that $\bar{v}_{21} = \bar{v}_{22}$: 

\begin{center}
 \hspace{1.6cm}\Stone{$v_{31}$}\Stone{$v_{32}$}\Stone{$v_{33}$}\hspace{1cm}\Stone{$\bar{v}_{31}$}\Stone{$\bar{v}_{32}$}\Stone{$\bar{v}_{33}$}\\[0.2pt]
 $T(v)$= \hspace{0.3cm}\Stone{$v_{21}$}\Stone{$v_{22}$} \hspace{0.8cm} $T(\bar{v})$
=  \hspace{0.2cm} \Stone{$\bar{v}_{21}$}\Stone{$\bar{v}_{22}$} \\[0.2pt]
 \hspace{1.3cm}\Stone{$v_{11}$}\hspace{4cm}\Stone{$\bar{v}_{11}.$}\\
\end{center}
Then  $T(\bar{v}+(m,n,k))$ and $\mathcal{D}T(\bar{v}+(m,n,k))$ can be considered as formal limits in the following way:

\begin{equation}\label{Equation: singular tableau as formal limit}
T(\bar{v}+(m,n,k)):=\lim\limits_{v\to \bar{v}}T(v+(m,n,k)),
\end{equation}

\begin{equation}\label{Equation: singular derivative tableau as formal limit}
\mathcal{D}T(\bar{v}+(m,n,k)):=\lim\limits_{v\to \bar{v}}\left(\frac{T(v+(m,n,k))-T(v+(n,m,k))}{v_{21}-v_{22}}\right).
\end{equation}
 \end{remark}

One essential property of generic Gelfand-Tsetlin modules  described in Theorem \ref{Theorem: Generic GT modules} is that $\Gamma$ {\it ``separates"} the basis  elements of $\mathcal{B}(T(v))$, that is, for any two different tableaux $T_{1},\ T_{2}$ in $\mathcal{B}(T(v))$ there exists an element $\gamma\in\Gamma$ such that $\gamma \cdot T_{1}=T_{1}$ and $\gamma \cdot T_{2}=0$ (see Remark \ref{Remark: Gamma separates generic tableaux}). In the case of  $1$-singular modules $V(T(\bar{v}))$ this is also true and follows from the fact that no derivative tableau $\mathcal{D}(T(\bar{v}+w))$ is an eigenvector for the action of $c_{k2}\in \Gamma$. A detailed proof can be found in \cite{GoR}, Section $\S 5$.
\begin{theorem}\label{Theorem: Gamma separates tableaux in 1-singular module}
Let $\bar{v}$ be any $1$-singular vector and $\mathcal{B}(T(\bar{v}))$ be  as before. Then $\Gamma$ separates the tableaux in the basis $\mathcal{B}(T(\bar{v}))$. 
\end{theorem}

In the case of $T_{3}(\mathbb{C})$ (equivalently, Gelfand-Tsetlin tableaux of height $3$) every singular vector is a $1$-singular vector. Therefore, in the case of $\mathfrak{gl}(3)$ the $1$-singular modules exhaust all singular Gelfand-Tsetlin modules.

\begin{example}\label{Example: Singular Verma}
The simple Verma $\mathfrak{sl}(3)$-module $M = M(-1,-1)$ admits a tableaux realization as a subquotient of the module $V(T(\bar{v}))$, where $\bar{v}$ is the vector $(-1,-1,-1,-1,-1,-1)$. This module  contains Gelfand-Tsetlin characters of dimension $2$. For example, if $\chi$ is the Gelfand-Tsetlin character associated with the tableaux  $T(\bar{v}+(-1,0,-1))$ and $\mathcal{D}T(\bar{v}+(0,-1,-1))$, then $\dim(M(\chi))=2$ (see \S \ref{subsection: Realizations of all simple singular Gelfand-Tsetlin modules for sl(3)} (C13) for details).
\end{example}

\section{Gelfand-Tsetlin modules of $\mathfrak{sl}(3)$ and modules of  $C({\mathfrak h})$}\label{Section: Gelfand-Tsetlin modules of sl(3)}

From now on we focus on the case $n=3$ and $\g=\mathfrak{sl}(3)$. We fix the standard Gelfand-Tsetlin subalgebra $\Ga$ of $\g$, that is
the one corresponding to the chain whose second component is generated by $E_{12}$ and $E_{21}$. The corresponding category of Gelfand-Tsetlin modules $\mathcal{GT}(3)$ will be denoted simply by $\mathcal{GT}$.

 Let $C({\mathfrak h})$ be the centralizer of the Cartan subalgebra ${\mathfrak h}$ in $U(\g)$, where
  $${\mathfrak h}=\Span_{\mathbb C}\{H_1:=E_{11}-E_{22},\ H_2:=E_{22}-E_{33}\}.$$  
In this section we collect some  properties of modules in $\mathcal{GT}$ that are related to the category of modules of $C({\mathfrak h})$. The results are based on the works  \cite{Fut1}, \cite{Fut2}, \cite{Britten Futorny Lemire}, but for reader's convenience we  provide  proofs for some statemenrs. 

The following result provides an important relation between the simple  $C({\mathfrak h})$-modules and the simple weight modules (for a proof, see for example \cite{Fut3}):

 \begin{lemma}
For any simple $C({\mathfrak h})$-module $W$ there exists an simple weight $\g$-module $M$ such that
$M_{\lambda}\simeq W$ for some $\lambda\in {\mathfrak h}^*$.  Conversely, if $M$ is a simple  weight
$\g$-module then $M_{\lambda}$ is a simple $C({\mathfrak h})$-module.
\end{lemma}

Denote  $A:=E_{12}E_{21}$, $B:=E_{23}E_{32}$. Recall that the center of  $U({\mathfrak g})$ is generated by $c_{32}$ and $c_{33}$, see (\ref{Equation: c_mk}).  For convenience we will also use the following generators of  the center of $U({\mathfrak g})$:
$$c_{1}=\frac{1}{12}c_{32} \text{ \ \ \ \ \ \ \  and \ \ \ \ \ \ }\ c_{2}=\frac{3}{2}c_{32}-c_{33}.$$

The next two lemmas provide important technical properties of $C({\mathfrak h})$  and the simple $C({\mathfrak h})$-modules.

 \begin{lemma}[\cite{Britten Futorny Lemire}, Lemma 1.1]
The centralizer $C({\mathfrak h})$ is an associative algebra generated by $H_1$, $H_2$, $A$, $B$, $c_1$, and $c_2$.
\end{lemma}

\begin{lemma} [\cite{Fut1}]\label{Lemma: Relations AB and constants ,p,eta} Let $W$ be a simple $C({\mathfrak h})$-module, and let  $H_i=h_i {\rm Id}$, $c_i=\gamma_i {\rm Id}$ on $W$,  for some constants $h_i, \gamma_i$,  $i = 1,2$. Then the following identities hold on $W$.
\begin{eqnarray*}
aA & = & A^2+AB+BA+ABA-\frac{1}{2}A^2 B-\frac{1}{2}B A^2 +rB+\tau I, \\
aB & =& B^2+AB+BA+BAB-\frac{1}{2}B^2 A-\frac{1}{2}A B^2 +r_1B+\tau_1 I,\\
\frac{1}{4}(AB-BA)^2&=&
ABA+BAB+\frac{1}{2}r B^2+\frac{1}{2}r_1 A^2-\\
& &(\frac{a}{2}-1)(AB+BA)+(\tau + r)B+(\tau_1+r_1)A+\eta I,
\end{eqnarray*}

where
\begin{eqnarray*}
&&r:=\frac{1}{2}(h_{1}^{2}-2h_{1}), \; r_{1}:=\frac{1}{2}(h_{2}^{2}-2h_{2}),\; a:=6\gamma_{1}+h_1+h_2-\frac{1}{3}h_{1}^{2}-\frac{1}{3}h_{2}^{2}+\frac{1}{6}h_{1}h_{2},\\
&&3p:=\frac{1}{9}(h_{1}-h_{2})^{3}-\gamma_{2}+6\gamma_{1}(h_{2}-h_{1}+3)-h_{1}^{2}-h_{2}^{2}-h_{1}h_{2}+2h_{1}-2h_{2}, \\
&&\tau:=\frac{1}{2}h_{1}p+h_{1}h_{2}, \;  \tau_{1}=-\frac{1}{2}ph_{2}^{2}+\frac{1}{3}h_{2}\left(18\gamma_{1}-h_{1}^{2}-h_{2}^{2}-h_{1}h_{2}+3h_{1}\right),\\
&&\eta:=\frac{1}{4}p^{2}+\frac{1}{6}p(h_{1}h_{2}+h_{1}^{2}+h_{2}^{2}-18\gamma_{1})+\frac{1}{4}h_{1}h_{2}(h_{1}h_{2}+4-2a).
\end{eqnarray*}

\end{lemma}


Let $M$ be a simple module in $\mathcal{GT}$. In particular, it is a weight module. Consider any $\lambda\in {\mathfrak h}^*$
from the weight support of $M$.  The central elements $c_1$ and $c_2$ act on $M$, and hence on $M_{\lambda}$, as a multiplication by some complex scalars $\gamma_1$ and $\gamma_2$, respectively. If $\lambda(H_i)=h_{i}$,  then 
 $H_i=h_i {\rm Id}$, $i=1,2$,  on $M_{\lambda}$. Since $M$ is a Gelfand-Tsetlin module, then each component
$M(\chi)$ is finite dimensional. Hence, one can
choose a basis of  $M_{\lambda}$ with respect to which $A|_{M_{\lambda}}$ has
a Jordan canonical form $A_{\lambda}$.

Consider the following polynomial in $2$ variables:
$$g_{\lambda}(x,y)=(x-y)^2-2(x+y)-2r.$$
Recall $r=\frac{1}{2}(h_1^2-2h_1)$. Note that $g_{\lambda}(x,y)$ depends only on $h_{1}=\lambda(H_{1})$.

\begin{definition} Let $\lambda\in {\mathfrak h}^*$.
\begin{itemize}
\item[(i)] A  sequence  $(\mu_i)_ {i \in J}$ is \emph{$\lambda$-connected} (or, simply, \emph{connected}) if $g_{\lambda}(\mu_i, \mu_{i+1})=0$ for all $i \in J$, for some connected subset $J$ of  $\mathbb{Z}$.  A subsequence $(\mu_{j'})_{j'\in  J'}$ of  a connected sequence $(\mu_j)_{j\in  J}$ will be called a \emph{connected subsequence} if $J'$ is connected.
\item[(ii)] A $\lambda$-connected sequence $(\mu_j)_{j\in  J}$ with $\mu_i\neq\mu_j$ for any $i\neq j$, is called \emph{$\lambda$-connected chain}.
\item[(iii)] We will say that a set $B$ is \emph{$\lambda$-connected} if the elements of $B$ can be ordered in a $\lambda$-connected chain $(\mu_j)_{j\in  J}$. We will call the sequence   $(\mu_j)_{j\in  J}$ \emph{the $\lambda$-connected chain associated to $B$}.
\end{itemize}
\end{definition}

 We note that if  $B$ is $\lambda$-connected, then there are at most two  $\lambda$-connected chains associated to $B$.

\begin{lemma}
Let $M$ be a simple Gelfand-Tsetlin module and $\lambda\in {\mathfrak h}^*$ a weight of $M$. Then the distinct eigenvalues of
$A_{\lambda}$ is a $\lambda$-connected set.
\end{lemma}

\begin{proof}
This lemma follows from Lemmas 2.2 and 2.3 in \cite{Britten Futorny Lemire}, but for reader's convenience, a brief proof is written. Let $[A_{\lambda}]={\rm diag}(A_i)$ be the matrix of $A_{\lambda}$ in a fixed Jordan canonical basis of $M_{\lambda}$, where $A_i$ corresponds to the generalized eigenspace with eigenvalue $\mu_i$, and let  $[B_{\lambda}]=(B_{ij})$ be the block matrix of $B_{\lambda}$ relative to this basis, for which $B_{ii}$ and $A_{i}$ are in the same position. If $i \neq j$, looking at the $(i,j)$-th block of the matrix equation corresponding to the first relation in Lemma \ref{Lemma: Relations AB and constants ,p,eta}, we obtain: 
$$0=A_i B_{ij}+B_{ij} A_{j}+A_i B_{ij} A_j-\frac{1}{2}B_{ij} A^2_j-\frac{1}{2} A^2_i B_{ij}+r B_{ij}.$$ 
Applying the above identity to suitable elements of the basis of $M_{\lambda}$, one can see
 that $B_{ij}=0$ if $(\mu_i+\mu_j)-\frac{1}{2}(\mu_i-\mu_j)^2+r\neq 0$ (for details see the proof of \cite[ Lemma 2.2]{Britten Futorny Lemire}).  Since $M$ is simple, then $M_{\lambda}$ is a simple $C(\mathfrak h)$-module. However, if the set of eigenvalues of $A$ is a union of two ``disconnected'' sets, then one easily can prove that $M_{\lambda}$ is a direct sum of two modules, which is a contradiction (for details see the proof of \cite[ Lemma 2.3]{Britten Futorny Lemire}).
\end{proof}

Until the end of this section we assume that $M$ is a simple Gelfand-Tsetlin module, $\lambda \in {\rm Supp} \ M$.

\begin{lemma}\label{Lemma: relation between eigenvalues in connected chain} 
 If $(\mu_{j})_{j\in  J}$ is a $\lambda$-connected sequence  with $|J|\geq 3$, then the following recurrence relation holds for any $i$:
 \begin{equation}\label{Equation: relation for eigenvalues in connected chains}
 \mu_{i+1}+\mu_{i-1}=2\mu_{i}+2.
 \end{equation}
 \end{lemma}

\begin{proof}
As $(\mu_{j})_{j\in  J}$ is a connected sequence, we have $g_{\lambda}(\mu_{i},\mu_{i-1})=0$ and $g_{\lambda}(\mu_{i},\mu_{i+1})=0$. By solving both  quadratic equations in terms of $\mu_{i}$ we have:
\begin{align*}
\{\mu_{i-1},\mu_{i+1}\}=\{\mu_{i}+1\pm\sqrt{1+4\mu_{i}+h_1^2-2h_1}\}
\end{align*}
therefore, $\mu_{i+1}+\mu_{i-1}=2\mu_{i}+2$.
\end{proof}

\begin{definition} Let $(\mu_{j})_{j\in  J}$ be a $\lambda$-connected sequence.  We say that $(\mu_{j})_{j\in  J}$ is:
\begin{itemize}
\item[(i)] \emph{degenerate} if $\mu_i=-\frac{1}{2}r$ for some $i\in J$.
\item[(ii)] \emph{critical} if $\mu_i=-1/4 - \frac{1}{2} r$ for some $i\in J$.
\item[(iii)] \emph{singular} if  $(\mu_{j})_{j\in  J}$ is a connected subsequence of a degenerate or a critical connected sequence.
\item[(iv)] \emph{generic} if it is not singular.
\end{itemize}

\end{definition}

\begin{example}
if $\lambda=(0,0)$, then the sequence $\{0,2,6,12,\ldots\}=\{n(n+1)\ |\ n\geq 0\}$ is a degenerate connected chain and $C_{i}:=\{n(n+1)\ |\ n\geq i\}$ is a connected chain that is not degenerate  for each $i>0$.
 \end{example}

\begin{lemma}\label{Lemma: connected chains explicitly}
Let $(\mu_{j})_{j\in  J}$ be a $\lambda$-connected sequence.
\begin{itemize}
\item[(i)] If $(\mu_{i})_{j\in  J}$ is degenerate, then $\mu_{i} \in \{n(n+1)-\frac{r}{2}\ |\ \ n\geq 0\}$ for all $i \in J$.
\item[(ii)] If $(\mu_{i})_{j\in  J}$ is critical, then  $\mu_{i} \in \{n^2-\frac{1}{4}-\frac{r}{2}\ |\ \ n\geq 0\}$ for all $i \in J$.
\item[(iii)] If $(\mu_{i})_{j\in  J}$ is generic, then $\mu_{i} \in \{n^2+n\sqrt{1+4\mu_{0}+2r}+\mu_{0}\; |\; n\in\mathbb{Z}\}$ for all $i \in J$, where $\sqrt{1+4\mu_{0}+2r}$ is a fixed solution of $x^2 = 1+4\mu_{0}+2r$.
\end{itemize}
\end{lemma}

\begin{proof}
By Lemma \ref{Lemma: relation between eigenvalues in connected chain}, in order to  determine a connected sequence, it is enough to know two connected values. In the case of a degenerate sequence we have $\mu_{i}=-\frac{1}{2}r$ for some $i$, and hence, $\mu_{i+1}=\mu_{i}$ or $\mu_{i-1}=\mu_{i}$. Assume $\mu_{i-1}=\mu_{i}$ and $i=0$  for simplicity. Then $\mu_n = n(n+1)-\frac{r}{2}$, $n \in {\mathbb Z}$, give the unique solution of the recursive equation (\ref{Equation: relation for eigenvalues in connected chains}). Therefore, all $\mu_{i}$ are in the set
$\{n(n+1)-\frac{r}{2}\ |\ \ n\geq 0\}$. This proves part (i). For parts (ii) and (iii) we reason in the same manner. Namely,  for  a critical chain, we use that $\mu_{i}=-\frac{1}{4}-\frac{1}{2}r$ for some $i$, hence $\mu_{i+1}=\mu_{i}+1$ or $\mu_{i-1}=\mu_{i}+1$, while 
in the generic case, given $\mu_{i}$ in the connected sequence, we have $\mu_{i+1}\in\{\mu_{i}+1\pm\sqrt{1+4\mu_{i}+2r}\}$. 
\end{proof}

The properties of $A_{\lambda}$ are described in the following theorem.

\begin{theorem}\label{Theorem: Description of eigenvalues of A_lambda}
Let $M$ be a simple Gelfand-Tsetlin module. Then the following hold.
\begin{itemize}
\item[(i)] For every $\lambda \in \Supp M$, every eigenvalue  of $A_{\lambda}$ has multiplicity 
at most $2$.
\item[(ii)] If $(\mu_{i})_{i\in J}$ is  a connected chain of the set of distinct eigenvalues of $A_{\lambda}=A\mid_{M_{\lambda}}$, then
\begin{itemize}
\item[(a)] if the chain $(\mu_{i})_{i\in J}$ is generic then all eigenvalues of $A_{\lambda}$ are distinct;
\item[(b)]  if the chain $(\mu_{i})_{i\in J}$ is degenerate then the chain can be chosen so that $\mu_1=-\frac{r}{2}$, and if the multiplicity of $\mu_i$ equals $1$
then  the multiplicity of $\mu_{i+1}$ is also $1$;
\item[(c)] if the chain $(\mu_{i})_{i\in J}$ is critical then  the chain can be chosen so that  $\mu_1+1=\mu_2$, the multiplicity of $\mu_1$ is $1$, and if the multiplicity of $\mu_i$ equals $1$ for $i>1$, then  the multiplicity of $\mu_{i+1}$  also equals $1$;
\item[(d)] if the chain $(\mu_{i})_{i\in J}$  is singular but not degenerate or critical, then all eigenvalues of $A_{\lambda}$ are distinct.
\end{itemize}

\end{itemize}
\end{theorem}

\begin{proof}
The proof of all parts can be found in \cite{Fut3}, \cite{Fut2}. The strategy is to apply the relations from Lemma \ref{Lemma: Relations AB and constants ,p,eta} to a Jordan form of $A_{\lambda}$. Proofs of parts (b) and (d) can be also found in \cite{Britten Futorny Lemire}, Theorem 2.7. \end{proof}

As a consequence of Theorem \ref{Theorem: Description of eigenvalues of A_lambda} we have the following statement.

\begin{corollary}\label{cor-ker}
Let $M$ be a simple weight $\mathfrak{g}$-module. Then for any $\lambda\in \mathfrak h^*$ and any $i\neq j$ we have
$\dim(\ker(E_{ij}|_{M_{\lambda}})) \leq 1$.
\end{corollary}

\begin{proof}
Suppose that $\ker(E_{ij}|_{M{_\lambda}}) \neq 0$. Consider the Lie subalgebra $\mathfrak a$ of $\mathfrak{sl}(3)$ isomorphic to $\mathfrak{sl}(2)$ that is generated by $E_{ij}$ and $E_{ji}$. Then $M$ is a Gelfand-Tsetlin  $\mathfrak a$-module with respect to the 
Gelfand-Tsetlin subalgebra generated by  $\mathfrak h$, the center of $U(\mathfrak{sl}(3))$, and the center of $U(\mathfrak a)$. The statement follows from Theorem \ref{Theorem: Description of eigenvalues of A_lambda}. 
\end{proof}

\begin{remark}\label{Remark:connected chains for GT tableaux}
In what follows, we give an interpretation of the eigenvalues of $A_{\lambda}$ in terms of tableaux and the Gelfand-Tsetlin formulas. Set for $i\in {\mathbb Z}$
\begin{center}
 \hspace{1.5cm} \Stone{a}\Stone{b}\Stone{c}\\[0.2pt]
 $T(v_{i})$=\hspace{0.5cm} \Stone{x+i}\Stone{y-i}\hspace{1.2cm}\\[0.2pt]
 \hspace{1.3cm} \Stone{z}\\
\end{center}
\noindent
The set of all tableaux in $\mathcal{B}(T(v_{0}))$ with fixed $\mathfrak{g}$-weight $\lambda \in {\mathfrak h}^*$ (see Definition \ref{Definition: weight of a tableau}) is $\{T(v_{i})\ |\ i\in \mathbb{Z}\}$. Moreover, the Gelfand-Tsetlin formulas imply that $A=E_{12}E_{21}$ acts on $T(v_{i})$ as multiplication by $\mu_{i}:=-(x+i-z)(y-i-z)$, and $r=\frac{1}{2}((2z-x-y-1)^{2}-(2z-x-y-1))$. We have:

\begin{itemize}
\item[(i)] $(\mu_{i})_{i\in \mathbb{Z}}$  is degenerate with $\mu_{j}=-\frac{1}{2}r$ if and only if $x-y\in \{2j-1,2j+1\}$.
\item[(ii)] $(\mu_{i})_{i\in \mathbb{Z}}$ is critical with $\mu_{j}=-\frac{1}{4}-\frac{1}{2}r$ if and only if $x-y=2j$.
\end{itemize}
In particular, $(\mu_{i})_{i\in \mathbb{Z}}$  is singular if and only the tableau $T(v_0)$ (respectively, $\chi_{v_0}$) is singular and $\{\mu_{i}\}$ is generic if and only if $T(v_0)$ (respectively, $\chi_{v_0}$) is generic. Note also that $\chi(A)=-\frac{1}{4}-\frac{1}{2}r$ if and only if $\chi=\chi_{w}$ with $w_{21}=w_{22}$. Finally, $\chi(A)=-\frac{1}{2}r$ if and only if $\chi=\chi_{w}$ with $|w_{21}-w_{22}|=1$.
\end{remark}

\begin{definition}
We say that a Gelfand-Tsetlin character  $\chi$ is  \emph{critical} (respectively,  \emph{degenerate}) if $\chi = \chi_{v}$ for some $v$ such that $v_{21}=v_{22}$  (respectively,  $|v_{21}-v_{22}| = 1$).
\end{definition}

Since $A\in \Ga$, we can extend the concepts of generic and singular chains to Gelfand-Tsetlin modules.
\begin{definition}
A Gelfand-Tsetlin module $M$ is called \emph{generic} if every Gelfand-Tsetlin character of $M$ is generic. A Gelfand-Tsetlin module $M$ is called \emph{singular} if it has a singular Gelfand-Tsetlin character.
\end{definition}

Note that any finite-dimensional module is a singular Gelfand-Tsetlin module, moreover, any $1$-singular module as defined in Section \S \ref{subsection: Singular Gelfand-Tsetlin modules} is a singular Gelfand-Tsetlin module. Also, generic modules as defined in \S \ref{subsection: Generic modules} are generic Gelfand-Tsetlin modules.

\begin{proposition}
 If a simple  Gelfand-Tsetlin $\mathfrak{g}$-module $M$ is singular, then each Gelfand-Tsetlin character of $M$ is singular.
  \end{proposition}
  
\begin{proof}
The statement follows by a direct computation. Let $\chi$ be a singular Gelfand-Tsetlin character of $M$, $v\in M(\chi)$,
$v\neq 0$.  If $E_{12}v\neq 0$ then we easily check that $E_{12}v\in M(\chi')$ for a singular $\chi'$.  Similar reasoning applies for $E_{21}v$. Suppose now
 $E_{23}v\neq 0$. Then $E_{23}v\in M(\chi')\oplus M(\chi'')$, where $\chi'$ and $\chi''$ are both singular (one of the subspaces can be zero). Moreover, if $\chi$ belongs to a critical (respectively, degenerate) connected chain, then $\chi'$ and $\chi''$ belong to
 a degenerate (respectively, critical) connected chain. We reason similarly for $E_{32}$.
\end{proof}

\begin{definition} Given a Gelfand-Tsetlin character $\chi$ and $M$ a Gelfand-Tsetlin module, we say that $M$ is a \emph{simple extension} of the character $\chi$ if $M$ is simple and $\chi\in\Supp_{GT}(M)$ (i.e. $M(\chi)\neq 0$).
\end{definition}

\begin{lemma}\label{Lemma: eigenvalues of A_lambda and chi determine unique extension}
Let $M$ be a simple Gelfand-Tsetlin module, $\chi\in \Supp_{GT}(M)$ and $\lambda=\chi |_\mathfrak{h}$. If there exists a basis  $\{w_{i}\}_{i\in I}$  of $M_{\lambda}$ such that the action of $B_{\lambda}$ on this basis is completely determined by the eigenvalues of $A_{\lambda}$ and $\chi$, then $M$ is the unique simple extension of $\chi$.
\end{lemma}
\begin{proof}
Under these conditions, the simple $C({\mathfrak h})$-module $M_{\lambda}$ is defined uniquely, moreover, as  $M^{(1)}_{\lambda}\simeq M^{(2)}_{\lambda}$ implies $M^{(1)}\simeq M^{(2)}$, the uniqueness follows.
\end{proof}

The generic Gelfand-Tsetlin modules are completely determined by any of their characters, as the following result shows.

\begin{theorem}\label{Theorem: Uniqueness for generic modules}
If $M$ is a generic simple Gelfand-Tsetlin module then for any $\chi\in\Supp_{GT}(M)$ the subspace
$M(\chi)$ is one dimensional and $M$ is the unique simple extension of $\chi$.
\end{theorem}

\begin{proof} The result can be found in \cite{Fut3} and \cite{Fut2}, but for the sake of completeness we provide a proof. Let $\chi\in \Supp_{GT}(M)$ and let $\lambda = \chi|_{\mathfrak h}$ be the associated to $\chi$ weight. Let $\mu_{0}\in\mathbb{C}$
be an eigenvalue of the operator $A_{\lambda}= A|_{M_{\lambda}}$. Then all eigenvalues of $A_{\lambda}$ form a connected chain, i.e. belong to a sequence  $\mu_{i}=i^{2}+i \sqrt{1+4\mu_{0}+2r}+\mu_{0}$, $i\in\mathbb{Z}$ for some choice of the square root (see Lemma \ref{Lemma: connected chains explicitly}(iii)).

Using relations between $A$ and $B$ we can choose a basis $\{w_{i}\ |\ i\in\mathbb{Z}\}$ (this set can be finite or bounded from one side or unbounded) of $M_{\lambda}$ such that $$A_{\lambda}w_{i}=\mu_{i}w_{i} \text{;\ \  and }
B_{\lambda}w_{i}=\begin{cases}
w_{i-1}+b_{i}w_{i}+d_{i+1}w_{i+1}, & i<0\\
w_{-1}+b_{0}w_{0}+w_{1}, & i=0\\
d_{i}w_{i-1}+b_{i}w_{i}+w_{i+1}, & i>0
\end{cases}
$$
with

$$b_{i}:=\frac{a\mu_{i}-\mu_{i}^{2}-\tau}{2\mu_{i}+r},$$
$$d_{i}:=\frac{\xi(\mu_{i-1})(3+\mu_{i-1}-\mu_{i})-\theta(\mu_{i-1})\left(\frac{7}{2}\mu_{i-1}-\frac{3}{2}\mu_{i}+3+r\right)}{4(\mu_{i-1}-\mu_{i}+1)\left(\mu_{i-1}-\frac{3}{4}+\frac{1}{2}r\right)},$$
$$\xi(\mu_{i}):=\frac{1}{2}(2\mu_{i}+r)b_{i}^{2}-(2\mu_{i}+r)b_{i}-\frac{1}{2}r_{1}\mu_{i}^{2}-(r_{1}+\tau_{1})\mu_{i}-\eta,$$
$$\theta(\mu_{i}):=(a-2\mu_{i})b_{i}-b_{i}^{2}-r_{1}\mu_{i}-\tau_{1},$$
where $r$, $r_{1}$, $\eta$, $\tau$ and $\tau_1$ are defined in Lemma \ref{Lemma: Relations AB and constants ,p,eta}.

 Hence, in this case  $B_{\lambda}$ is completely determined by $\chi$ and $A_{\lambda}$.
 The uniqueness follows from Lemma \ref{Lemma: eigenvalues of A_lambda and chi determine unique extension}. \end{proof}

Note that in singular cases the subspace
$M(\chi)$ can be  $2$-dimensional (see Example \ref{Example: Singular Verma}). Also, in these cases for a given $\chi\in \Ga^{*}$ there can exist two non-isomorphic simple extensions of $\chi$. Such examples were first constructed in \cite{Fut1}.

\section{simple extensions of singular Gelfand-Tsetlin characters}\label{Section: Further properties of Gelfand-Tsetlin modules for sl(3)}
In this section we provide sufficient conditions for a singular Gelfand-Tsetlin character to admit a unique simple extension.

\begin{theorem}\label{Theorem: uniqueness for critical characters}
If $\chi$ is a critical Gelfand-Tsetlin character then $\chi$ admits a unique simple extension.
\end{theorem}

\begin{proof}
By Theorem \ref{Theorem: finiteness-for-gl-n} there exist at most $2$ simple modules $M^{(1)}$ and $M^{(2)}$ such that $\chi\in\Supp_{GT}(M^{(i)})$ for $i=1,2$. Assume that we have two such modules and let $\lambda\in\Supp(M^{(1)})\cap\Supp(M^{(2)})$ such that $\lambda=\chi|_{\mathfrak h}$. For $i=1,2$, consider the restriction $A^{(i)}_{\lambda}$ of $A=E_{12}E_{21}$ on $M^{(i)}_{\lambda}$. Since $M^{(i)}$ are Gelfand-Tsetlin modules, then we can choose bases ${\mathcal B}_{1}=\{w_0,\ldots,w_m\}$, $0\leq m\leq\infty$ and ${\mathcal B}_{2}=\{w'_0,\ldots,w'_k\}$, $0\leq k\leq\infty$ of $M^{(1)}_{\lambda}$ and $M^{(2)}_{\lambda}$ such that the matrix $[A_{\lambda}^{(i)}]$ of $A^{(i)}_{\lambda}$ with respect to ${\mathcal B}_{i}$ is in a Jordan normal form and each eigenvalue of $A^{(i)}_{\lambda}$ has algebraic multiplicity at most $2$.

By Theorem \ref{Theorem: Description of eigenvalues of A_lambda}(ii)(c), the eigenvalue $\chi(A)$ of $A^{(1)}_{\lambda}$ and $A^{(2)}_{\lambda}$ has multiplicity $1$. Suppose first that all eigenvalues of both $A^{(1)}_{\lambda}$ and $A^{(2)}_{\lambda}$ have multiplicity $1$. Then they can be ordered in  connected chains $\{\mu_{n}\ |\  0\leq n \leq m\}$ and $\{\mu_{m}\ |\ 0\leq m\leq k\}$ with $$\mu_{0}=\chi(A)=-\frac{1}{4}-\frac{1}{2}r$$
and $\mu_{i}=i^2-\frac{1}{4}-\frac{1}{2}r$ for  $i\geq 1$ (see Lemma \ref{Lemma: connected chains explicitly}(ii)).

Applying the relations from Lemma \ref{Lemma: Relations AB and constants ,p,eta} we obtain
\begin{equation}\label{Equation: A and B relations}
\begin{split}
Aw_{i}=\mu_i w_i ,\ i\geq 0,\ Bw_{i}=
\begin{cases}
b_{0}w_{0}+w_{1}, & i=0\\
c_{i}w_{i-1}+b_{i}w_{i}+w_{i+1}, & 0<i\leq m
\end{cases}\\
Aw'_{i}=\mu_i w'_i ,\ i\geq 0,\ Bw'_{i}=
\begin{cases}
b_{0}w'_{0}+w'_{1}, & i=0\\
c_{i}w'_{i-1}+b_{i}w'_{i}+w'_{i+1}, & 0<i\leq k
\end{cases}
\end{split}
\end{equation}
where 
\begin{align*}
c_{1}=&\frac{\theta(\mu_{0})}{2},\\ 
c_{2}=&\begin{cases}
\frac{\theta(\mu_{1})}{(1+\mu_{2}-\mu_{1})}, & n=1\\
d_{2}, & n>1
\end{cases},\\
c_{i}=&d_{i},\ \ \ \  i>2
\end{align*} and
$$2\xi(\mu_{0})=\left(\frac{3}{2}+2\mu_{0}+r\right)\theta(\mu_{0}).$$
\noindent
If $m<k$, then $c_{m+1}=0$ implying that $M^{(2)}$ is reducible. Similarly, if $m>k$, then $c_{k+1}=0$ and $M^{(1)}$ is reducible. It follows that $k=m$. But then formulas (\ref{Equation: A and B relations}) define uniquely a simple $C({\mathfrak h})$-module $M^{(1)}_{\lambda}\simeq M^{(2)}_{\lambda}$. Therefore, $M^{(1)}\simeq M^{(2)}$.

Suppose now that the algebraic multiplicity of some $\mu_{i}$ is two in $M^{(1)}_{\lambda}$. For simplicity assume that
$$[A_{\lambda}^{(1)}]=\left(
  \begin{array}{cccc}
    \mu_{0} & \mid & 0 & 0 \\
    \_\ \_\ \_\ \_ & \mid & \_\ \_\ \_ & \_\ \_\ \_\\
    0 & \mid & \mu_{1} & 1 \\
    0 & \mid & 0 & \mu_{1} \\
  \end{array}
\right) \text{ and }
[B_{\lambda}^{(1)}]=\left(
  \begin{array}{cccc}
    b_{11} & \mid & b_{12} & b_{13} \\
    \_\ \_\ \_\ \_ & \mid & \_\ \_\ \_ & \_\ \_\ \_\\
    b_{21} & \mid & b_{22} & b_{23} \\
    b_{31} & \mid & b_{32} & b_{33} \\
  \end{array}
\right),$$
where $[B_{\lambda}^{(1)}]$ stands for the matrix of $B|_{M^{(1)}_{\lambda}}$ relative to $\mathcal B_1$.
Applying the relations from Lemma \ref{Lemma: Relations AB and constants ,p,eta} we obtain $b_{32}=0$. Note that due to the irreducibility of $M_{\lambda}^{(1)}$ as an $C({\mathfrak h})$-module we have $b_{12}\neq 0$ and $b_{31}\neq 0$. Hence, using row operations, one can change the basis $\mathcal B_{1}$ so that $b_{13}$ becomes $0$.\\
\noindent
Now, applying the relations from Lemma \ref{Lemma: Relations AB and constants ,p,eta} we obtain
$$
\begin{cases}
b_{22}=b_{33}, & \\
 \frac{1}{2}b_{11}+\frac{3}{2}b_{22}=a-2\mu_{0}-1, & \\
\frac{1}{2}b_{12}b_{31}=-ab_{22}+b_{22}^{2}+2\mu_{1}b_{22}+r_{1}\mu_{1}+\tau_{1}, & \\
\frac{3}{2}b_{23}=-1-b_{11}-2b_{22}, & \\
2b_{12}b_{21}=-(2\mu_{0}+1-a)(4\mu_{0}+1)-3r_{1}\mu_{1}+r_{1}-3\tau_{1}, & \\
\frac{1}{2}b_{31}b_{12}=(4\mu_{0}+3)(2\mu_{0}+1-a)+r_{1}\mu_{1}+\tau_{1}, & \\
(b_{11}+b_{22})b_{31}=0. &
\end{cases}
$$
By changing the basis if needed, we can assume $b_{12}=1$. Therefore, the matrix $[B_{\lambda}^{(1)}]$ is completely determined by $\chi$ and the matrix $[A_{\lambda}^{(1)}]$. We can show that the latter holds for any Jordan normal form $[A_{\lambda}^{(1)}]$.

Consider now the matrix $[A^{(2)}_{\lambda}]$.
If this Jordan normal form is not equivalent to  $[A^{(1)}_{\lambda}]$, then one of the modules $ M^{(1)}_{\lambda}$ or $ M^{(2)}_{\lambda}$ will be reducible. Indeed, this can be immediately seen from the form of matrices $[B^{(1)}_{\lambda}]$ and $[B^{(2)}_{\lambda}]$.  On the other hand, if $[A^{(1)}_{\lambda}]=[A^{(2)}_{\lambda}]$ then $M^{(1)}_{\lambda}\simeq M^{(2)}_{\lambda}$ as $C({\mathfrak h})$-modules and hence  $M^{(1)}\simeq M^{(2)}$.  \end{proof}

If $\chi$ is a singular Gelfand-Tsetlin character in a critical connected chain and $\chi$ is not critical then there might exist two  simple extensions of $\chi$ (see \cite{Fut1} for examples). On the other hand we have

\begin{corollary}\label{Corollary: non-critical-diag}
Suppose that $\chi$ is a singular character in a critical connected chain and $\chi$ is not critical. If $\lambda = \chi|_{\mathfrak h}$,
then there exists a unique simple extension $M$ of $\chi$ with diagonalizable $A_{\lambda}$.
\end{corollary}

\begin{proof}
Indeed, if $A_{\lambda}$ is diagonalizable then $B_{\lambda}$ is determined uniquely. As it was shown in the proof of
Theorem~\ref{Theorem: uniqueness for critical characters}, it is sufficient to know one eigenvalue of $A_{\lambda}$ to reconstruct  the whole $A_{\lambda}$ in a simple module. Hence, the statement follows.
\end{proof}

\begin{lemma}\label{Lemma: two jordan cells} Let $M$ be a simple Gelfand-Tsetlin module,
$\chi$  a degenerate character of $M$ associated with the weight $\lambda\in {\mathfrak h}^*$. Let $\rho_1=-\frac{1}{2}r$, and $\rho_2=2-\frac{1}{2}r$ be connected eigenvalues of $A_{\lambda}$ with $\chi(A)=\rho_1$.
 Suppose that both $\rho_1$ and $\rho_2$ have multiplicity $2$.  Then the Gelfand-Tsetlin support of $M$ contains  a critical character $\chi'$.
 \end{lemma}

\begin{proof}
Let $u_1, u_2, u_3, u_4$ be non-zero elements of $M$ such that
 $$A_{\lambda}u_1=\rho_1 u_1, \, A_{\lambda}u_2=u_1+\rho_1 u_2,$$ 
 $$A_{\lambda}u_3=\rho_2 u_3, \, A_{\lambda}u_4=u_3+\rho_2 u_4.$$

 Suppose that $M_{\lambda + \epsilon_2-\epsilon_3}$  does not contain a critical character. Then the eigenvalues
 $\{\mu_{k},\mu_{k+1}\ldots\, \mu_m\}$, $k>1$,
 of $A_{\lambda + \epsilon_2-\epsilon_3}$ form a part of a critical connected chain but
 without the critical character $\mu_1$. By Theorem \ref{Theorem: Description of eigenvalues of A_lambda} these eigenvalues  are of multiplicity $1$. Let $v_{1}, v_{2}$ be eigenvectors of $A_{\lambda+ \epsilon_2-\epsilon_3}$. Then we have:
$$
\begin{cases}
E_{23}(u_{1})=a_{1}v_{1}, & \\
E_{23}(u_{2})=a_{2}v_{1}, & \\
E_{23}(u_{3})=a_{3}v_{1}+a_{4}v_{2}, & \\
E_{23}(u_{4})=a_{5}v_{1}+a_{6}v_{2}. &
\end{cases}
$$
Since $u_1, u_2, u_3$ and $u_4$ are linearly independent and their images span at least a two dimensional space we have that  $\dim(\ker(E_{23}\mid_{M_{\lambda}}))\geq 2$ which is impossible by Corollary \ref{cor-ker}.\end{proof}

From Lemma~\ref{Lemma: two jordan cells} we immediately the following.

\begin{corollary}\label{Corollary: unique-with two Jordan cells} Let $\chi$ be a degenerate Gelfand-Tsetlin character such that $\mu_1 = \chi(A)$ and $\lambda  = \chi|_{\mathfrak h}$.
If $\{ \mu_1, \mu_2\}$ is a  $\lambda$-connected  set,  then there exist at most one simple extension of $\chi$ such that both $\mu_1$ and $\mu_2$ have
multiplicity $2$.
\end{corollary}

\begin{proof}
Indeed, any such simple module $M$ will contain a critical character $\chi'$ determined by the condition
$E_{23}(M(\chi))\subset M(\chi')+M(\chi'')$. But, by Theorem~\ref{Theorem: uniqueness for critical characters},
$\chi'$
defines $M$ uniquely.
\end{proof}

\begin{lemma}\label{Lemma: no critical}
Let $M$ be a simple Gelfand-Tsetlin module such that $M$ is singular but has no critical characters. Then
$A$ is diagonalizable on $M$.
\end{lemma}

\begin{proof}
Fix $\lambda\in \Supp M$, then the distinct eigenvalues of
$A_{\lambda}$ form a singular $\lambda$-connected chain. If this chain is critical then $A_{\lambda}$ is diagonalizable since there is no critical eigenvalue, see \ref{Theorem: Description of eigenvalues of A_lambda}(ii)(d).  Suppose that the chain is degenerate. Then by Theorem \ref{Theorem: Description of eigenvalues of A_lambda}(ii)(b) one can  order the distinct eigenvalues of $A_{\lambda}$ in the following way: $\{\mu_1, \mu_2, \ldots, \mu_m\}$, where $\mu_1=-\frac{1}{2}r$, and if the multiplicity of $\mu_i$ equals $1$ then the multiplicity of $\mu_{i+1}$ is also $1$. Suppose that  $M$ has a character $\widetilde{\chi}$ such that $\widetilde{\chi}(A)=\mu_1$ and $\mu_{1}$ has multiplicity $2$. If $\mu_2$ has multiplicity $2$, then by Lemma~\ref{Lemma: two jordan cells} there exists a  critical character $\chi'$ in the Gelfand-Tsetlin support of every simple extension of $\widetilde{\chi}$, and we obtain a contradiction.

Assume now that $\mu_1$ has multiplicity $2$ but $\mu_{2}$ has multiplicity $1$. Consider the weight subspace $M'=M_{\lambda+\epsilon_2-\epsilon_3}$ and $A'=A|_{M'}$. Observe that $M'\neq 0$, since otherwise $\dim (\ker E_{23}|_{M_{\lambda}})\geq 2$ which is a contradiction  by Corollary \ref{cor-ker}. If $A'$ has no critical eigenvalue then neither does $A_{\lambda+\epsilon_2-\epsilon_3 +k(\epsilon_1 - \epsilon_2)}$, for all integer $k$. In this case all these subspaces $M_{\lambda+\epsilon_2-\epsilon_3 +k(\epsilon_1 - \epsilon_2)}$, $k\in \mathbb Z$ can generate only one eigenvector of $A'$ with eigenvalue $\mu_1$ and, hence, produce only multiplicity $1$ eigenvalue $\mu_1$ of $A'$. But this contradicts to the irreducibility of $M$. Therefore, $A'$ must contain a critical eigenvalue giving a contradiction again. Therefore, $A_{\lambda}$ is diagonalizable, which completes the proof.
\end{proof}

The proof of Lemma~\ref{Lemma: no critical} implies also  the following statement.

\begin{corollary}\label{Corollary: jordan-critical}
Let $M$ be a simple Gelfand-Tsetlin module and $\chi$ a Gelfand-Tsetlin character of $M$ associated with $\lambda\in {\mathfrak h}^*$ and  such that $\dim(M(\chi))=2$. Then
$M$ has a critical character $\chi'$ associated with the weight $\lambda+\epsilon_2-\epsilon_3$.
\end{corollary}

\begin{theorem}\label{Theorem: uniqueness for 2-dimensional characters}
Let $M$ be a simple Gelfand-Tsetlin module and $\chi$ be a Gelfand-Tsetlin character such that $\dim(M(\chi))=2$. Then $M$ is the unique simple extension of $\chi$. 
\end{theorem}
\begin{proof}
Let $\lambda = \chi|_{\mathfrak h}$. Since $\dim(M(\chi))=2$, the distinct eigenvalues of $A_{\lambda}$ form a singular $\lambda$-connected chain $\{\mu_1, \ldots, \mu_m\}$, $m\leq \infty$. Moreover, there exists $\mu_{i}$ of multiplicity $2$. We proceed with two  cases.

\medskip
\noindent {\it Case 1. The chain $\{\mu_1, \ldots, \mu_m\}$ is critical}.  By Theorem \ref{Theorem: Description of eigenvalues of A_lambda} all distinct eigenvalues can be ordered in the following way: $\{\mu_1, \mu_2, \ldots\, \mu_m\},$ where $\mu_1+1=\mu_2$, multiplicity of $\mu_1$ is $1$, and if the multiplicity of $\mu_i$ equals $1$ for $i>1$ then the multiplicity of $\mu_{i+1}$ is also $1$. Therefore,
the module $M$ has a critical character $\chi '$ such that $\chi'(A)=\mu_1$. Thus, every simple extension of $\chi$ contains $\chi '$ in its Gelfand-Tsetlin support. Applying Theorem~\ref{Theorem: uniqueness for critical characters} we conclude that $M$ is unique.

\medskip
\noindent {\it Case 2. The chain $\{\mu_1, \ldots, \mu_m\}$ is degenerate}. By Theorem \ref{Theorem: Description of eigenvalues of A_lambda} one can  order the distinct eigenvalues of $A_{\lambda}$ in the following way: $\{\mu_1, \mu_2, \ldots, \mu_m\}$, where $\mu_1=-\frac{r}{2}$, and if the multiplicity of $\mu_i$ equals $1$ then the multiplicity of $\mu_{i+1}$ is also $1$. Therefore  $M$ has a character $\widetilde{\chi}$ such that $\widetilde{\chi}(A)=\mu_1$ and $\mu_{1}$ has multiplicity $2$.   By Corollary~\ref{Corollary: jordan-critical}, $A|_{M_{\lambda+\epsilon_2-\epsilon_3}}$  must contain a critical eigenvalue. Thus, $M$ is unique by Theorem~\ref{Theorem: uniqueness for critical characters}.   This completes the proof.
\end{proof}

\begin{theorem}\label{Theorem: uniqueness for modules without critical characters}
Let $M$ be a simple Gelfand-Tsetlin module such that $M$ is singular but has no critical characters. Then for each character $\chi\in \Supp_{GT}(M)$, $M$ is the unique simple extension of $\chi$ with the property that it has no critical characters. 
\end{theorem}

\begin{proof} Let as usual  $\lambda = \chi|_{\mathfrak h}$ and $A_{\lambda}=A|_{M_{\lambda}}$. It follows from Lemma~\ref{Lemma: no critical} that $A_{\lambda}$ is diagonalizable. We proceed in two steps.

\medskip
\noindent {\it Step I. Suppose  that $\chi$ belongs to a critical connected chain.} As $M$ does not have critical characters, $\chi(A)$ belongs to a critical connected chain
$\mu_{i}=i^{2}-\frac{1}{4}-\frac{1}{2}r$, $1\leq n\leq i\leq m\leq \infty$ for some integers $n$ and $m$.
Then as in Theorem~\ref{Theorem: uniqueness for critical characters}, there exists a basis $\{w_i, n\leq i\leq m\}$ such that

\begin{equation}\label{Equation: eq2}
Aw_{i}=\mu_i w_i ,  Bw_{i}=
\begin{cases}
b_{n}w_{n}+w_{n+1}, & i=n\\
d_{i}w_{i-1}+b_{i}w_{i}+w_{i+1}, & i>n\\
d_{m}w_{m-1}+b_{m}w_{m}, & i=m.
\end{cases}
\end{equation}

Hence, $B_{\lambda}=B$ is determined uniquely by $\chi$ and $A_{\lambda}$. The uniqueness follows from Lemma \ref{Lemma: eigenvalues of A_lambda and chi determine unique extension}. 
 
\medskip
\noindent {\it Step II.  Suppose that $\chi(A)$ belongs to a degenerate chain $\{\mu_{1},\mu_{2},\ldots\}$ where
$\mu_n=n(n-1)-\frac{1}{2}r$, $n\geq 1$} (see Lemma \ref{Lemma: connected chains explicitly}(i)). We proceed considering two cases depending on the connected chain $\{\mu_{1},\mu_{2},\ldots\}$.

\medskip
\noindent
{\it Case 1. The chain $\{\mu_{1},\mu_{2},\ldots\}$ does not contain a degenerate character}, that is the eigenvalues of $A_{\lambda}$ are $\{\mu_{k},\mu_{k+1},\ldots, \mu_s\}$ for some $k>1$ and some $s\leq \infty$.  Applying  relations  from Lemma \ref{Lemma: Relations AB and constants ,p,eta}
one can choose a basis $\{w_k,...,w_s\}$ of $M_{\lambda}$ such that  the matrix of $A_{\lambda}$ is diagonal and the matrix of $B_{\lambda}$
has a tridiagonal form as in the generic case.
 Suppose there exists another simple extension $W$ of $\chi$ satisfying the conditions of the theorem  such that the eigenvalues of $A|_{W_{\lambda}}$ are $\{\mu_{d},\mu_{d+1},\ldots, \mu_t\}$ for some $d>1$ and some $t\leq \infty$.
 If  $d=k$ and $s=t$, then the diagonal matrices $[A|_{M_{\lambda}}]=[A|_{W_{\lambda}}]$ will give  the same matrix of $B_{\lambda}$, hence $M\simeq W$ by Lemma \ref{Lemma: eigenvalues of A_lambda and chi determine unique extension}.

 Suppose $d<k$ (note that in this case $k\leq t$).  Then applying relations  from Lemma \ref{Lemma: Relations AB and constants ,p,eta}
we obtain that $W_{\lambda}$ has a $C({\mathfrak h})$-submodule $U$
 such that the eigenvalues of $A|_{U}$ are $\{\mu_{k},\mu_{k+1},\ldots, \mu_{t}\}$.
 Hence $M$ has a nontrivial proper submodule $\widetilde{M}$ such that the eigenvalues of $A|_{\widetilde{M}_{\lambda}}$ are $\{\mu_{k},\mu_{k+1},\ldots, \mu_{t}\}$. This contradicts the irreducibility of $M$. The case $d>k$
 is treated analogously.

\medskip
\noindent
{\it Case 2. The chain $\{\mu_{1},\mu_{2},\ldots\}$ does contain a degenerate character}, that is the eigenvalues of $A_{\lambda}$ are $\{\mu_{1},\mu_{2},\ldots, \mu_s\}$ for some some $s\leq \infty$. Let  $\chi'$  be the character associated with $\mu_{1}$. Using  relations from Lemma \ref{Lemma: Relations AB and constants ,p,eta} we see that $r^{2}+2ar+4\tau=0$ and there exists a basis $\{w_i, 1\leq i\leq s\}$ of $M_{\lambda}$ such that

\begin{equation}
Aw_{i}=\lambda_i w_i, 1\leq i\leq s \text{ , } Bw_{i}=
\begin{cases}
T w_{1}+w_{2}, & i=1\\
q_{i}w_{i-1}+b_{i}w_{i}+w_{i+1}, & 1<i<s\\
q_{s}w_{s-1}+b_{s}w_{s}, & i=s,
\end{cases}
\end{equation}
where \begin{align*}
q_{2}=&
\frac{1}{3}\left(-\frac{1}{8}r_{1}r^{2}+\frac{1}{2}(r_{1}+\tau_{1})r-\eta\right) \\
q_{i}=&d_{i},\  i>2 
\end{align*}
 and $T$ is a root of the equation
$$x^{2}-(r+a)x+\tau_{1}-\frac{1}{8}r_{1}r^{2}+\frac{1}{2}r\tau_{1}-\eta=0.$$

Let $W$ be another simple extension  of $\chi$.  Then $\mu_{1}$ must be an eigenvalue of
$A|_{W_{\lambda}}$, otherwise $M$ is not simple.  In fact,
 the  quadratic equation on $T$ shows that there might exist two non-isomorphic simple modules with the same degenerate chain $\{\mu_1, \ldots, \mu_s\}$. We will show that only one such module will satisfy the conditions of the theorem.

  The hypothesis that there is no critical characters in all connected chains implies that  $E_{23}(M(\chi'))\subset M(\tilde{\chi})$, where  $\tilde{\chi}$ is a character such that $\tilde{\chi}(A)$  belongs to a critical connected chain (without critical characters by hypothesis). Also
$E_{23}(W(\chi'))\subset W(\tilde{\chi})$. If both  $E_{23}(M(\chi'))$ and $E_{23}(W(\chi'))$ are non-zero then
$M(\tilde{\chi})\neq 0$ and $W(\tilde{\chi})\neq 0$.
We immediately conclude that $M\simeq W$ by Theorem~\ref{Theorem: uniqueness for critical characters}. Suppose
$E_{23}(M(\chi'))=E_{23}(W(\chi'))=0$. Apply the same arguments for $E_{13}$. If both  $E_{13}(M(\chi'))$ and $E_{13}(W(\chi'))$ are non-zero, then $M\simeq W$ as above. On the other hand, if $E_{13}(M(\chi'))=E_{13}(W(\chi'))=0$,
then $M$ and $W$ are simple quotients of the same generalized Verma module generated by a weight vector $v$
such that $E_{23}v=E_{13}v=0$. But such generalized Verma module has a unique simple quotient implying
$M\simeq W$.  Hence, it remains to consider mixed cases. We finish the proof considering three subcases. Recall that a weight module is \emph{pointed} if all its nonzero weight spaces are $1$-dimensional. All other cases are considered analogously.

\medskip
\noindent
{\it Case 2(a). Suppose,    $E_{23}(M(\chi')), E_{13}(M(\chi'))\neq 0$, and 
 $E_{23}(W(\chi')), E_{13}(W(\chi'))  = 0$}. Then $W$ is a quotient of the generalized Verma module $M_1$ generated by an element $v\in M_1$
 such that $E_{23}v=E_{13}v=0$.   Suppose $E_{32}(W(\chi'))\neq 0$ and $E_{31}(W(\chi'))\neq 0$. If
 one of $E_{32}(M(\chi'))$ or $E_{31}(M(\chi'))$ is non-zero then we are done. Suppose   $E_{32}(M(\chi'))=E_{31}(M(\chi'))=0$ and thus $M$ is a quotient of generalized Verma module $M_2$ generated by $v'$ such that
 $E_{32}v'=E_{31}v'=0$. In order not to have critical characters both $M$ and $W$ must be pointed modules, that is all weight spaces have dimension $1$.  Let $\chi'(H_1)=h_1$,  $\chi'(H_2)=h_2$. Comparing the values
 of Casimir elements on $M$ and $W$ we obtain $h_1=-2h_2$.  This condition guarantees that $M$ and $W$ have
 common degenerate character $\chi'$.  Let us find the condition when $W$ is a  pointed module. It is sufficient to check when the following system has a non-trivial solution:
\begin{eqnarray*}
0=E_{23}(\alpha E_{32}v+\beta E_{31}E_{12}v)&=& h_2\alpha v+ \beta E_{21}E_{12}\\
0=E_{13}(\alpha E_{32}v+\beta E_{31}E_{12}v)&=&\alpha E_{12}v + \beta (h_1+h_2+1)E_{12}.
\end{eqnarray*}

\noindent
Assume that $E_{12}v\neq 0$. Then we have
$h_2\alpha -\beta h_1 -\frac{1}{2} \beta r=0$, and $
\alpha  + \beta (h_1+h_2+1)=0$. If $h_1=0$, then $M\simeq W\simeq \mathbb C$. If $h_1\neq 0$ then $h_1=2$, $h_2=-1$ and $\chi'(A)=0$. It follows that $E_{23}v=E_{13}v=E_{21}v=0$ with $h_1=2$, $h_2=-1$ or  $E_{23}v=E_{13}v=E_{12}v=0$ with $h_1=-2$, $h_2=1$.

   Consider first the case $h_1=-2$, and let $\mu\in {\mathfrak h}^*$  be such that $\mu(H_1)=\mu(H_2)=-1$. Then
   $W_{\mu}$ is $1$-dimensional and $E_{12}W_{\mu}=0$. Hence, $W_{\mu}$ is a Gelfand-Tsetlin subspace and
   $A|_{W_{\mu}}=-{\rm Id}$. But, this is a critical value and, thus, $W$ contains critical characters, which is a contradiction.

Suppose now that $h_1=2$ and consider $\mu\in {\mathfrak h}^*$   such that $\mu(H_1)=1$, $\mu(H_2)=-2$. Then
   $W_{\mu}$ is $1$-dimensional and $E_{21}W_{\mu}=0$. Hence, $W_{\mu}$ is a Gelfand-Tsetlin subspace and
   $A|_{W_{\mu}}=0$. Again, this is a critical value which is a contradiction.

 Suppose now that $E_{32}(W(\chi'))= 0$. Then $E_{12}(W(\chi'))= 0$ and $W$ is a highest weight module of highest weight $\chi'|_H$ and $\chi'(H_2)=h_2=0$. Since $\chi'$ is degenerate we have $h_1=0$ or $h_1=-2$. In the case
 $h_1=0$ we obtain $M\simeq W\simeq \mathbb C$. Now suppose $h_1=-2$ and $A|_{W(\chi')}=-2 {\rm Id}$. This highest weight module has no critical characters.  Since $E_{31}W(\chi')\neq 0$ we have $E_{31}M(\chi')= 0$, otherwise $M\simeq W$ as before. Since $A|_{M(\chi')}=-2 {\rm Id}$ and all characters have multiplicity $1$, we have $E_{12}M(\chi')=0$.
 Thus, in addition we have
 $E_{32}M(\chi')=0$.  Consider a weight $\mu\in {\mathfrak h}^*$ such that $\mu(H_1)=-1$, $\mu(H_2)=1$. The subspace
 $M_{\mu}$ is $1$-dimensional. In fact, this is  a critical Gelfand-Tsetlin subspace, since $A|_{M_{\mu}}=-{\rm Id}$. Hence, $M$ does not satisfy the conditions of the theorem and $W$  again is a unique required module.

\medskip
\noindent
{\it Case 2(b). Suppose    $E_{23}(M(\chi')), E_{13}(W(\chi')) \neq 0$, and 
 $E_{23}(W(\chi')), E_{13}(M(\chi'))  = 0$.}  Now we act by $E_{32}$ and $E_{31}$. Suppose first that 
  $E_{32}(M(\chi')) = 0$ and  $E_{31}(W(\chi')) = 0$.  Hence, $M$ contains a non-zero vector $v$ such that
  $E_{13}v=E_{32}v=E_{12}v=0$.  On the other hand,  $W$ contains a non-zero vector $v'$ such that
  $E_{31}v'=E_{23}v'=E_{21}v'=0$. Moreover, $H_1v=H_1v'=0$. But, since $A$ is diagonalizable on $M$ and
  on $W$, we have $E_{21}v=E_{12}v'=0$. Thus $M\simeq W\simeq \mathbb C$.

  Suppose now  $E_{31}(M(\chi')) = 0$ and  $E_{32}(W(\chi')) = 0$. Hence, $M$ contains a non-zero vector $v$ such that
  $E_{13}v=E_{31}v=0$, and  $W$ contains a non-zero vector $v'$ such that
  $E_{32}v'=E_{23}v'=0$. We obtain that $H_1 v=H_1 v'=0$ and $H_2 v=H_2 v'=0$. Moreover, $Av=Av'=0$. Since $A$ is diagonalizable on $M$ and $W$ we have $E_{12}v=E_{12}v'=0$ and  $E_{21}v=E_{21}v'=0$. Thus $M\simeq W\simeq \mathbb C$.

  Finally, let $E_{31}(M(\chi')) = 0$ and  $E_{32}(M(\chi')) = 0$. Hence,  $M$ contains a non-zero vector $v$ such that
  $E_{13}v=E_{31}v=E_{32}v=0$, implying $E_{12}v=0$. So, either $H_1v=H_2v=0$ and $M\simeq W\simeq \mathbb C$, or $H_1 v=-2$ and $H_2 v=2$. In the latter case $W$ contains a non-zero vector $v'$ such that
  $E_{23}v'=0$, $H_1 v'=-2v'$ and $Av'=-2v'$.  Hence, $E_{12}v'=0$ and $E_{13}v'=0$ since $A$ is diagonalizable.
  We have $c_1v=c_1v'$ implying $H_2v=H_2v'=0$ which is a contradiction.

\medskip
\noindent
{\it Case 2(c). Suppose    $E_{23}(M(\chi'))\neq 0$, and 
 $E_{23}(W(\chi')) , E_{13}(M(\chi')), E_{13}(W(\chi')) =  0$.}
 Now we act by $E_{32}$ and $E_{31}$ on $M(\chi')$ and $W(\chi')$. Without loss of generality we may assume that
  $E_{32}(M(\chi')) = 0$ and  $E_{31}(W(\chi')) = 0$. Therefore $E_{21}W(\chi')=[E_{23}, E_{31}]W(\chi')=0$ and  $AW(\chi')=0$.
  Hence, we  have either $H_1w=H_2 w=0$
  or $H_1w=2w$, $H_2w=-2w$ for any $w\in W(\chi')$. In the first case we obtain $M\simeq W\simeq \mathbb C$.
  Consider the second case.
  Since $AM(\chi')=0$, we must have $E_{21}(M(\chi'))=0$ (otherwise $A$ will not be diagonalizable on $M_{\mu}$),
  where  $E_{21}(M(\chi'))\subset M_{\mu}$.
  Thus   $E_{23}(M(\chi'))=[E_{21}, E_{13}]M(\chi)=0$. Which is a contradiction.
\end{proof}

We next state the main theorem in this section. 

\begin{theorem}\label{Theorem: characterization of irr extensions} 
Let $M$ be a simple Gelfand-Tsetlin $\mathfrak{g}$-module and $\chi \in 
{\rm Supp}_{GT}(M)$. Consider the following conditions: 
\begin{itemize}
\item[(i)] $\chi$ is non-critical  and $\dim M(\chi') \leq 1$ for any  $\chi'\in\Ga^{*}$;
\item[(ii)] $\chi$ is non critical and $\dim(M(\chi))=2$;
\item[(iii)] $\chi$ is critical;
\item[(iv)] $M$ has no critical characters.
\end{itemize}
If any of (ii)--(iv) holds, then $M$ is the unique simple extension of $\chi$. If (i) holds, then $M$ is the unique simple extension of $\chi$ with property (i), but $\chi$ may have another simple extension with two-dimensional Gelfand-Tsetlin multiplicities. 
\end{theorem}

\begin{proof} Let $\chi$ be a non  critical character of $M$. Suppose first that $M$ is generic. Then $M$ is the only simple extension of $\chi$ by Theorem~\ref{Theorem: Uniqueness for generic modules}. Suppose now that $M$ is singular and satisfies the conditions in (i). Assume first that $\chi(A)$ belongs to a critical chain but $\chi(A)$ itself is non-critical.
Then  $M$ is the unique simple extension of $\chi$  by Corollary~\ref{Corollary: non-critical-diag}. Assume now that $\chi(A)$ belongs to a degenerate chain $\{\mu_{1},\mu_{2},\ldots\}$ where
$\mu_n=n(n-1)-\frac{1}{2}r$, $n\geq 1$ (see Lemma \ref{Lemma: connected chains explicitly}(i)). Suppose that $M^{(1)}$ and $M^{(2)}$ are two simple extensions of $\chi$ satisfying the conditions of (i). If $\mu_1$ is not an eigenvalue of $A^{(i)}_{\lambda}:=A|_{M^{(i)}_{\lambda}}$, where $\lambda=\chi|_{\mathfrak{h}}$, then $M^{(1)}\simeq M^{(2)}$ since $A^{(i)}_{\lambda}$ (and thus $M^{(i)}$) is uniquely determined by $\chi(A)$ in this case.  Assume that
$\mu_1$  is  an eigenvalue of both $A^{(i)}_{\lambda}$, $i=1,2$.  Consider the weight $\tilde{\lambda}:=\lambda + \epsilon_2-\epsilon_3$ and the eigenvalues of $A^{(i)}_{\tilde{\lambda}}$. They belong to the same critical chain. If both $A^{(i)}_{\tilde{\lambda}}$ have a common non-critical eigenvalue then $M^{(1)}\simeq M^{(2)}$ by
Corollary~\ref{Corollary: non-critical-diag}. Hence, may assume that $A^{(i)}_{\tilde{\lambda}}$ have distinct eigenvalues. This is only possible if $\dim M^{(i)}_{\lambda}=1$ and $\dim M^{(i)}_{\tilde{\lambda}}\leq 1$, $i=1,2$.  Let $B|_{ M^{(i)}_{\lambda}}=b_i {\rm Id}$. Recall that a weight module is torsion free if all root vectors act injectively on the module. Now consider the following cases.

\medskip

\noindent {\it Case 1.  $\chi(A)-k\chi(H_1)\neq 0$, $b_i-m\chi(H_2)\neq 0$
for $i=1,2$ and $k,m \in {\mathbb Z}$.} In this case both $M^{(1)}$ and $M^{(2)}$ are pointed torsion free modules. Set $\tilde{\lambda}(H_i)=\tilde{h}_i$, $i=1,2$. Then we have $A^{(1)}_{\tilde{\lambda}}=\tilde{\mu}_1  {\rm Id}$,  $A^{(2)}_{\tilde{\lambda}}=\tilde{\mu}_2  {\rm Id}$, $B^{(1)}_{\rho}=\tilde{b}_1  {\rm Id}$ and $B^{(2)}_{\tilde{\lambda}}=\tilde{b}_2  {\rm Id}$, where we can assume that $\tilde{\mu}_1=-\frac{1}{4}-\frac{1}{4}\tilde{h}^{2}_{1}+\tilde{h}_1$ and $\tilde{\mu}_2=\tilde{\mu}_1+1$.  Using the second identity in Lemma \ref{Lemma: Relations AB and constants ,p,eta} we also have 
\begin{equation}\label{Equation: rel between eigen of A and B}
a\tilde{b}_{i}=2\tilde{\mu}_{i}\tilde{b}_{i}+\tilde{b}_{i}^{2}+r_{1}\tilde{\mu}_{i}+\tau_{1}, \ \ i=1,2.
\end{equation}  

Keeping in mind that  $\tau_1$ and $a$ depend of $\tilde{h}_1$ and $\tilde{h}_2$, $r$ depends of $\tilde{h}_1$ and $b_i$, $r_1$ depend of $\tilde{h}_2$,  $(a-2\tilde{\mu}_{1})\tilde{b}_{1}-\tilde{b}_{1}^{2}-r_{1}\tilde{\mu}_{1}-\tau_{1}$ can be express as a polynomial in $\tilde{h}_1$, $\tilde{h}_2$. Let us consider the two-variable polynomial 
 $$
 f_i(x,y) = (a(x,y)-2\mu_i(x))\tilde{b}_{i}(y)-\tilde{b}_{i}^{2}(y)-r_{1}(y)\mu_i(x)-\tau_{1} (x,y),\ \ i=1,2,
 $$ Then   by  (\ref{Equation: rel between eigen of A and B}), $f_i(\tilde{h}_1,\tilde{h}_2)=0$.
An easy calculation shows that 
\begin{eqnarray*}
A^{(1)}_{\tilde{\lambda}+ 2\epsilon_2- 2\epsilon_3}=\tilde{\mu}_{2}(\tilde{h}_1-2) {\rm Id},&
A^{(2)}_{\tilde{\lambda}+ 2\epsilon_2- 2\epsilon_3}=\tilde{\mu}_{1}(\tilde{h}_1-2) {\rm Id},\\
B^{(1)}_{\tilde{\lambda}+ 2\epsilon_2- 2\epsilon_3}=\tilde{b}_{2}(\tilde{h}_2+4) {\rm Id},&
B^{(2)}_{\tilde{\lambda}+ 2\epsilon_2- 2\epsilon_3}=\tilde{b}_{1}(\tilde{h}_2+4) {\rm Id}.
\end{eqnarray*}

Then  $f_i(\tilde{h}_1-2, \tilde{h}_2+4)=0$, $i=1,2$. Similarly, using the operator $E_{12}^2$ we obtain  $f_i(\tilde{h}_1+4,\tilde{h}_2-2) = 0$. Hence,  $f_i(\tilde{h}_1+6,\tilde{h}_2) = 0$. If we repeat this argument again we will obtain $f_i(\tilde{h}_1+12,\tilde{h}_2)=0$ and so on. Hence if $g(x)=f_1(x,\tilde{h}_{2})$, we have shown that $g(t)=0$ implies $g(t+6)=0$, so $g$ has infinitely many roots, thus $g=0$, which is impossible.

\medskip

\noindent {\it Case 2. $\chi(A)-k\chi(H_1) = 0$, $b_i-m\chi(H_2)\neq 0$
for some $k  \in {\mathbb Z}$ and all $i=1,2$, \,  $m \in {\mathbb Z}$.} Without loss of generality we may assume that
  $\chi(A)=0$.   Since $A$ is diagonalizable on both modules, then this immediately implies that both $M^{(1)}$ and $M^{(2)}$ are quotients of generalized Verma modules
(induced from an infinite-dimensional  simple $\mathfrak{sl}(2)$-module $W$) that have the same central character and the same weight support. Since the generalized Harish-Chandra homomorphism defines  $W$ uniquely we conclude that $M^{(1)}\simeq M^{(2)}$.

\medskip

\noindent {\it Case 3. $\chi(A)-k\chi(H_1) = 0$, $b_1-m_1\chi(H_2) = 0$,  $b_2-m_2\chi(H_2) \neq 0$, 
for some $k, m_1  \in {\mathbb Z}$ and all $i=1,2$, \,  $m_2 \in {\mathbb Z}$.} This case is handled in a similar way as Case 2. 

\medskip

\noindent {\it Case 4. $\chi(A)-k\chi(H_1) = 0$ and $b_i=m_i\chi(H_2)$ for some integer $k, m_i$, $i=1,2$.} Therefore, as in Case 2,
both $M^{(1)}$ and $M^{(2)}$ are quotients of the same generalized Verma module. Hence, $m_1=m_2$
and $M^{(1)}\simeq M^{(2)}$. This completes the proof of (i).

\medskip
 If $\chi$ is a non critical character such that $\dim(M(\chi))=2$, then $M$ is a unique simple extension of $\chi$
by Theorem~\ref{Theorem: uniqueness for 2-dimensional characters} implying the result for (ii).
The uniqueness of the extension if (iii) holds follows immediately from Theorem~\ref{Theorem: uniqueness for critical characters}. It remains to consider (iv).
Suppose $M$ is generic. Then the uniqueness of simple extension for any character of $M$ again follows from Theorem~\ref{Theorem: Uniqueness for generic modules}. If $M$ is singular but without critical characters, then we apply 
Theorem~\ref{Theorem: uniqueness for modules without critical characters}.
\end{proof}

\section{Realizations of all simple Gelfand-Tsetlin modules for $\mathfrak{sl}(3)$}\label{Section: Realizations of all simple Gelfand-Tsetlin modules for sl(3)}
In this section we  give an explicit realization of all simple Gelfand-Tsetlin modules of $\mathfrak{sl}(3)$. For this purpose we consider any Gelfand-Tsetlin character $\chi\in\Gamma^{*}$ and construct a Gelfand-Tsetlin module $M$ such that any simple extension of $\chi$ is isomorphic to some subquotient of $M$ (recall that, by  Theorem \ref{Theorem: finiteness-for-gl-n}, the number of non-isomorphic simple extensions is at least one and at most two).

Remark \ref{Remark: correspondence between characters and tableaux} provides  a natural correspondence between the Gelfand-Tsetlin characters and the Gelfand-Tsetlin tableaux. Hence, given a character $\chi$ we can associate a tableau $T(v)$ and the problem of constructing  simple extensions of $\chi$ is reduced to the problem of finding simple modules with tableaux realization containing $T(v)$ as a basis element. 
Recall that any Gelfand-Tsetlin tableau $T(v)$ of height $3$ is either generic ($v_{21}-v_{22}\notin\mathbb{Z}$) or $1$-singular ($v_{21}-v_{22}\in\mathbb{Z}$), and the constructions in \S \ref{Section: Families of Gelfand-Tsetlin modules for gl(n)} allow us to describe an explicit Gelfand-Tsetlin module $V(T(v))$ for any $T(v)$. This, combined with Theorem \ref{Theorem: characterization of irr extensions}, implies that for the desired classification, it is sufficient to describe all simple subquotients  of the modules $V(T(v))$.

\subsection{Structure of generic $\mathfrak{sl}(3)$-modules $V(T(v))$}\label{subsection: Structure of generic sl(3)-modules V(T(v))} In this subsection we consider all possible generic Gelfand-Tsetlin  tableaux $T(v)$ and describe all simple subquotients of the $\mathfrak{sl}(3)$-module $V(T(v))$. The description includes an explicit basis for each simple subquotient, its weight support and its Loewy decomposition.  Since $\mathfrak{g} = \mathfrak{sl}(3)$, the action of $E_{11}+E_{22}+E_{33}$  is zero, thus $w_{31}+w_{32}+w_{33}+3=0$ for any Gelfand-Tsetlin tableaux $T(w)$.  We first rewrite Theorem \ref{Theorem: Generic GT modules} in the case of $\mathfrak{sl}(3)$.

\begin{theorem}\label{Theorem: Generic modules for gl(3)}
 If $T(v)$ is a generic Gelfand-Tsetlin tableau of  height  $3$, then the vector space $V(T(v))$ spanned by the set of tableaux $\mathcal{B}(T(v))$ has a structure of a Gelfand-Tsetlin $\mathfrak{sl}(3)$-module with the action of $\mathfrak{sl}(3)$ on $V(T(v))$ given by the Gelfand-Tsetlin formulas:
\begin{eqnarray*}
E_{12}(T(w))=&  -(w_{21}-w_{11})(w_{22}-w_{11})T(w+\delta^{11}),\\
E_{21}(T(w))=&  T(w-\delta^{11}),\\
E_{32}(T(w))=&  \displaystyle\frac{w_{21}-w_{11}}{w_{21}-w_{22}}T(w-\delta^{21})-\frac{w_{22}-w_{11}}{w_{21}-w_{22}}T(w-\delta^{22}),\\
E_{23}(T(w))=& \displaystyle\frac{(w_{31}-w_{21})(w_{32}-w_{21})(w_{33}-w_{21})}{w_{21}-w_{22}}T(w+\delta^{21})\\
& -\displaystyle\frac{(w_{31}-w_{22})(w_{32}-w_{22})(w_{33}-w_{22})}{w_{21}-w_{22}}T(w+\delta^{22}).
\end{eqnarray*}
\begin{eqnarray*}
E_{13}(T(w))=& \displaystyle\frac{(w_{21}-w_{31})(w_{21}-w_{32})(w_{21}-w_{33})(w_{22}-w_{11})}{w_{21}-w_{22}}T(w+\delta^{21}+\delta^{11})\\
& +\displaystyle\frac{(w_{31}-w_{22})(w_{32}-w_{22})(w_{33}-w_{22})(w_{21}-w_{11})}{w_{21}-w_{22}}T(w+\delta^{22}+\delta^{11}),\\
E_{31}(T(w))=&\displaystyle\frac{1}{w_{21}-w_{22}}T(w-\delta^{21}-\delta^{11})-\frac{1}{w_{21}-w_{22}}T(w-\delta^{22}-\delta^{11}).
\end{eqnarray*}

\begin{eqnarray}
H_{1}(T(w))=& (2w_{11}-(w_{21}+w_{22}+1))T(w),\label{Equation: GT formulas H1}\\
H_{2}(T(w))=& (2(w_{21}+w_{22}+1)-w_{11})T(w),\label{Equation: GT formulas H2}\end{eqnarray}
\end{theorem}

By Theorem \ref{Theorem: Uniqueness for generic modules}, a generic character admits a unique simple extension. In order to describe such simple extension, given a tableau $T(w)\in\mathcal{B}(T(v))$ we will describe explicit basis of tableaux for the simple subquotient of $V(T(v))$ that contains $T(w)$.

By (\ref{Equation: GT formulas H1}) and (\ref{Equation: GT formulas H2}) it is clear that $\mathcal{B}(T(v))$ is an eigenbasis for the action of the generators of the Cartan subalgebra $\mathfrak{h}$. In particular, any subquotient of $V(T(v))$ is a weight module. The following proposition describes explicit bases for the weight subspaces of the subquotiens of $V(T(v))$.

\begin{proposition}\label{Proposition: Basis for weight spaces in generic blocks}
Let $M$ be a Gelfand-Tsetlin module with basis of tableaux $\mathcal{B}_{M}\subset\mathcal{B}(T(v))$ for some generic tableau $T(v)$. If $T(w)\in\mathcal{B}_{M}$ is a tableau of weight $\lambda$, then the weight space $M_{\lambda}$ is spanned by the set of tableaux $\{T(w+(i,-i,0))\ |\ i\in \mathbb{Z}\}\cap\mathcal{B}_{M}$.
\end{proposition}
\begin{proof}
As $\mathcal{B}(T(v))=\mathcal{B}(T(w))$, we just need to characterize tableaux of the form $T(w+(m,n,k))$ in $\mathcal{B}_{M}$ with the same weight $\lambda$ of $T(w)$. By the Gelfand-Tsetlin formulas we have
\begin{equation*}
\begin{split}
H_{1}(T(w+(m,n,k))) & =(2(w_{11}+k)-(w_{21}+m+w_{22}+n+1))T(w+(m,n,k))\\
H_{2}(T(w+(m,n,k))) & = (2(w_{21}+m+w_{22}+n+1)-(w_{11}+k))T(w+(m,n,k))
\end{split}
\end{equation*}
In particular, the weight of $T(w)$ is $$\lambda=(2w_{11}-(w_{21}+w_{22}+1),2(w_{21}+w_{22}+1)-w_{11}).$$
Furthermore, a tableau $T(w+(m,n,k))$ in $\mathcal{B}_{M}$ has weight $\lambda$ if $m,n,k$ satisfy $m+n=0$ and $k=0$.
\end{proof}

In this section we will use Theorem \ref{Theorem: basis for Irr generic modules}  to describe all simple subquotients of the generic $\mathfrak{sl}(3)$-modules $V(T(v))$. Let us recall this result. 

Let $T(v)$ be a generic Gelfand-Tsetlin tableau of height $3$ and $T(w)\in\mathcal{B}(T(v))$ and $\Omega^{+}(T(u))=\{(r,s,t)\ |\ u_{rs}-u_{r-1,t}\in\mathbb{Z}_{\geq 0}\}$.
The complex vector space with basis 
 $\mathcal{N}(T(w))=\{T(w')\in\mathcal{B}(T(v))\ |\ \Omega^{+}(T(w))\subseteq\Omega^{+}(T(w'))\}$
 is a submodule of $V(T(v))$ containing $T(w)$. Moreover, $\mathcal{I}(T(w)):=\{T(w')\in\mathcal{B}(T(v))\ |\ \Omega^{+}(T(w))=\Omega^{+}(T(w'))\}$ is a basis of the unique simple extension of $\chi_w$, where $\chi_{w}\in\Gamma^{*}$ is given by $\chi_{w}(c_{rs})=\gamma_{rs}(w)$.

 By Theorem \ref{Theorem: basis for Irr generic modules}, bases of the subquotients of $V(T(v))$ can be  described by subsets of $ \mathbb{Z}^3$. In order to describe these bases, we introduce new  notation.

\begin{definition}\label{Definition: M(B)}
Let $T(w)$ be a tableau and $\mathcal{B}$ be a subset of ${\mathbb Z}^3$.  Assume that $M$ is a Gelfand-Tsetlin module with basis $\{ T(w+(m,n,k)) \; | \; (m,n,k)\in\mathcal{B} \}$. Then we will denote $M$ by $M(\mathcal{B}, T(w))$, or simply by $M(\mathcal{B})$ if $T(w)$ is fixed. If  $M(\mathcal{B})$ is simple, then we will write $L(\mathcal{B})$ for $M(\mathcal{B})$.
\end{definition}

\begin{example}
With the notation of Definition \ref{Definition: M(B)}, the simple module from Example \ref{Example: explicit basis for irr gen module} can be written as follows:
{\scriptsize $$L\left(
\begin{split}
&m\leq 0,\\
&k\leq m,\\
&n> -2
\end{split}
\right)=M\left(
\begin{split}
&m\leq 0\\
&k\leq m\\
&n> -2
\end{split}
;\ T(v)\right)$$}
where $v=(a,b,c,a,b+2,a)$.
\end{example}

\subsection{Realizations of all simple generic Gelfand-Tsetlin $\mathfrak{sl}(3)$-modules}\label{subsection: Realization of simple generic Gelfand-Tsetlin modules for sl(3)} In this section we  describe all simple objects in every generic block $\mathcal{GT}_{T(v)}$ (see Definition \ref{Definition: blocks} and Remark \ref{Remark: correspondence between characters and tableaux}). Such description will include an explicit tableaux basis of each simple subquotient $M$ in $\mathcal{GT}_{T(v)}$ and  the weight support of $M$. For the weight support we will use Proposition \ref{Proposition: Basis for weight spaces in generic blocks} and the explicit basis. If the weight multiplicities are finite, a picture of the weight support along with the multiplicities  is provided. We also present the  components of the Loewy series of the universal module $V(T(v))$. A rigorous proof based on Theorem \ref{Theorem: basis for Irr generic modules}, Proposition \ref{Proposition: Basis for weight spaces in generic blocks}, and Theorem \ref{Theorem: Loewy decomposition generic case} is given for Case $(G6)$ only, however, for all other cases  the  reasoning is the same.

Until the end of this section we use the following convention. The entries of the Gelfand-Tsetlin tableaux will be integer shifts of some of the complex numbers $a,b,c,x,y,z$. We also assume that if any two of $a,b,c,x,y,z$ appear in the same row or in consecutive rows of a given tableau, then their difference is not integer. 

\begin{remark}\label{Remark: Number of simple subquotients for generic blocks}
By Theorem \ref{Theorem: Uniqueness for generic modules}, any generic character has a unique simple extension. In particular, if $T(v)$ is generic, the number of simple subquotients of $V(T(v))$ is equal to the number of simple modules in $\mathcal{GT}_{T(v)}$. This number depends only on $\Omega(T(v))$ (see Theorem 7.6 in \cite{FGR2}).
\end{remark}


\begin{itemize}
\item[$(G1)$] Consider the following Gelfand-Tsetlin tableau:

\begin{center}

 \hspace{1.5cm}\Stone{$a$}\Stone{$b$}\Stone{$c$}\\[0.2pt]
 $T(v)$=\hspace{0.3cm} \Stone{$x$}\Stone{$y$}\\[0.2pt]
 \hspace{1.3cm}\Stone{$z$}\\
\end{center}

\noindent
The module $V(T(v))$ is simple and, then $\mathcal{GT}_{T(v)}$  has unique (up to isomorphism) simple module, and this module has infinite-dimensional weight multiplicities.

\begin{center}
\begin{tabular}{|>{$}l<{$}|>{$}c<{$}|}\hline
\textsf{Module} &   \hspace{0.3cm}\text{Basis}\\ \hline
L_{1} &
L(\mathbb{Z}^{3})\\ \hline
\end{tabular}
\end{center}

\item[$(G2)$] Let $T(v)$ be the tableau

\begin{center}

 \hspace{1.5cm}\Stone{$a$}\Stone{$b$}\Stone{$c$}\\[0.2pt]
 \hspace{1.4cm} \Stone{$x$}\Stone{$y$}\\[0.2pt]
 \hspace{1.3cm}\Stone{$x$}\\
\end{center}

\noindent
\begin{itemize}
\item[I.] \textbf{simple subquotients}.\\
In this case the module $V(T(v))$ has  $2$ simple subquotients and they have  infinite-dimensional weight multiplicities:
\vspace{0.3cm}
\begin{center}
\begin{tabular}{|>{$}l<{$}|>{$}c<{$}|}\hline
\textsf{Module} &   \hspace{0.3cm}\text{Basis}\\ \hline
L_{1}&L\left( k\leq m
\right)\\\hline
L_{2}&L\left(
m<k
\right)\\\hline
\end{tabular}
\end{center}
\vspace{0.3cm}
\item[II.] \textbf{Loewy series.}\\
$$L_{1} ,\  L_{2}.$$

\end{itemize}


\item[$(G3)$] Consider the tableau:

\begin{center}

 \hspace{1.5cm}\Stone{$a$}\Stone{$b$}\Stone{$c$}\\[0.2pt]
 $T(v)$=\hspace{0.3cm} \Stone{$a$}\Stone{$y$}\\[0.2pt]
 \hspace{1.3cm}\Stone{$z$}\\
\end{center}

\noindent
\begin{itemize}
\item[I.] \textbf{simple subquotients}.\\
In this case the module $V(T(v))$ has  $2$ simple subquotients  and they have infinite-dimensional weight multiplicities: 
\vspace{0.3cm}
\begin{center}
\begin{tabular}{|>{$}l<{$}|>{$}c<{$}|}\hline
\textsf{Module} &   \hspace{0.3cm}\text{Basis}\\ \hline
L_{1}&L\left( m\leq 0
\right)\\\hline
L_{2}&L\left(
0<m
\right)\\\hline
\end{tabular}
\end{center}
\vspace{0.3cm}
\item[II.] \textbf{Loewy series.}
$$L_{1} ,\  L_{2}.$$
\end{itemize}


\item[$(G4)$] Consider the tableau:

\begin{center}

 \hspace{1.5cm}\Stone{$a$}\Stone{$b$}\Stone{$c$}\\[0.2pt]
 $T(v)$=\hspace{0.3cm} \Stone{$a$}\Stone{$y$}\\[0.2pt]
 \hspace{1.3cm}\Stone{$y$}\\
\end{center}
\noindent
\begin{itemize}
\item[I.] \textbf{simple subquotients}.\\
In this case the module $V(T(v))$ has  $4$ simple subquotients. The bases and corresponding weight lattices are given by:\\

\begin{itemize}

\item[(i)] Modules with infinite-dimensional weight multiplicities:
\vspace{0.3cm}
\begin{center}
\begin{tabular}{|>{$}l<{$}|>{$}c<{$}|}\hline
\textsf{Module} &   \hspace{0.3cm}\text{Basis}\\ \hline
L_{1}&L\left( \begin{array}{c}
m\leq 0\\
k\leq n
\end{array}
\right)\\\hline
L_{4}&L\left( \begin{array}{c}
0<m\\
n<k
\end{array}\right)\\\hline
\end{tabular}
\end{center}

\vspace{0.3cm}
\item[(ii)] Modules with unbounded finite weight multiplicities.  
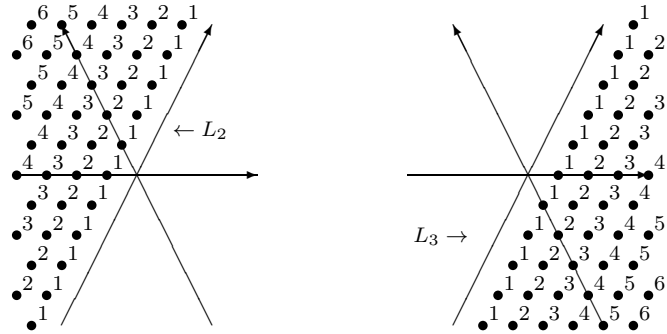
\vspace{0.3cm}
\begin{center}
\begin{tabular}{|>{$}l<{$}|>{$}c<{$}|}\hline
\textsf{Module} &   \hspace{0.3cm}\text{Basis}\\ \hline
L_{2}&L\left( \begin{array}{c}
0<m\\
k\leq n
\end{array}
\right)\\\hline
L_{3}&L\left(\begin{array}{c}
m\leq 0\\
n<k
\end{array}\right)\\\hline
\end{tabular}%
\end{center}

\vspace{0.3cm}
The origin of the picture above corresponds to the $\mathfrak{sl}(3)$-weight associated to the tableau $T(v)$.
\end{itemize}
\vspace{0.3cm}
\item[II.] \textbf{Loewy series.}\\
 $$L_{1} ,\  L_{2}\oplus L_{3} ,\  L_{4}.$$

\end{itemize}


\item[$(G5)$] Let $T(v)$ be the generic tableau 

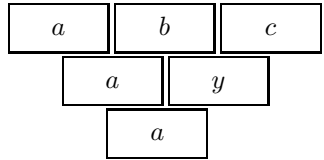
\begin{center}

 \hspace{1.5cm}\Stone{$a$}\Stone{$b$}\Stone{$c$}\\[0.2pt]
 \hspace{1.4cm} \Stone{$a$}\Stone{$y$}\\[0.2pt]
 \hspace{1.3cm}\Stone{$a$}\\
\end{center}

\noindent
\begin{itemize}
\item[I.] \textbf{simple subquotients}.\\

In this case the module $V(T(v))$ has  $4$ simple subquotients. The bases and corresponding weight lattices are given by:\\

\begin{itemize}

\item[(i)] Two modules with infinite-dimensional weight multiplicities:
\vspace{0.3cm}
\begin{center}
\begin{tabular}{|>{$}l<{$}|>{$}c<{$}|}\hline
\textsf{Module} &   \hspace{0.3cm}\text{Basis}\\ \hline
L_{2}&L\left( \begin{array}{c}
m\leq 0\\
m<k
\end{array}\right)\\\hline
L_{3}&L\left( \begin{array}{c}
0<m\\
k\leq m
\end{array}
\right)\\\hline
\end{tabular}
\end{center}

\vspace{0.3cm}

\item[(ii)] Two modules with unbounded finite weight multiplicities. In this case, the origin of the weight lattice corresponds to the $\mathfrak{sl}(3)$-weight associated to the tableau $T(v+\delta^{11})$.
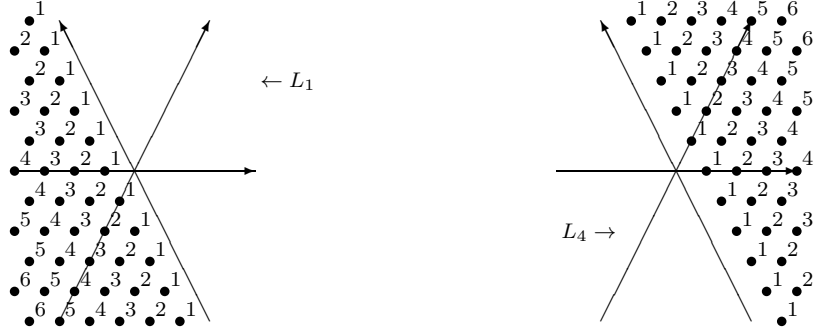
\vspace{0.3cm}
\begin{center}
\begin{tabular}{|>{$}l<{$}|>{$}c<{$}|}\hline
\textsf{Module} &   \hspace{0.3cm}\text{Basis}\\ \hline
L_{1}&L\left( \begin{array}{c}
m\leq 0\\
k\leq m
\end{array}\right)\\\hline
L_{4}&L\left( \begin{array}{c}
0<m\\
m<k
\end{array}
\right)\\\hline
\end{tabular}%
\end{center}
\end{itemize}
\vspace{0.5cm}

\item[II.] \textbf{Loewy series.}\\
 $$L_{1} ,\  L_{2}\oplus L_{3} ,\  L_{4}$$
\end{itemize}


\item[$(G6)$] Consider the tableau:

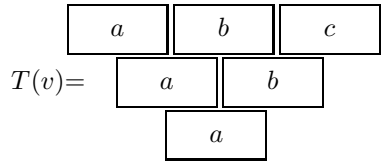
\begin{center}

 \hspace{1.5cm}\Stone{$a$}\Stone{$b$}\Stone{$c$}\\[0.2pt]
 $T(v)$=\hspace{0.3cm}  \Stone{$a$}\Stone{$b$}\\[0.2pt]
 \hspace{1.3cm}\Stone{$a$}\\
\end{center}

\noindent
\begin{itemize}
\item[I.] \textbf{simple subquotients}.\\
In this case the module $V(T(v))$ has  $8$ simple subquotients. The bases and corresponding weight lattices are given by:

\begin{itemize}

\item[(i)] Two modules with infinite-dimensional weight spaces
\vspace{0.3cm}
\begin{center}
\begin{tabular}{|>{$}l<{$}|>{$}c<{$}|}\hline
\textsf{Module} &   \hspace{0.3cm}\text{Basis}\\ \hline
L_{4}&L\left( \begin{array}{c}
0<m\\
n\leq 0\\
k\leq m
\end{array}\right)\\\hline
L_{5}&L\left( \begin{array}{c}
m\leq 0\\
0<n\\
m<k
\end{array}
\right)\\\hline
\end{tabular}
\end{center}

\vspace{0.3cm}
\item[(ii)] Six modules with unbounded finite weight multiplicities.  In this case, the origin of the weight lattice corresponds to the $\mathfrak{sl}(3)$-weight associated to the tableau $T(v+\delta^{11}+\delta^{21})$.
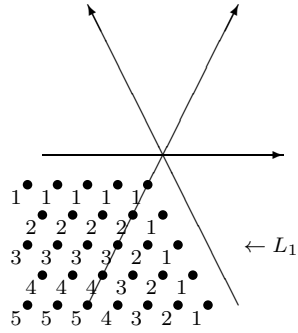
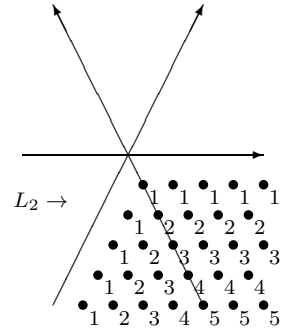
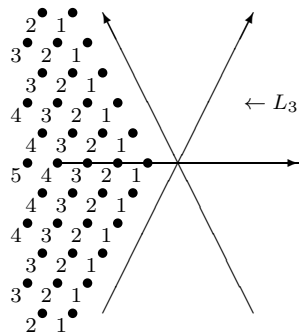
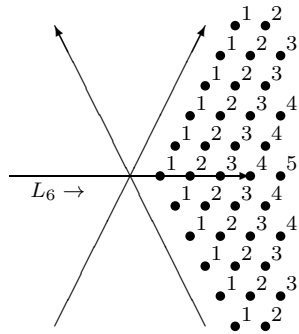
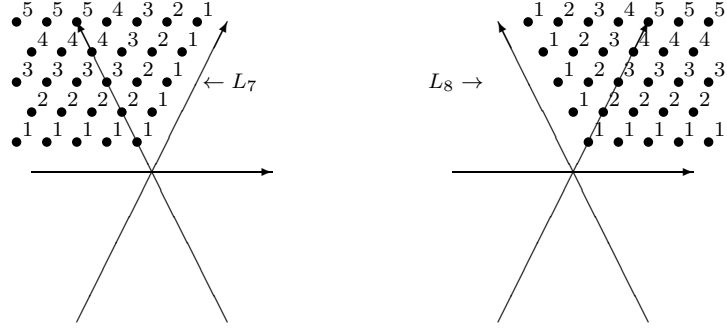
\vspace{0.3cm}
\begin{center}
\begin{tabular}{|>{$}l<{$}|>{$}c<{$}|}\hline
\textsf{Module} &   \hspace{0.3cm}\text{Basis}\\ \hline
L_{1}&L\left( \begin{array}{c}
m\leq 0\\
n\leq 0\\
k\leq m
\end{array}\right)\\\hline
L_{2}&L\left( \begin{array}{c}
m\leq 0\\
n\leq 0\\
m<k
\end{array}
\right)\\\hline
L_{3}&L\left( \begin{array}{c}
m\leq 0\\
0<n\\
k\leq m
\end{array}\right)\\\hline
L_{6}&L\left( \begin{array}{c}
0<m\\
n\leq 0\\
m<k
\end{array}
\right)\\\hline
L_{7}&L\left( \begin{array}{c}
0<m\\
0<n\\
k\leq m
\end{array}\right)\\\hline
L_{8}&L\left( \begin{array}{c}
0<m\\
0<n\\
m<k
\end{array}
\right)\\\hline
\end{tabular}%
\end{center}
\end{itemize}
\vspace{0.5cm}

\item[II.] \textbf{Loewy series.}\\
$$L_{1} ,\  L_{2}\oplus L_{3}\oplus L_{4} ,\  L_{5}\oplus L_{6}\oplus L_{7} ,\  L_{8}.$$
\end{itemize}

\begin{proof} If $M$ denotes the universal module $V(T(v))$, we proceed as follows: first prove that $L_{1}$ is a simple submodule of $M$, then that $M_{1}:=M/L_{1}$ has simple submodules isomorphic to $L_{2},\ L_{3}$ and $L_{4}$,  then  prove that $M_{2}:=M_{1}/(L_{2}\oplus L_{3}\oplus L_{4})$ has simple submodules isomorphic to $L_{5},\ L_{6}$ and $L_{7}$, and finally show that $L_{8}=M_{2}/(L_{5}\oplus L_{6}\oplus L_{7})$ is a simple module.

By Theorem \ref{Theorem: basis for Irr generic modules} we see that, $L_{1}$ (respectively $L_{2}$, $L_{3}$, $L_{4}$ , $L_{5}$, $L_{6}$, $L_{7}$ and $L_{8}$) is a simple subquotient of $V(T(v))$ containing the tableau $T(v)$ (respectively $T(v+(0,0,1))$, $T(v+(0,1,0))$, $T(v+(1,0,1))$ , $T(v+(0,1,1))$, $T(v+(1,0,2))$, $T(v+(1,1,1))$, and $T(v+(1,1,2))$).  To find the the layers of the Loewy series decomposition for $V(T(v))$ we apply Theorem \ref{Theorem: Loewy decomposition generic case} and Remark \ref{Remark: Loewy series for generic modules}. We  describe these layers in four steps.

{\bf Step 1.}  $L_{1}$ is a simple submodule of $M$. By Theorem \ref{Theorem: basis for Irr generic modules}(i) the module with basis $\mathcal{N}(T(v))=\{T(w)\ |\ \Omega^{+}(T(v))\subseteq\Omega^{+}(T(w))\}$ is a submodule of $M$, but $\Omega^{+}(T(v))=\{(3,1,1),(3,2,2),(2,1,1)\}=\Omega(T(v))$, thus $\mathcal{N}(T(v))=\mathcal{I}(T(v))$. Hence, $L_{1}$ is a simple submodule of $M$.

{\bf Step $2$.}  $M_{1}:=M/L_{1}$ has simple submodules isomorphic to $L_{2},\ L_{3}$ and $L_{4}$. For these modules we have that 
\begin{align*}
\Omega^{+}(T(v+(1,0,1)))&=\{(3,2,2),(2,1,1)\},\\
\Omega^{+}(T(v+(0,0,1)))&=\{(3,1,1),(3,2,2)\}\\
\Omega^{+}(T(v+(0,1,0)))&=\{(3,1,1),(2,1,1)\}.
\end{align*} 
Hence, by Theorem \ref{Theorem: basis for Irr generic modules}(i) the modules with bases $\mathcal{I}(T(v+(1,0,1)))\cup\mathcal{I}(T(v))$, $\mathcal{I}(T(v+(0,0,1)))\cup\mathcal{I}(T(v))$ and $\mathcal{I}(T(v+(0,1,0)))\cup\mathcal{I}(T(v))$ are submodules of $M$. Therefore, the modules with bases $\mathcal{I}(T(v+(1,0,1)))$; $\mathcal{I}(T(v+(0,0,1)))$ and $\mathcal{I}(T(v+(0,1,0)))$ are submodules of $M_{1}:=M/L_{1}$ since $L_{1}$ is has basis $\mathcal{I}(T(v))$.

{\bf Step $3$.}  $M_{2}:=M_{1}/(L_{2}\oplus L_{3}\oplus L_{4})$ has simple submodules isomorphic to $L_{5},\ L_{6}$ and $L_{7}$. For these modules we have
\begin{align*}
\Omega^{+}(T(v+(0,1,1)))&=\{(3,1,1)\},\\
\Omega^{+}(T(v+(1,1,1)))&=\{(2,1,1)\}\\
\Omega^{+}(T(v+(1,0,2)))&=\{(3,2,2)\}.
\end{align*}
Hence, by Theorem \ref{Theorem: basis for Irr generic modules}(i) the modules with bases $\mathcal{I}(T(v+(1,0,1)))\cup\mathcal{I}(T(v))$; $\mathcal{I}(T(v+(0,0,1)))\cup\mathcal{I}(T(v))$ and $\mathcal{I}(T(v+(0,1,0)))\cup\mathcal{I}(T(v))$ are submodules of $M$. Therefore, the modules with bases $\mathcal{I}(T(v+(1,0,1)))$; $\mathcal{I}(T(v+(0,0,1)))$ and $\mathcal{I}(T(v+(0,1,0)))$ are submodules of $M_{2}:=M_{1}/(L_{2}\oplus L_{3}\oplus L_{4})$ because $L_{1}$ has a basis $\mathcal{I}(T(v))$.

{\bf Step $4$.}  $M_{3}/(L_{5}\oplus L_{6}\oplus L_{7})\simeq L_{8}$. In fact, $\Omega^{+}(T(v+(1,1,2)))=\emptyset$ so the submodule of $V(T(v))$ generated by $T(v+(1,1,2))$ is $V(T(v))$ and the simple subquotient containing $T(v+(1,1,2))$ has the same basis as $M_{3}/(L_{5}\oplus L_{6}\oplus L_{7})$ so, we have 
$M_{3}/(L_{5}\oplus L_{6}\oplus L_{7})\simeq L_{8}$.

\end{proof}


\item[$(G7)$] Consider the tableau:

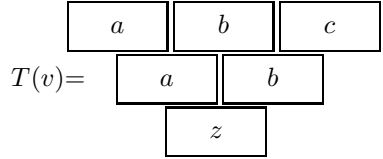
\begin{center}

 \hspace{1.5cm}\Stone{$a$}\Stone{$b$}\Stone{$c$}\\[0.2pt]
 $T(v)$=\hspace{0.3cm} \Stone{$a$}\Stone{$b$}\\[0.2pt]
 \hspace{1.3cm}\Stone{$z$}\\
\end{center}

\noindent
\begin{itemize}
\item[I.] \textbf{simple subquotients}.\\
In this case the module $V(T(v))$ has  $4$ simple subquotients.

\begin{itemize}

\item[(i)] Two modules with infinite-dimensional weight multiplicities:

\vspace{0.3cm}
\begin{center}
\begin{tabular}{|>{$}l<{$}|>{$}c<{$}|}\hline
\textsf{Module} &   \hspace{0.3cm}\text{Basis}\\ \hline
L_{2}&L\left( \begin{array}{c}
0<m\\
n\leq 0
\end{array}
\right)\\\hline
L_{3}&L\left( \begin{array}{c}
m\leq 0\\
0<n
\end{array}\right)\\\hline
\end{tabular}
\end{center}
\vspace{0.3cm}
\item[(ii)] Two modules with unbounded finite weight multiplicities. 
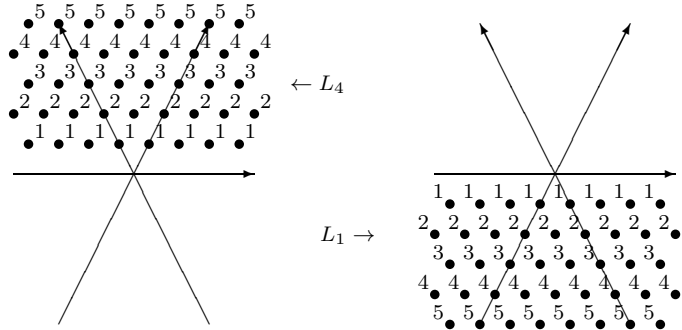
\vspace{0.5cm}
\begin{center}
\begin{tabular}{|>{$}l<{$}|>{$}c<{$}|}\hline
\textsf{Module} &   \hspace{0.3cm}\text{Basis}\\ \hline
L_{1}&L\left( \begin{array}{c}
m\leq 0\\
n\leq 0
\end{array}
\right)\\\hline
L_{4}&L\left( \begin{array}{c}
0<m\\
0<n
\end{array}\right)\\\hline
\end{tabular}%
\end{center}

\end{itemize}

\vspace{0.3cm}
The origin of the picture above corresponds to the $\mathfrak{sl}(3)$-weight associated to the tableau $T(v+\delta^{11}+\delta^{21})$.
  \vspace{0.3cm}

\item[II.] \textbf{Loewy series.}\\
 $$L_{1} ,\  L_{2}\oplus L_{3} ,\  L_{4}.$$
\end{itemize}


 \item[$(G8)$]  Set $t\in\mathbb{Z}_{>0}$ and consider the following Gelfand-Tsetlin tableau:

\begin{center}

 \hspace{1.5cm}\Stone{$a$}\Stone{$a-t$}\Stone{$c$}\\[0.2pt]
 $T(v)$=\hspace{0.3cm} \Stone{$a$}\Stone{$y$}\\[0.2pt]
 \hspace{1.3cm}\Stone{$z$}\\
\end{center}
\noindent
\begin{itemize}
\item[I.] \textbf{simple subquotients}.\\
In this case the module $V(T(v))$ has  $3$ simple subquotients. 

\begin{itemize}
\item[(i)] Two modules with infinite-dimensional weight multiplicities:
\vspace{0.3cm}
\begin{center}
\begin{tabular}{|>{$}l<{$}|>{$}c<{$}|}\hline
\textsf{Module} &   \hspace{0.3cm}\text{Basis}\\ \hline
L_{1}&L\left( \begin{array}{c}
m\leq -t
\end{array}
\right)\\\hline
L_{3}&L\left( \begin{array}{c}
0<m
\end{array}\right)\\\hline
\end{tabular}
\end{center}

\vspace{0.3cm}
\item[(ii)] A cuspidal module with $t$-dimensional weight spaces:
\vspace{0.3cm}
\begin{center}
\begin{tabular}{|>{$}l<{$}|>{$}c<{$}|}\hline
\textsf{Module} &   \hspace{0.3cm}\text{Basis}\\ \hline
L_{2} &
L\left( \begin{array}{c}
-t<m\leq 0
\end{array}\right)\\ \hline
\end{tabular}
\end{center}
\end{itemize}
\vspace{0.3cm}
\item[II.] \textbf{Loewy series.}\\
 $$L_{1} ,\  L_{2} ,\  L_{3}.$$
\end{itemize}


 \item[$(G9)$] For each $t\in\mathbb{Z}_{> 0}$, let consider the following generic tableau:

\begin{center}

 \hspace{1.5cm}\Stone{$a$}\Stone{$a-t$}\Stone{$c$}\\[0.2pt]
 $T(v)$=\hspace{0.3cm} \Stone{$a$}\Stone{$y$}\\[0.2pt]
 \hspace{1.3cm}\Stone{$a$}\\
\end{center}

\noindent
\begin{itemize}
\item[I.] \textbf{simple subquotients}.\\
In this case the module $V(T(v))$ has  $6$ simple subquotients. The bases and corresponding weight lattices are given by:

\begin{itemize}

\item[(i)] Two modules with infinite-dimensional weight multiplicities:
\vspace{0.3cm}
\begin{center}
\begin{tabular}{|>{$}l<{$}|>{$}c<{$}|}\hline
\textsf{Module} &   \hspace{0.3cm}\text{Basis}\\ \hline
L_{2}&L\left( \begin{array}{c}
m\leq -t\\
m<k
\end{array}
\right)\\\hline
L_{5}&L\left( \begin{array}{c}
0<m\\
k\leq m
\end{array}\right)\\\hline
\end{tabular}
\end{center}
 \vspace{0.3cm}
 \item[(ii)] Two modules with unbounded finite weight multiplicities. 
 
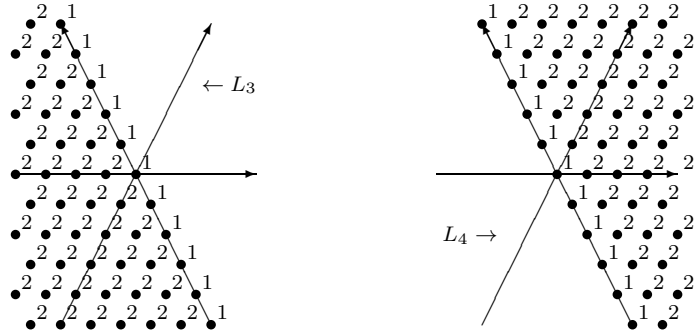
\vspace{0.5cm}

 \begin{center}
\begin{tabular}{|>{$}l<{$}|>{$}c<{$}|}\hline
\textsf{Module} &   \hspace{0.3cm}\text{Basis}\\ \hline
L_{1}&L\left( \begin{array}{c}
m\leq -t\\
k\leq m
\end{array}
\right)\\\hline
L_{6}&L\left( \begin{array}{c}
0<m\\
m<k
\end{array}\right)\\\hline
\end{tabular}%
\end{center}

\vspace{0.3cm}

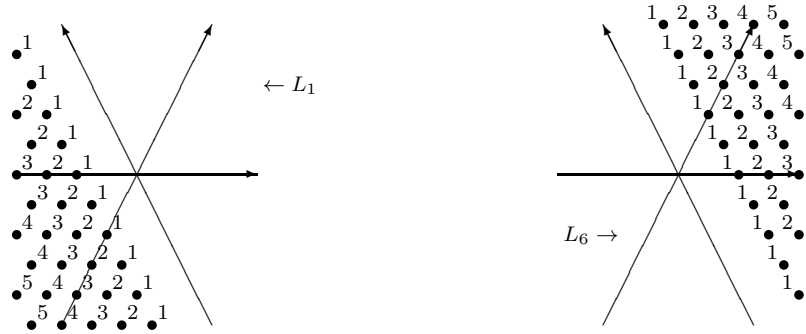
\item[(iii)] Two modules with weight multiplicities bounded by $t$. 
\vspace{0.3cm}
 \begin{center}
\begin{tabular}{|>{$}l<{$}|>{$}c<{$}|}\hline
\textsf{Module} &   \hspace{0.3cm}\text{Basis}\\ \hline
L_{3}&L\left( \begin{array}{c}
-t <m\leq 0\\
k\leq m
\end{array}
\right)\\\hline
L_{4}&L\left( \begin{array}{c}
-t <m\leq 0\\
m<k
\end{array}\right)\\\hline
\end{tabular}%
\end{center}
\end{itemize}

\vspace{0.3cm}
 The pictures above correspond to the case $t=2$, and the origin is the $\mathfrak{sl}(3)$-weight associated to the tableau $T(v-2\delta^{21})$. \vspace{0.3cm}
\noindent
\item[II.] \textbf{Loewy series.}\\
$$L_{1} ,\  L_{2}\oplus L_{3} ,\  L_{4}\oplus L_{5} ,\  L_{6}.$$
\end{itemize}


 \item[$(G10)$] For each $t\in\mathbb{Z}_{> 0}$, let $T(v)$ be the following Gelfand-Tsetlin tableau:

\begin{center}

 \hspace{1.5cm}\Stone{$a$}\Stone{$a-t$}\Stone{$c$}\\[0.2pt]
 \hspace{1.4cm} \Stone{$a$}\Stone{$y$}\\[0.2pt]
 \hspace{1.3cm}\Stone{$y$}\\
\end{center}
\noindent
\begin{itemize}
\item[I.] \textbf{simple subquotients}.\\
In this case the module $V(T(v))$ has  $6$ simple subquotients. The bases and corresponding weight lattices are given by:

\begin{itemize}

\item[(i)] Two modules with infinite-dimensional weight multiplicities:
\vspace{0.3cm}
 \begin{center}
\begin{tabular}{|>{$}l<{$}|>{$}c<{$}|}\hline
\textsf{Module} &   \hspace{0.3cm}\text{Basis}\\ \hline
L_{1}&L\left( \begin{array}{c}
m\leq -t\\
k\leq n
\end{array}
\right)\\\hline
L_{6}&L\left( \begin{array}{c}
0<m\\
n<k
\end{array}\right)\\\hline
\end{tabular}
\end{center}

\vspace{0.3cm}
\item[(ii)] Two modules with unbounded finite weight multiplicities. 
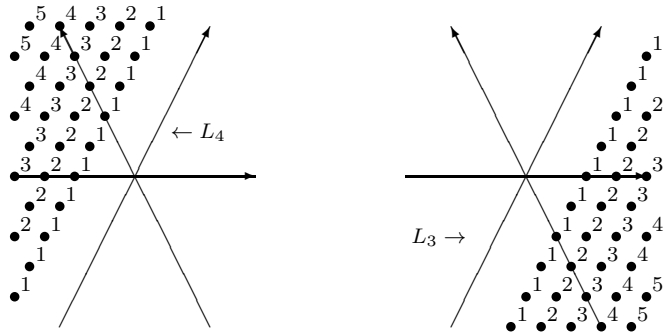
\vspace{0.3cm}

 \begin{center}
\begin{tabular}{|>{$}l<{$}|>{$}c<{$}|}\hline
\textsf{Module} &   \hspace{0.3cm}\text{Basis}\\ \hline
L_{3}&L\left( \begin{array}{c}
m\leq -t\\
n<k
\end{array}
\right)\\\hline
L_{4}&L\left( \begin{array}{c}
0<m\\
k\leq n
\end{array}\right)\\\hline
\end{tabular}%
\end{center}

\vspace{0.3cm}

\item[(iii)] Two modules with finite weight multiplicities  bounded by $t$.
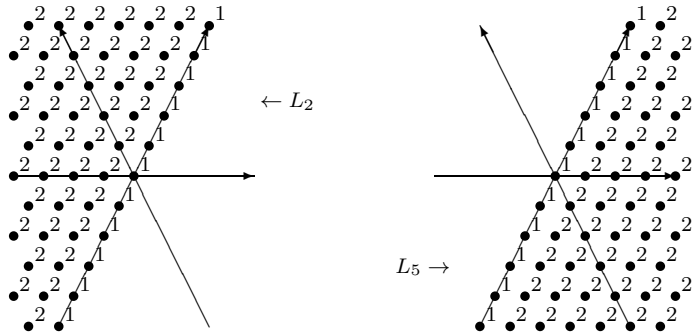
\vspace{0.3cm}
 \begin{center}
\begin{tabular}{|>{$}l<{$}|>{$}c<{$}|}\hline
\textsf{Module} &   \hspace{0.3cm}\text{Basis}\\ \hline
L_{2}&L\left( \begin{array}{c}
-t <m\leq 0\\
k\leq n
\end{array}
\right)\\\hline
L_{5}&L\left( \begin{array}{c}
-t <m\leq 0\\
n<k
\end{array}\right)\\\hline
\end{tabular}%
\end{center}
\end{itemize}

\vspace{0.3cm}
 The pictures above correspond to the case $t=2$, and the origin is the $\mathfrak{sl}(3)$-weight associated to the tableau $T(v-\delta^{22})$. 
\vspace{0.3cm}
\item[II.] \textbf{Loewy series.}\\

$$L_{1} ,\  L_{2}\oplus L_{3} ,\  L_{4}\oplus L_{5} ,\  L_{6}.$$
\end{itemize}


 \item[$(G11)$] For any $t\in\mathbb{Z}_{> 0}$, let $T(v)$ be the tableau:

\begin{center}

 \hspace{1.5cm}\Stone{$a$}\Stone{$a-t$}\Stone{$c$}\\[0.2pt]
 \hspace{1.4cm} \Stone{$a$}\Stone{$c$}\\[0.2pt]
 \hspace{1.3cm}\Stone{$a$}\\
\end{center}
\noindent
\begin{itemize}
\item[I.] \textbf{simple subquotients}.\\
In this case the module $V(T(v))$ has  $12$ simple subquotients.  We provide pictures of the weight lattice corresponding to the case $t=2$, and with origin at the $\mathfrak{sl}(3)$-weight associated to the tableau $T(v)$. 
\begin{itemize}

\item[(i)] Two modules with infinite-dimensional weight multiplicities:
\vspace{0.3cm}
 \begin{center}
\begin{tabular}{|>{$}l<{$}|>{$}c<{$}|}\hline
\textsf{Module} &   \hspace{0.3cm}\text{Basis}\\ \hline
L_{5}&L\left( \begin{array}{c}
m\leq -t\\
0<n\\
m<k
\end{array}
\right)\\\hline
L_{8}&L\left( \begin{array}{c}
0<m\\
n\leq 0\\
k\leq m
\end{array}\right)\\\hline
\end{tabular}
\end{center}
\vspace{0.3cm}

\item[(ii)] Six modules with unbounded finite weight multiplicities. 
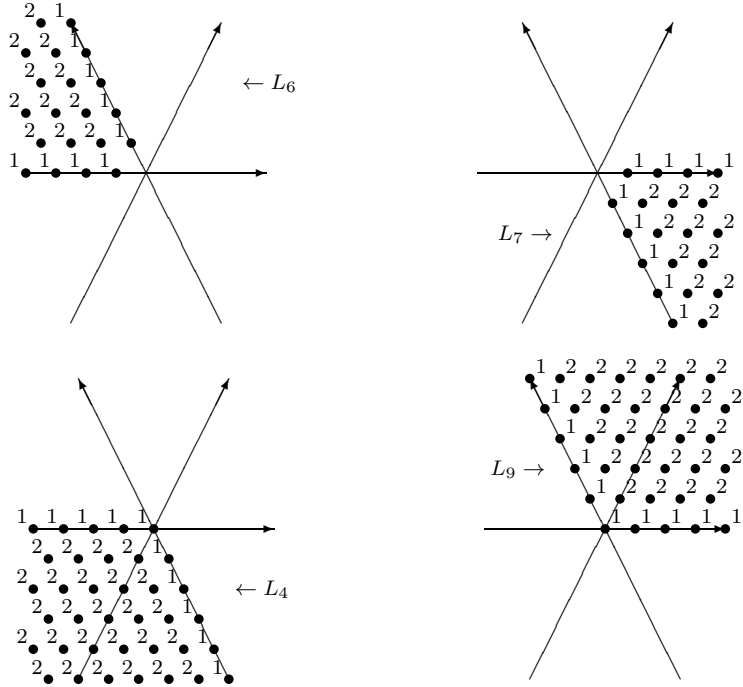
\vspace{0.3cm}
 \begin{center}
\begin{tabular}{|>{$}l<{$}|>{$}c<{$}|}\hline
\textsf{Module} &   \hspace{0.3cm}\text{Basis}\\ \hline
L_{1}&L\left( \begin{array}{c}
m\leq -t\\
n\leq 0\\
k\leq m
\end{array}
\right)\\\hline
L_{2}&L\left( \begin{array}{c}
m\leq -t\\
n\leq 0\\
m<k
\end{array}\right)\\\hline
L_{3}&L\left( \begin{array}{c}
m\leq -t\\
0<n\\
k\leq m
\end{array}
\right)\\\hline
L_{10}&L\left( \begin{array}{c}
0<m\\
n\leq 0\\
m<k
\end{array}\right)\\\hline
L_{11}&L\left( \begin{array}{c}
0<m\\
0<n\\
k\leq m
\end{array}
\right)\\\hline
L_{12}&L\left( \begin{array}{c}
0<m\\
0<n\\
m<k
\end{array}\right)\\\hline
\end{tabular}%
\end{center}

\vspace{0.3cm}
\item[(iii)] Four modules with finite weight multiplicities bounded by $t$. 
\vspace{0.3cm}

\begin{center}
\begin{tabular}{|>{$}l<{$}|>{$}c<{$}|}\hline
\textsf{Module} &   \hspace{0.3cm}\text{Basis}\\ \hline
L_{4}&L\left( \begin{array}{c}
-t <m\leq 0\\
n\leq 0\\
k\leq m
\end{array}
\right)\\\hline
L_{6}&L\left( \begin{array}{c}
-t<m\leq 0\\
0<n\\
k\leq m
\end{array}\right)\\\hline
L_{7}&L\left( \begin{array}{c}
-t <m\leq 0\\
n\leq 0\\
m<k
\end{array}
\right)\\\hline
L_{9}&L\left( \begin{array}{c}
-t <m\leq 0\\
0<n\\
m<k
\end{array}\right)\\\hline
\end{tabular}%
\end{center}

\end{itemize}

\vspace{0.3cm}
\item[II.] \textbf{Loewy series.}\\
$$L_{1} ,\  L_{2}\oplus L_{3}\oplus L_{4} ,\  L_{5}\oplus L_{6}\oplus L_{7}\oplus L_{8} ,\  L_{9}\oplus L_{10}\oplus L_{11} ,\  L_{12}.$$
\end{itemize}


\item[$(G12)$] Consider $t\in\mathbb{Z}_{> 0}$, and $T(v)$ to be the Gelfand-Tsetlin tableau:

\begin{center}

 \hspace{1.5cm}\Stone{$a$}\Stone{$b$}\Stone{$b-t$}\\[0.2pt]
 $T(v)$=\hspace{0.3cm} \Stone{$a$}\Stone{$b$}\\[0.2pt]
 \hspace{1.3cm}\Stone{$a$}\\
\end{center}

\noindent
\begin{itemize}
\item[I.] \textbf{simple subquotients}.\\
In this case the module $V(T(v))$ has  $12$ simple subquotients.  We provide pictures of the weight lattice corresponding to the case $t=2$, and with origin at the $\mathfrak{sl}(3)$-weight associated to the tableau $T(v+\delta^{11})$.
\begin{itemize}

\item[(i)] Two modules with infinite-dimensional weight multiplicities:
\vspace{0.3cm}
\begin{center}
\begin{tabular}{|>{$}l<{$}|>{$}c<{$}|}\hline
\textsf{Module} &   \hspace{0.3cm}\text{Basis}\\ \hline
L_{4} &
L\left( \begin{array}{c}
0<m\\
n\leq -t\\
k\leq m
\end{array}\right)\\ \hline
L_{9} &
L\left( \begin{array}{c}
m\leq 0\\
0<n\\
m<k
\end{array}\right)\\ \hline
\end{tabular}
\end{center}

\vspace{0.3cm}
\item[(ii)] Six modules with unbounded weight multiplicities. 
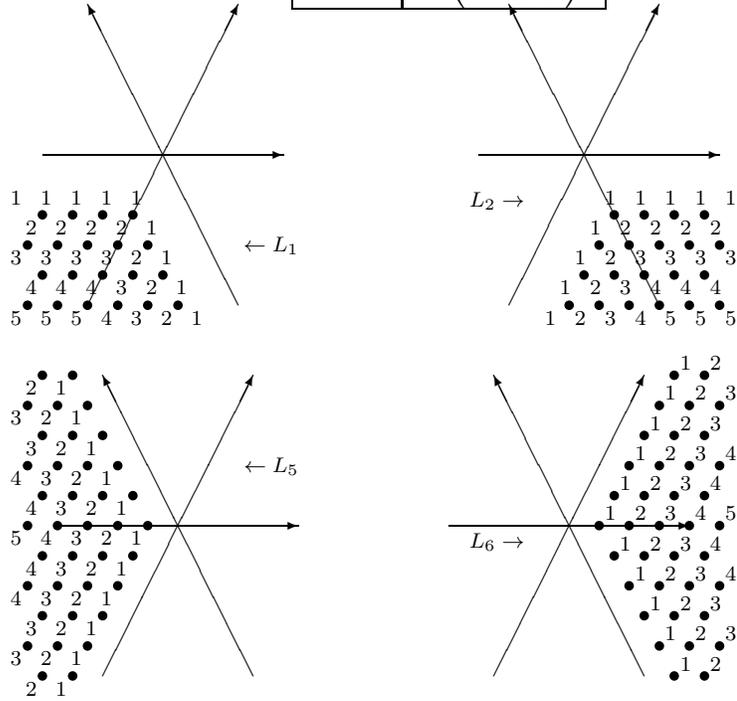
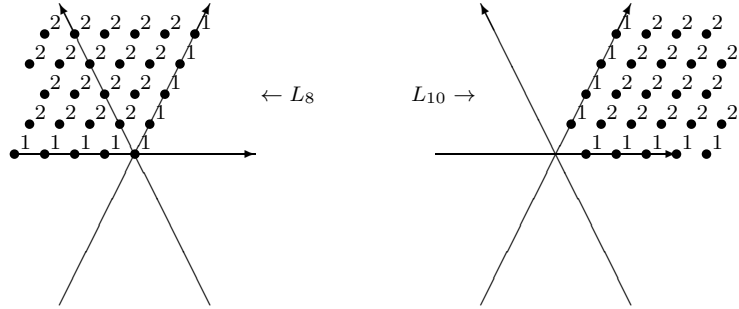
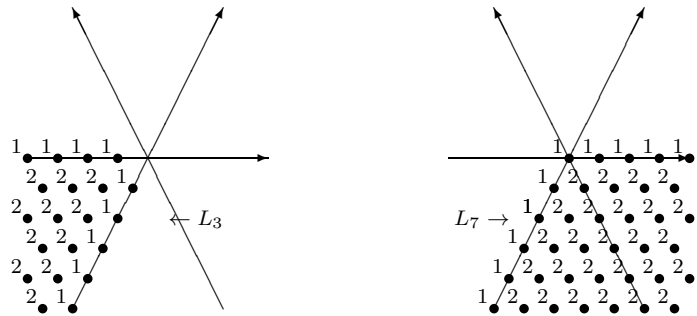
\begin{center}
\vspace{0.3cm}
\begin{tabular}{|>{$}l<{$}|>{$}c<{$}|}\hline
\textsf{Module} &   \hspace{0.3cm}\text{Basis}\\ \hline
L_{1}&L\left( \begin{array}{c}
m\leq 0\\
n\leq -t\\
k\leq m
\end{array}
\right)\\\hline
L_{2}&L\left( \begin{array}{c}
m\leq 0\\
n\leq -t\\
m<k
\end{array}\right)\\\hline
L_{5}&L\left( \begin{array}{c}
m\leq 0\\ 
0<n\\
k\leq m
\end{array}
\right)\\\hline
L_{6}&L\left( \begin{array}{c}
0<m\\
n\leq -t\\
m<k
\end{array}\right)\\\hline
L_{11}&L\left( \begin{array}{c}
0<m\\
0<n\\
k\leq m
\end{array}
\right)\\\hline
L_{12}&L\left( \begin{array}{c}
0<m\\
0<n\\
m<k
\end{array}\right)\\\hline
\end{tabular}%
\end{center}

\vspace{0.3cm}

\item[(iii)] Four modules with finite weight multiplicities bounded by $t$. \vspace{0.3cm}
\begin{center}
\begin{tabular}{|>{$}l<{$}|>{$}c<{$}|}\hline
\textsf{Module} &   \hspace{0.3cm}\text{Basis}\\ \hline
L_{3}&L\left( \begin{array}{c}
-t<n\leq 0\\
m\leq 0\\
k\leq m
\end{array}
\right)\\\hline
L_{7}&L\left( \begin{array}{c}
m\leq 0\\
-t<n\leq 0\\
m<k
\end{array}\right)\\\hline
L_{8}&L\left( \begin{array}{c}
0<m\\
-t<n\leq 0\\
k\leq m
\end{array}
\right)\\\hline
L_{10}&L\left( \begin{array}{c}
0<m\\
-t<n\leq 0\\
m<k
\end{array}\right)\\\hline
\end{tabular}%
\end{center}
\end{itemize}
\vspace{0.3cm}
\item[II.] \textbf{Loewy series.}\\

$$L_{1} ,\  L_{2}\oplus L_{3}\oplus L_{4} ,\  L_{5}\oplus L_{6}\oplus L_{7}\oplus L_{8} ,\  L_{9}\oplus L_{10}\oplus L_{11} ,\  L_{12}.$$
\end{itemize}


 \item[$(G13)$] Consider $t\in\mathbb{Z}_{> 0}$. Set $T(v)$ to be the following Gelfand-Tsetlin tableau:

\begin{center}

 \hspace{1.5cm}\Stone{$a$}\Stone{$a-t$}\Stone{$c$}\\[0.2pt]
 \hspace{1.4cm} \Stone{$a$}\Stone{$c$}\\[0.2pt]
 \hspace{1.3cm}\Stone{$z$}\\
\end{center}
\noindent
\begin{itemize}
\item[I.] \textbf{simple subquotients}.\\
In this case the module $V(T(v))$ has  $6$ simple subquotients.  We provide pictures of the weight lattice corresponding to the case $t=2$, and with origin at the $\mathfrak{sl}(3)$-weight associated to the tableau $T(v+\delta^{11})$.
\begin{itemize}

\item[(i)] Two modules with infinite-dimensional weight multiplicities:
\vspace{0.3cm}
\begin{center}
\begin{tabular}{|>{$}l<{$}|>{$}c<{$}|}\hline
\textsf{Module} &   \hspace{0.3cm}\text{Basis}\\ \hline
L_{3} &
L\left( \begin{array}{c}
m\leq -t\\
0<n
\end{array}\right)\\ \hline
L_{4} &
L\left( \begin{array}{c}
0<m\\
n\leq 0
\end{array}\right)\\ \hline
\end{tabular}
\end{center}
\vspace{0.3cm}
\item[(ii)] Two modules with unbounded finite weight multiplicities:
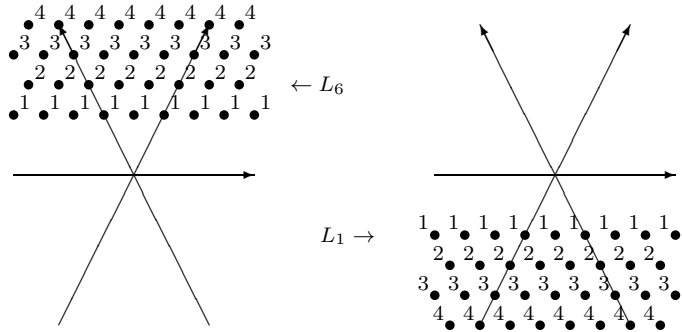
\vspace{0.3cm}
\begin{center}
\begin{tabular}{|>{$}l<{$}|>{$}c<{$}|}\hline
\textsf{Module} &   \hspace{0.3cm}\text{Basis}\\ \hline
L_{1} &
L\left( \begin{array}{c}
m\leq -t\\
n\leq 0
\end{array}\right)\\ \hline
L_{6} &
L\left( \begin{array}{c}
0<m\\
0<n
\end{array}\right)\\ \hline
\end{tabular}%
\end{center}
\vspace{0.3cm}

\item[(iii)] Two modules with finite weight multiplicities bounded by $t$.  

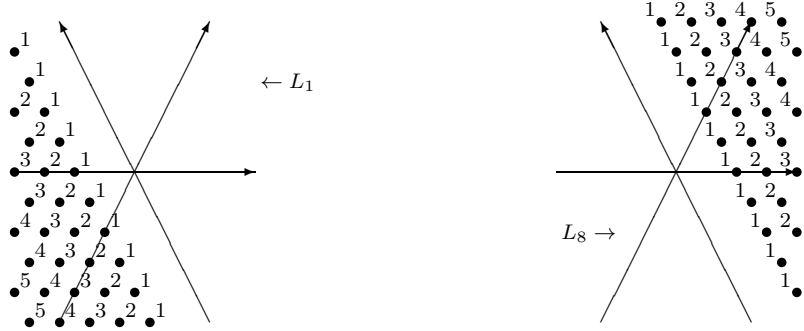
\vspace{0.3cm}
\begin{center}
\begin{tabular}{|>{$}l<{$}|>{$}c<{$}|}\hline
\textsf{Module} &   \hspace{0.3cm}\text{Basis}\\ \hline
L_{2} &
L\left( \begin{array}{c}
-t <m\leq 0\\
n\leq 0
\end{array}\right)\\ \hline
L_{5} &
L\left( \begin{array}{c}
-t <m\leq 0\\
0<n
\end{array}\right)\\ \hline
\end{tabular}%
\end{center}

\end{itemize}
\vspace{0.3cm}
\item[II.] \textbf{Loewy series.}\\
$$L_{1} ,\  L_{2}\oplus L_{3} ,\  L_{4}\oplus L_{5} ,\  L_{6}.$$

\end{itemize}


 \item[$(G14)$] Set $t,s\in\mathbb{Z}_{> 0}$ with $t<s$ and let $T(v)$ be the following Gelfand-Tsetlin tableau:

\begin{center}

 \hspace{1.5cm}\Stone{$a$}\Stone{$a-t$}\Stone{$a-s$}\\[0.2pt]
 $T(v)$=\hspace{0.3cm} \Stone{$a$}\Stone{$y$}\\[0.2pt]
 \hspace{1.3cm}\Stone{$a$}\\
\end{center}
\noindent
\begin{itemize}
\item[I.] \textbf{simple subquotients}.\\
In this case the module $V(T(v))$ has  $8$ simple subquotients.  We provide pictures of the weight lattice corresponding to the case $t=1$, $s=2$, and with origin at the $\mathfrak{sl}(3)$-weight associated to the tableau $T(v)$. 

\begin{itemize}

\item[(i)] Two modules with infinite-dimensional weight multiplicities:
\vspace{0.3cm}
\begin{center}
\begin{tabular}{|>{$}l<{$}|>{$}c<{$}|}\hline
\textsf{Module} &   \hspace{0.3cm}\text{Basis}\\ \hline
L_{3} &
L\left( \begin{array}{c}
m\leq -s\\
m<k
\end{array}\right)\\ \hline
L_{6} &
L\left( \begin{array}{c}
0<m\\
k\leq m
\end{array}\right)\\ \hline
\end{tabular}
\end{center}

\vspace{0.3cm}
\item[(ii)] Two modules with unbounded finite weight multiplicities:
\vspace{0.3cm}
\begin{center}
\begin{tabular}{|>{$}l<{$}|>{$}c<{$}|}\hline
\textsf{Module} &   \hspace{0.3cm}\text{Basis}\\ \hline
L_{1} &
L\left( \begin{array}{c}
m\leq -s\\
k\leq m
\end{array}\right)\\ \hline
L_{8} &
L\left( \begin{array}{c}
0<m\\
m<k
\end{array}\right)\\ \hline
\end{tabular}%
\end{center}

\vspace{0.3cm}

\item[(iii)] Four modules with finite weight multiplicities. The modules $L_4$ and $L_7$ have multiplicities bounded by $t$, and $L_2$, $L_5$ have multiplicities bounded by $s-t$. 
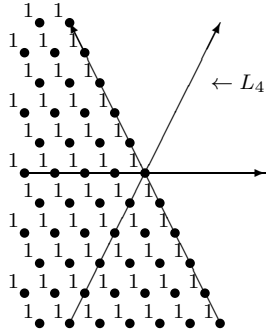
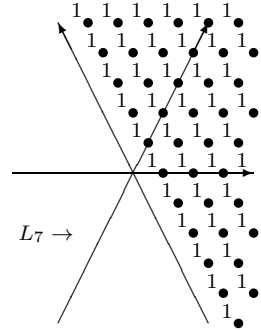
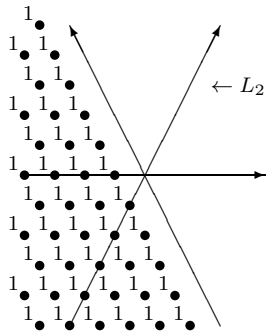
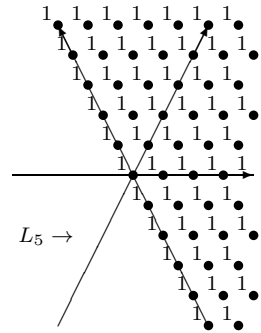
\vspace{0.3cm}
\begin{center}
\begin{tabular}{|>{$}l<{$}|>{$}c<{$}|}\hline
\textsf{Module} &   \hspace{0.3cm}\text{Basis}\\ \hline
L_{2} &
L\left( \begin{array}{c}
-s <m\leq -t\\
k\leq m
\end{array}\right)\\ \hline
L_{4} &
L\left( \begin{array}{c}
-t <m\leq 0\\
k\leq m
\end{array}\right)\\ \hline
L_{5} &
L\left( \begin{array}{c}
-s <m\leq -t\\
m<k
\end{array}\right)\\ \hline
L_{7} &
L\left( \begin{array}{c}
-t <m\leq 0\\
m<k
\end{array}\right)\\ \hline
\end{tabular}%
\end{center}

\end{itemize}
\vspace{0.3cm}
\item[II.] \textbf{Loewy series.}\\
$$L_{1} ,\  L_{2}\oplus L_{3} ,\  L_{4}\oplus L_{5} ,\  L_{6}\oplus L_{7} ,\  L_{8}.$$
\end{itemize}


 \item[$(G15)$] Set $t, s\in\mathbb{Z}_{> 0}$ with $t<s$ and let $T(v)$ be the following tableau:

\begin{center}

 \hspace{1.5cm}\Stone{$a$}\Stone{$a-t$}\Stone{$a-s$}\\[0.2pt]
 $T(v)$=\hspace{0.3cm} \Stone{$a$}\Stone{$y$}\\[0.2pt]
 \hspace{1.3cm}\Stone{$y$}\\
\end{center}
\noindent
\begin{itemize}
\item[I.] \textbf{simple subquotients}.\\
In this case the module $V(T(v))$ has  $8$ simple subquotients.  We provide pictures of the weight lattice corresponding to the case $t=1$, $s=2$, and with origin at the $\mathfrak{sl}(3)$-weight associated to the tableau $T(v-\delta^{22})$.

\begin{itemize}

\item[(i)] Two modules with infinite-dimensional weight multiplicities:
\vspace{0.3cm}
\begin{center}
\begin{tabular}{|>{$}l<{$}|>{$}c<{$}|}\hline
\textsf{Module} &   \hspace{0.3cm}\text{Basis}\\ \hline
L_{1} &
L\left( \begin{array}{c}
m\leq -s\\
k\leq n
\end{array}\right)\\ \hline
L_{8} &
L\left( \begin{array}{c}
0<m\\
n<k
\end{array}\right)\\ \hline
\end{tabular}
\end{center}
\vspace{0.3cm}

\item[(ii)] Two modules with unbounded finite weight multiplicities:
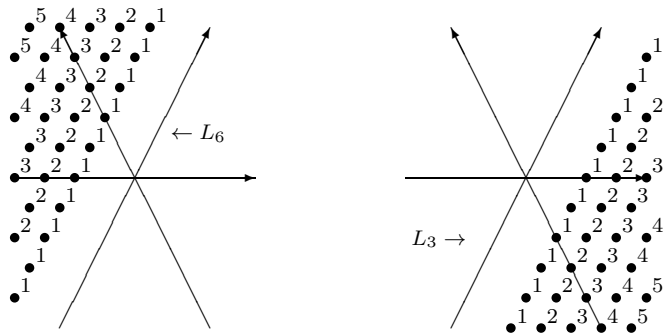
\vspace{0.3cm}
\begin{center}
\begin{tabular}{|>{$}l<{$}|>{$}c<{$}|}\hline
\textsf{Module} &   \hspace{0.3cm}\text{Basis}\\ \hline
L_{3} &
L\left( \begin{array}{c}
m\leq -s\\
n<k
\end{array}\right)\\ \hline
L_{6} &
L\left( \begin{array}{c}
0<m\\
k\leq n
\end{array}\right)\\ \hline
\end{tabular}%
\end{center}

\vspace{0.3cm}

\item[(iii)] Four modules with bounded weight multiplicities. The modules $L_4$ and $L_7$ have multiplicities bounded by $t$, and $L_2$, $L_5$ have multiplicities bounded by $s-t$. 
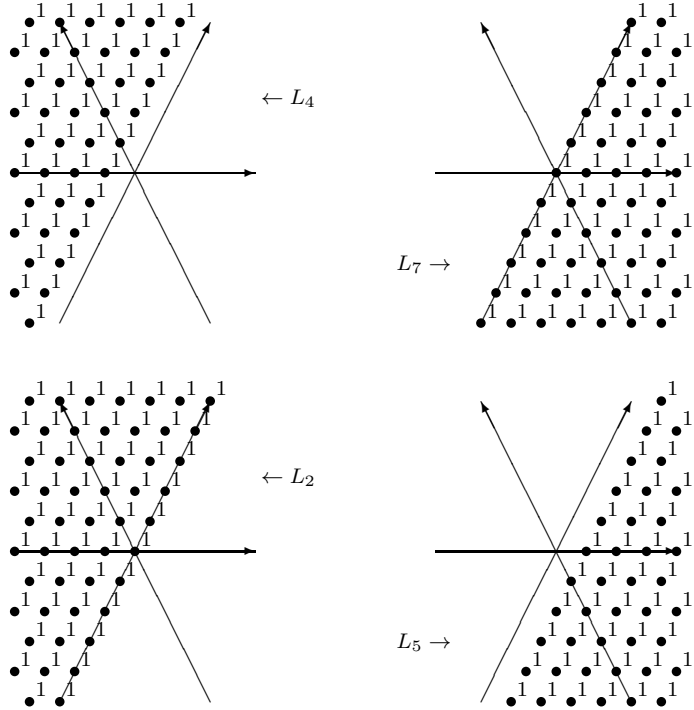
\vspace{0.3cm}
\begin{center}
\begin{tabular}{|>{$}l<{$}|>{$}c<{$}|}\hline
\textsf{Module} &   \hspace{0.3cm}\text{Basis}\\ \hline
L_{2} &
L\left( \begin{array}{c}
-s <m\leq -t\\
k\leq n
\end{array}\right)\\ \hline
L_{4} &
L\left( \begin{array}{c}
-t <m\leq 0\\
k\leq n
\end{array}\right)\\ \hline
L_{5} &
L\left( \begin{array}{c}
-s <m\leq -t\\
n<k
\end{array}\right)\\ \hline
L_{7} &
L\left( \begin{array}{c}
-t <m\leq 0\\
n<k
\end{array}\right)\\ \hline
\end{tabular}%
\end{center}

\end{itemize}
\vspace{0.3cm}
\item[II.] \textbf{Loewy series.}\\
$$L_{1} ,\  L_{2}\oplus L_{3} ,\  L_{4}\oplus L_{5} ,\  L_{6}\oplus L_{7} ,\  L_{8}.$$
\end{itemize}

  \item[$(G16)$] For any $t,s\in\mathbb{Z}_{> 0}$ with $t<s$ let $T(v)$ be the following tableau:

\begin{center}

 \hspace{1.5cm}\Stone{$a$}\Stone{$a-t$}\Stone{$a-s$}\\[0.2pt]
 $T(v)$=\hspace{0.3cm} \Stone{$a$}\Stone{$y$}\\[0.2pt]
 \hspace{1.3cm}\Stone{$z$}\\
\end{center}
\noindent
\begin{itemize}
\item[I.] \textbf{simple subquotients}.\\
In this case the module $V(T(v))$ has  $4$ simple subquotients. The bases and corresponding weight multiplicities are given by:

\begin{itemize}

\item[(i)] Two modules with infinite-dimensional weight multiplicities:
\vspace{0.3cm}
\begin{center}
\begin{tabular}{|>{$}l<{$}|>{$}c<{$}|}\hline
\textsf{Module} &   \hspace{0.3cm}\text{Basis}\\ \hline
L_{1} &
L\left( \begin{array}{c}
m\leq -s
\end{array}\right)\\ \hline
L_{4} &
L\left( \begin{array}{c}
0<m
\end{array}\right)\\ \hline
\end{tabular}
\end{center}

\vspace{0.3cm}
\item[(ii)] Two cuspidal modules; $L_{2}$ with weight multiplicities $s-t$ and $L_{3}$ with weight multiplicities $t$:
\vspace{0.3cm}
\begin{center}
\begin{tabular}{|>{$}l<{$}|>{$}c<{$}|}\hline
\textsf{Module} &   \hspace{0.3cm}\text{Basis}\\ \hline
L_{2} &
L\left( \begin{array}{c}
-s< m\leq -t
\end{array}\right)\\ \hline
L_{3} &
L\left( \begin{array}{c}
-t< m\leq 0
\end{array}\right)\\ \hline
\end{tabular}
\end{center}
\end{itemize}
\vspace{0.3cm}
\item[II.] \textbf{Loewy series.}\\
$$L_{1} ,\  L_{2},\  L_{3} ,\  L_{4}.$$
\end{itemize}

\end{itemize}

\subsection{Structure of singular $\mathfrak{sl}(3)$-modules $V(T(\bar{v}))$}\label{subsection: Structure of singular sl(3)-modules V(T(v))}
In this subsection we  describe all simple singular Gelfand-Tsetlin  $\mathfrak{sl}(3)$-modules. Like in the generic case it is  enough to explicitly present  all simple subquotients of $V(T(\bar{v}))$ for every $1$-singular vector $\bar{v}$.

\subsubsection{Singular Gelfand-Tsetlin formulas}\label{subsubsection: Singular Gelfand-Tsetlin formulas} Recall the construction of the $1$-singular $\mathfrak{sl}(3)$-modules  $V(T(\bar{v}))$ in \S\ref{subsection: Singular Gelfand-Tsetlin modules}. We adapt this construction to $\g = \mathfrak{sl}(3)$. Since the singularity appears in row $2$, we  fix $\bar{v}\in T_{3}(\mathbb{C})$ such that $\bar{v}_{21}=\bar{v}_{22}$. Also, we denote by $\tau$ the permutation in $S_{3}\times S_{2}\times S_{1}$ that interchanges  the $(2,1)$th and $(2,2)$th entries and is identity on row $2$.

Recall that $V(T(\bar{v}))$, as a vector space, is generated by the set of tableaux  $\{T(\bar{v}+z),\mathcal{D}T(\bar{v}+z')\ |\ z,z'\in\mathbb{Z}^{3}\}$ satisfying the relations $
T(\bar{v} + z) - T(\bar{v} +\tau(z)) = 0$ and $
{\mathcal D} T(\bar{v} + z) + {\mathcal D} T(\bar{v} +\tau(z)) = 0$. As explained in Remark \ref{Remark: The set of all singular tableaux is not a basis},
 $$
 \mathcal{B}(T(\bar{v}))=\{T(\bar{v}+z),\mathcal{D}T(\bar{v}+w)\ |\ z_{21}\leq z_{22}\ \text{ and }\ w_{21}>w_{22}\}.
 $$
 is a basis of $V(T(\bar{v}))$. 
 
\begin{definition}\label{Definition: basis B(T(bar(v))) and Tab(w)}
Given $w\in\mathbb{Z}^{3}$, the \emph{tableau associated to} $w$ with respect to $\mathcal{B}(T(\bar{v}))$ is defined by
\begin{equation*}\Tab(w):=\begin{cases} 
T(\bar{v}+w), & \text{ if } w_{21}\leq w_{22}\\
\mathcal{D}T(\bar{v}+w), & \text{ if } w_{21}>w_{22}.
\end{cases}
\end{equation*}
In particular, $\mathcal{B}(T(\bar{v}))=\{\Tab(w)\ |\  w\in\mathbb{Z}^{3}\}$.
\end{definition}

We next write explicitly the formulas for the action of $\mathfrak{sl}(3)$ on $V(T(\bar{v}))$. We again use the conventions $v=(v_{31},v_{32},v_{33},v_{21},v_{22},v_{11})$,  $\bar{v}=(a,b,c,x,x,z)$, and $w=(0,0,0,m,n,k)$. For any rational function $f$ in $\{v_{ij}\}_{1\leq j\leq i\leq 3}$, $\partial_{v_{ij}}^{\bar{v}}(f)$ will stand for $\frac{\partial f}{\partial v_{ij}}(\bar{v})$.

Recall that by Proposition \ref{Proposition: Action of E_ij in terms of e_ij(v)}, the Gelfand-Tsetlin formulas for generic modules can be written as follows:
$$H_i(T(v+z))=h_i(v+z)T(v+z)$$
$$E_{rs}T(v+z)=\sum_{\sigma\in\Phi_{rs}}e_{rs}(\sigma(v+z))T(v+z+\sigma(\varepsilon_{rs}))$$
where $h_{1}(w)=2w_{11}-(w_{21}+w_{22}+1)$, $h_{2}(w)=2(w_{21}+w_{22}+1)-w_{11}$, and 
 $e_{rs}(w)$, $\varepsilon_{rs}$ are defined in Example \ref{Example: GT formulas in terms of permutations for gl(3)}.

Using the relations $T(\bar{v}+z) - T(\bar{v}+\tau(z)) =0$ and $\mathcal{D}T(\bar{v}+z) + \mathcal{D}T(\bar{v}+\tau(z)) =0$ (see Section \ref{subsection: Singular Gelfand-Tsetlin modules}) we can write the above formulas in a simpler form. Set for convenience $\bar{w}=\bar{v}+w$. Then:
\begin{align}
H_{i}(\Tab(w))&=h_{i}(\bar{v}+w)\Tab(w), \text{ for } i=1,2.&\label{Formula: action of E_ii on singular tableaux}\\
E_{21}(\Tab(w))&=\Tab(w+\varepsilon_{21}).&\label{Formula: action of E_21 on singular tableaux}\\
E_{12}(\Tab(w))&=e_{12}(\bar{v}+w)\Tab(w+\varepsilon_{12}).&\label{Formula: action of E_12 on singular tableaux}
\end{align}

If $E=E_{ij}\in\{E_{32},E_{23},E_{31},E_{13}\}$, $\varepsilon$ denotes the corresponding $\varepsilon_{ij}$, and $\tilde{e}(w):=(v_{21}-v_{22})e(w)$, then the action of $E$ on $\Tab(w)$ is given by:
\begin{align}
&2\left(\tilde{e}(\bar{w})\mathcal{D}T(\bar{w}+\varepsilon) +
\mathcal{D}^{\bar{v}}(\tilde{e}(v+w))T(\bar{w}+\varepsilon)\right), & \text{ if } w_{21}=w_{22}.\label{Formula: action of E on critical tableaux}\\
&e(\bar{w})T(\bar{w}+\varepsilon)+e(\tau(\bar{w}))T(\bar{w}+\tau(\varepsilon)),& \text{if }  w_{21}<w_{22}.\label{Formula: action of E on regular tableaux}\\
&\begin{cases}
\mathcal{D}^{\bar{v}}(e(v+w))T(\bar{w}+\varepsilon)+
\mathcal{D}^{\bar{v}}(e(\tau(v+w)))T(\bar{w}+\tau(\varepsilon))&\\
+e(\bar{w})\mathcal{D}T(\bar{w}+\varepsilon)+e(\tau(\bar{w}))\mathcal{D}T(\bar{w}+\tau(\varepsilon)),&  
\end{cases}& \text{if }  w_{21}>w_{22}.\label{Formula: action of E on derivative tableaux}
\end{align}

\begin{remark}\label{Remark: coefficients are the same as classical  for tableaux of same kind}
Note that the Gelfand-Tsetlin formulas for singular tableaux have the same coefficients as in the classical formulas for tableaux of the same type (see formulas (\ref{Formula: action of E_ii on singular tableaux}), (\ref{Formula: action of E_21 on singular tableaux}), (\ref{Formula: action of E_12 on singular tableaux}), (\ref{Formula: action of E on critical tableaux}), (\ref{Formula: action of E on regular tableaux}),  (\ref{Formula: action of E on derivative tableaux})). More precisely, the action of the generators $E_{ij}$ on a regular tableau is a linear combination of regular tableaux with the same coefficients as those that appear in the classical Gelfand-Tsetlin formulas. On the other hand, the corresponding action on a derivative tableau is a linear combination of both regular and derivative tableaux, and  the coefficients of the derivative tableaux are the same as those that appear in the classical formulas.
\end{remark}
\subsubsection{Explicit formulas and computations}
Some explicit computations are included in the following example. 
\begin{example}
For any complex numbers $a$ and  $v_{ij}$, $1\leq j\leq i\leq 3$, consider the following Gelfand-Tsetlin tableaux.

\begin{center}
 \hspace{1.6cm}\Stone{$v_{31}$}\Stone{$v_{32}$}\Stone{$v_{33}$}\hspace{1cm}\Stone{$a$}\Stone{$a$}\Stone{$a$}\\[0.2pt]
 $T(v)$= \hspace{0.3cm}\Stone{$v_{21}$}\Stone{$v_{22}$} \hspace{0.8cm} $T(\bar{v})$=  \hspace{0.2cm} \Stone{$a$}\Stone{$a$}\\[0.2pt]
 \hspace{1.3cm}\Stone{$v_{11}$}\hspace{4cm}\Stone{$a$}\\
\end{center}

Consider $w=(0,1,0)=\delta^{22}$ and $w'=(1,0,0)=\delta^{21}$ in $\mathbb{Z}^{3}$. Then the following hold. 

\begin{itemize}
\item[(i)] Using (\ref{Formula: action of E on regular tableaux}), $E_{32}(T(\bar{v}+w)) = \mathcal{D}^{\bar{v}}((v_{21}-v_{22})E_{32}(T(v+w)))$ which equals to:
{\small \begin{align*}
&\mathcal{D}^{\bar{v}}\left((v_{21}-v_{22})\left(\frac{v_{21}-v_{11}}{v_{21}-(v_{22}+1)}T(v+w-\delta^{21})+\frac{(v_{22}+1)-v_{11}}{(v_{22}+1)-v_{21}}T(v+w-\delta^{22})\right)\right)\\
=& \mathcal{D}^{\bar{v}}\left(\frac{(v_{21}-v_{22})(v_{21}-v_{11})}{v_{21}-(v_{22}+1)}\right)T(\bar{w}-\delta^{21})+ ev({\bar{v}})\left(\frac{(v_{21}-v_{22})(v_{21}-v_{11})}{v_{21}-(v_{22}+1)}\right)\mathcal{D}T(\bar{w}-\delta^{21})\\
&+ \mathcal{D}^{\bar{v}}\left(\frac{(v_{21}-v_{22})((v_{22}+1)-v_{11})}{(v_{22}+1)-v_{21}}\right)T(\bar{v})+ ev({\bar{v}})\left(\frac{(v_{21}-v_{22})((v_{22}+1)-v_{11})}{(v_{22}+1)-v_{21}}\right)\mathcal{D}T(\bar{v})\\
=& ev({\bar{v}})\left(\frac{v_{21}-v_{11}}{v_{21}-(v_{22}+1)}\right)T(\bar{w}-\delta^{21})+ ev({\bar{v}})\left(\frac{(v_{21}-v_{22})(v_{21}-v_{11})}{v_{21}-(v_{22}+1)}\right)\mathcal{D}T(\bar{w}-\delta^{21})\\
&+ ev({\bar{v}})\left(\frac{(v_{22}+1)-v_{11}}{(v_{22}+1)-v_{21}}\right)T(\bar{v})\\
=& T(\bar{v}).
\end{align*}}

\item[(ii)] Using (\ref{Formula: action of E on derivative tableaux}), $E_{32}(\mathcal{D}T(\bar{v}+w'))=\mathcal{D}^{\bar{v}}(E_{32}(T(v+w')))$ which equals to:
{\small \begin{align*}
& \mathcal{D}^{\bar{v}}\left(\frac{v_{21}+1-v_{11}}{v_{21}+1-v_{22}}T(v+w'-\delta^{21})\right)+ \mathcal{D}^{\bar{v}}\left(\frac{v_{22}-v_{11}}{v_{22}-(v_{21}+1)}T(v+w'-\delta^{22})\right)\\
=& \mathcal{D}^{\bar{v}}\left(\frac{v_{21}+1-v_{11}}{v_{21}+1-v_{22}}\right)T(\bar{v})+ ev({\bar{v}})\left(\frac{v_{21}+1-v_{11}}{v_{21}+1-v_{22}}\right)\mathcal{D}T(\bar{v})\\
&+ \mathcal{D}^{\bar{v}}\left(\frac{v_{22}-v_{11}}{v_{22}-(v_{21}+1)}\right)T(\bar{v}+\delta^{21}-\delta^{22})+ ev({\bar{v}})\left(\frac{v_{22}-v_{11}}{v_{22}-(v_{21}+1)}\right)\mathcal{D}T(\bar{v}+\delta^{21}-\delta^{22})\\
=& \mathcal{D}^{\bar{v}}\left(\frac{v_{21}+1-v_{11}}{v_{21}+1-v_{22}}\right)T(\bar{v})+ \mathcal{D}^{\bar{v}}\left(\frac{v_{22}-v_{11}}{v_{22}-(v_{21}+1)}\right)T(\bar{v}+\delta^{21}-\delta^{22})\\
=& \frac{1}{2}T(\bar{v})-\frac{1}{2}T(\bar{v}+\delta^{21}-\delta^{22}).
\end{align*}}
\end{itemize}

\end{example}

The computations made in the last example can easily be applied  to  the formulas  (\ref{Formula: action of E_ii on singular tableaux}), (\ref{Formula: action of E_21 on singular tableaux}), (\ref{Formula: action of E_12 on singular tableaux}), (\ref{Formula: action of E on critical tableaux}), (\ref{Formula: action of E on regular tableaux}),  (\ref{Formula: action of E on derivative tableaux}). As a result we have the following set of formulas.

\noindent {\it Action of the generators on regular tableaux:}\\
\noindent
$
\begin{array}{lll}
E_{21}(T(\bar{w}))&=&T(\bar{w}-\delta^{11}).
\end{array}
$

\noindent
$
\begin{array}{lll}
E_{12}(T(\bar{w}))&=&-(x+m-z-k)(x+n-z-k)T(\bar{w}+\delta^{11}).
\end{array}
$

\noindent
$
\begin{array}{lll}
E_{32}(T(\bar{w}))&=&
\small{\begin{dcases}
T(\bar{w}-\delta^{21})+2(x+m-z-k)\mathcal{D}T(\bar{w}-\delta^{21}), & \text{if $m=n$}\\
\frac{x+m-z-k}{m-n}T(\bar{w}-\delta^{21})-\frac{x+n-z-k}{m-n}T(\bar{w}-\delta^{22}), & \text{if $m\neq n$}.
\end{dcases}}
\end{array}
$

\noindent
$
\begin{array}{lll}
E_{23}(T(\bar{w}))&=&
\begin{dcases}
\partial_{v_{21}}^{\bar{v}}\left(\prod_{i=1}^{3}(v_{3i}-v_{21}-m)\right)T(\bar{w}+\delta^{22}) & \\ \small{-2(a-x-m)(b-x-m)(c-x-m)}\mathcal{D}T(\bar{w}+\delta^{22}), & \text{if $m=n$}\\
 & \\
 \frac{(a-x-m)(b-x-m)(c-x-m)}{m-n}T(\bar{w}+\delta^{21})-&\\ \frac{(a-x-n)(b-x-n)(c-x-n)}{m-n}T(\bar{w}+\delta^{22}), & \text{if $m\neq n$}.
\end{dcases}
\end{array}
$

\medskip
\noindent  {\it Action of the generators on derivative tableaux (recall that we assume $m>n$):}

\noindent
$
\begin{array}{lll}
E_{21}(\mathcal{D}T(\bar{w}))=&\mathcal{D}T(\bar{w}-\delta^{11}).
\end{array}
$

\noindent
$
\begin{array}{lll}
E_{12}(\mathcal{D}T(\bar{w}))=&-(x+m-z-k)(x+n-z-k)\mathcal{D}T(\bar{w}+\delta^{11})+\\
&\frac{m-n}{2}T(\bar{w}+\delta^{11}).
\end{array}
$

\noindent
$
\begin{array}{lll}
E_{32}(\mathcal{D}T(\bar{w}))=&\displaystyle\frac{x+m-z-k}{m-n}\mathcal{D}T(\bar{w}-\delta^{21})-\frac{x+n-z-k}{m-n}\mathcal{D}T(\bar{w}-\delta^{22})+\\
&\frac{1}{2}\left( \frac{1}{m-n}-\frac{2(x+m-z-k)}{(m-n)^{2}}\right)T(\bar{w}-\delta^{21})-\\
&\frac{1}{2}\left( \frac{1}{m-n}-\frac{2(x+n-z-k)}{(n-m)^{2}}\right)T(\bar{w}-\delta^{22}).
\end{array}
$

\noindent
$
\begin{array}{lll}
E_{23}(\mathcal{D}T(\bar{w}))=&-\displaystyle\frac{(a-x-m)(b-x-m)(c-x-m)}{n-m}\mathcal{D}T(\bar{w}+\delta^{21})-\\ 
&\displaystyle\frac{(a-x-n)(b-x-n)(c-x-n)}{m-n}\mathcal{D}T(\bar{w}+\delta^{22})- \\  
&\partial_{v_{22}}^{\bar{v}}\left(\displaystyle\frac{\prod\limits_{i=1}^{3}(v_{3i}-v_{22}-n)}{m-n}\right)T(\bar{w}+\delta^{22})-\\
& 2\displaystyle\frac{(a-x-n)(b-x-n)(c-x-n)}{(m-n)^{2}}T(\bar{w}+\delta^{22})+ \\ 
&-\partial_{v_{21}}^{\bar{v}}\left(\displaystyle\frac{\prod\limits_{i=1}^{3}(v_{3i}-v_{21}-m)}{n-m}\right)T(\bar{w}+\delta^{21})-\\
&2\displaystyle\frac{(a-x-m)(b-x-m)(c-x-m)}{(m-n)^{2}}T(\bar{w}+\delta^{21}).
\end{array}
$


\begin{lemma}\label{Lemma: Action of Gamma for singular sl(3)}
The action of $\Gamma$ on $V(T(\bar{v}))$ is given by the formulas:
\begin{align}
c_{rs}(T(\bar{v}+w))=&\gamma_{rs}(\bar{v}+w)T(\bar{v}+w),\label{Formulas: Gamma acting on singular T}\\
c_{rs}(\mathcal{D}T(\bar{v}+w))=&\gamma_{rs}(\bar{v}+w)\mathcal{D}T(\bar{v}+w)+\mathcal{D}^{\bar{v}}(\gamma_{rs}(v+w))T(\bar{v}+w).\label{Formulas: Gamma acting on singular DT}
\end{align}
\end{lemma}
\begin{proof}
The identities follow  from Theorem \ref{Theorem: 1-singular modules}. Indeed,

\begin{align*}
c_{rs}(T(\bar{v}+z))=&\mathcal{D}^{\overline{v}}((v_{21} - v_{22})c_{rs}T(v+z))
\\
=&\mathcal{D}^{\overline{v}}((v_{21} - v_{22})\gamma_{rs}(v+z)T(v+z))
\\
=&\mathcal{D}^{\overline{v}}((v_{21} - v_{22})\gamma_{rs}(v+z))T(\bar{v}+z)+ev({\bar{v}})\left((v_{21} - v_{22})\gamma_{rs}(v+z)\right)\mathcal{D}T(\bar{v}+z)
\\
=&\gamma_{rs}(\bar{v}+z)T(\bar{v}+z),\\
&\\
c_{rs}(\mathcal{D}T(\bar{v}+z))=&\mathcal{D}^{\overline{v}}(c_{rs}(T(v+z)))
\\
=&\mathcal{D}^{\overline{v}}(\gamma_{rs}(v+z)T(v+z))
\\
=&\mathcal{D}^{\overline{v}}(\gamma_{rs}(v+z))T(\bar{v}+z)+ev({\bar{v}})\left(\gamma_{rs}(v+z)\right)\mathcal{D}T(\bar{v}+z)
\\
=&\mathcal{D}^{\overline{v}}(\gamma_{rs}(v+z))T(\bar{v}+z)+\gamma_{rs}(\bar{v}+z)\mathcal{D}T(\bar{v}+z).
\end{align*}
\end{proof}

\subsubsection{Submodules generated by singular tableaux}\label{Subsubection: Submodule generated by a singular tableau}
In this subsection we  obtain an analogous to  Theorem \ref{Theorem: basis for Irr generic modules}(i) for $1$-singular tableaux.\\

Recall that $\mathcal{B}(T(\bar{v})) = \{\Tab(z)\ |\ z\in\mathbb{Z}^{3}\}$ is a basis of $V(T(\bar{v}))$. 

\begin{definition}\label{Definition: Lambda+}
Let  $\bar{v}$ be a fixed critical vector, $\Tab(w)\in \mathcal{B}(T(\bar{v}))$, and $\bar{w}=\bar{v}+w$. Define:
$$\Lambda^{+}(\Tab(w))=\begin{cases}
\Omega^{+}(\Tab(w)),&\text{ if } w_{21}\leq w_{22}\\
\Omega^{+}(\Tab(\tau(w))),&\text{ if } w_{21}> w_{22}.
\end{cases}
$$

\end{definition}

\begin{lemma}\label{Lemma: from derivative tableau we can obtain regular tableau}
Assume that $w_{21}\neq w_{22}$. Then $T(\bar{v}+w)$ belongs to $U\cdot \mathcal{D}T(\bar{v}+w)$.
\end{lemma}
\begin{proof}
The action of $c_{22}-\gamma_{22}(\bar{v}+w)$ on $\mathcal{D}T(\bar{v}+w)$ is given by the formula (\ref{Formulas: Gamma acting on singular DT}) and can be easily check that is a nonzero multiple of $(w_{21}-w_{22})T(\bar{v}+w)$. 
\end{proof}

\begin{lemma}\label{Lemma: From critical tableau we obtain derivative satisfying Omega+ condition}
Suppose $\Tab(w)$ is a critical tableau. If $\Tab(w')$ is a derivative tableau such that $\Lambda^{+}(\Tab(w))\subseteq \Lambda^{+}(\Tab(w'))$, then  $\Tab(w')\in U\cdot \Tab(w)$. In particular, the simple subquotient $M$ of $V(T(\bar{v}))$ containing $\Tab(w)$ satisfies $\dim(M_{\chi_{w}})=2$.
\end{lemma}
\begin{proof}
The statement follows  from Remark \ref{Remark: coefficients are the same as classical  for tableaux of same kind} and formulas (\ref{Formula: action of E_21 on singular tableaux}, \ref{Formula: action of E_12 on singular tableaux}, \ref{Formula: action of E on critical tableaux}). In fact, the numerators of the coefficients of the derivative tableaux appearing in the decomposition of $gT(\bar{v}+w)$ as linear combination of basis elements, are either zero (if $g$ is a product of generators of the form $E_{21}$, $E_{12}$ or $E_{ii}$) or the same as the numerators of the coefficients that appear in the classical Gelfand-Tsetlin formulas. In the latter case $\Tab(w')$ is a derivative tableau, hence we can not have zero coefficients. Therefore, we can use the same arguments as in the proof of Lemma \ref{Theorem: basis for Irr generic modules}(i).
\end{proof}

\begin{definition}
For any tableau $\Tab(w)\in \mathcal{B}(T(\bar{v}))$ define 
$$\mathcal{A}(\Tab(w)):=\{\Tab(w')\in\mathcal{B}(T(\bar{v}))\ |\ \Lambda^{+}(\Tab(w))\subseteq \Lambda^{+}(\Tab(w'))\}.$$ 

By $\mathcal{C}(w)$ we will denote the set of all critical tableaux in $\mathcal{A}(\Tab(w))$ and by $\mathcal{R}(w)$ we will denote the set of all regular tableaux in $\mathcal{A}(\Tab(w))$. Also, set:
$$
\mathcal{N}(\Tab(w))=\begin{cases}
\mathcal{R}(w)\cup\left(\bigcup\limits_{w'\in \mathcal{C}(w)}\mathcal{A}(\Tab(w'))\right),& \text{ if } w_{21}<w_{22}\\
\mathcal{A}(\Tab(w)),& \text{ if } w_{21}\geq w_{22}.
\end{cases}
$$
\end{definition}

\begin{lemma}\label{Lemma: We can generate certain set of tableaux as in generic case}
For any tableau $\Tab(w)$ we have $\mathcal{N}(\Tab(w))\subseteq U\cdot \Tab(w)$.
\end{lemma}
\begin{proof}
The statement follows from Remark \ref{Remark: coefficients are the same as classical  for tableaux of same kind},  Lemma \ref{Lemma: from derivative tableau we can obtain regular tableau}, and  Lemma \ref{Lemma: From critical tableau we obtain derivative satisfying Omega+ condition}. More precisely, as the Gelfand-Tsetlin formulas for singular tableaux have the same coefficients as in the classical formulas for tableaux of the same type, we can use the reasoning in the proof of Theorem 6.8 in \cite{FGR2} and adapt it to the singular case.
\end{proof}
The following lemma together with Lemma \ref{Lemma: from derivative tableau we can obtain regular tableau} gives a sufficient condition in order to have modules with Gelfand-Tsetlin multiplicity $2$.
\begin{lemma}\label{Lemma: condition to get derivative tableau from regular tableau}
Suppose that $\Tab(w)$ is a regular tableau such that $\Omega^{+}(\Tab(w))=\Omega^{+}(\Tab(w'))$ for some critical tableau $\Tab(w')$, then $\Tab(\tau(w))\in U\cdot \Tab(w)$.
\end{lemma}
\begin{proof}
The statement follows directly from Lemma  \ref{Lemma: From critical tableau we obtain derivative satisfying Omega+ condition}.
\end{proof}

\begin{corollary}
Let $\Tab(w)$ be a regular tableau associated to a Gelfand-Tsetlin character $\chi$. If $\{Tab(w')\ |\ \Omega^{+}(\Tab(w'))=\Omega^{+}(\Tab(w))\}$ does not contain critical tableaux, then any simple subquotient $N$ of $V(T(\bar{v}))$ satisfies $\dim(N_{\chi})\leq 1$.
\end{corollary}
\begin{proof}
Since $\Tab(w)$ is regular,  $w_{21}<w_{22}$. Then  $\Tab(\tau(w))$ is a derivative tableau such that $\Tab(w)\in U\cdot \Tab(\tau(w))$ (see Lemma \ref{Lemma: from derivative tableau we can obtain regular tableau}). Therefore, it is enough to prove that $\Tab(\tau(w))\notin U\cdot \Tab(w)$. However,  this follows from Theorem \ref{Theorem: 1-singular modules} and the fact that we can not obtain critical tableaux from $\Tab(w)$ with the same $\Omega^{+}(\Tab(w))$, in particular we can not obtain derivative tableaux $\Tab(w')$ such that $\Omega^{+}(\Tab(w'))=\Omega^{+}(\Tab(w))$. 
\end{proof}
\begin{remark}
By definition of  $\mathcal{N}(\Tab(w))$, any $\Tab(w')$ in $\mathcal{N}(\Tab(w))$ satisfies the relation $|\Omega^{+}(\Tab(w))|\leq |\Omega^{+}(\Tab(w'))|$. However it is possible to have  $\Tab(w')\in U\cdot \Tab(w)$ with $|\Omega^{+}(\Tab(w'))|= |\Omega^{+}(\Tab(w))|-1$. For instance, consider $\bar{v}
=(a,b,c,x,x,x)$ such that $\{a-x, b-x, c-x\}\cap\mathbb{Z}=\emptyset$ and $w=(0,0,0)$, then $|\Omega^{+}(\Tab(w))|=2$ while $E_{32}\Tab(w)=\Tab(w-\delta^{21})$ and $|\Omega^{+}(\Tab(w-\delta^{21}))|=1$.
\end{remark}
Let us write $\Tab(w')\prec_{g} \Tab(w)$ if $\Tab(w')$ appears with no zero coefficient in the decomposition of  $g\cdot \Tab(w)$ for some generator $g\in\mathfrak{gl}(n)$.

\begin{lemma}[\cite{GoR} Lemma 7.4]\label{Lemma: tableaux where size of Omega+ decreases}
Suppose that $\Tab(w')\prec_{g} \Tab(w)$ with $g\in\mathfrak{gl}(n)$ of the form $E_{k,k+1}$ or $E_{k+1,k}$, then $|\Omega^{+}(\Tab(w'))|\geq |\Omega^{+}(\Tab(w))|-1$. Moreover, the complete list of Gelfand-Tsetlin tableaux $\Tab(w)$ and $\Tab(w')$  such that $\Tab(w')\prec_{g} \Tab(w)$ and $|\Omega^{+}(\Tab(w'))|= |\Omega^{+}(\Tab(w))|-1$ is as follows. 
\medskip
\begin{itemize}
\item[(i)]
\begin{center}
 \hspace{1.6cm}\Stone{$a$}\Stone{$b$}\Stone{$c$}\hspace{1cm}\Stone{$a$}\Stone{$b$}\Stone{$c$}\\[0.2pt]
 $\Tab(w)$= \Stone{$x$}\Stone{$x-t$} \hspace{0.5cm} $\Tab(w')$=  \Stone{$x-t$}\Stone{$x$}\\[0.2pt]
 \hspace{1.5cm}\Stone{$x$}\hspace{3.9cm}\Stone{$x+1$}\\\vspace{0.3cm}
\end{center}
for $t\in\mathbb{Z}_{>0}$. 
\vspace{0.3cm}

\item[(ii)]
\begin{center}
 \hspace{1.6cm}\Stone{$a$}\Stone{$b$}\Stone{$c$}\hspace{1cm}\Stone{$a$}\Stone{$b$}\Stone{$c$}\\[0.2pt]
 $\Tab(w)$= \Stone{$x$}\Stone{$x-t$} \hspace{0.5cm} $\Tab(w')$=  \Stone{$x-t$}\Stone{$x-1$}\\[0.2pt]
 \hspace{1.5cm}\Stone{$x$}\hspace{3.9cm}\Stone{$x$}\\
\end{center}
for $t\in\mathbb{Z}_{>0}$. 

\vspace{0.3cm}

\item[(iii)] 
\begin{center}
 \hspace{1.6cm}\Stone{$a$}\Stone{$b$}\Stone{$c$}\hspace{1cm}\Stone{$a$}\Stone{$b$}\Stone{$c$}\\[0.2pt]
 $\Tab(w)$= \Stone{$x$}\Stone{$x$} \hspace{0.5cm} $\Tab(w')$=  \Stone{$x-1$}\Stone{$x$}\\[0.2pt]
 \hspace{1.5cm}\Stone{$x$}\hspace{3.9cm}\Stone{$x$}.\\
\end{center}\vspace{0.3cm}

\item[(iv)] \begin{center}
 \hspace{1.6cm}\Stone{$x$}\Stone{$b$}\Stone{$c$}\hspace{1cm}\Stone{$x$}\Stone{$b$}\Stone{$c$}\\[0.2pt]
 $\Tab(w)$= \Stone{$x$}\Stone{$x-t$} \hspace{0.5cm} $\Tab(w')$=  \Stone{$x-t$}\Stone{$x+1$}\\[0.2pt]
 \hspace{1.5cm}\Stone{$z$}\hspace{3.9cm}\Stone{$z$}\\
\end{center}
for $t\in\mathbb{Z}_{>0}$, $b\neq x$, and $c\neq x$.
\vspace{0.3cm}
\vspace{0.3cm}
\item[(v)] 
\begin{center}
 \hspace{1.6cm}\Stone{$x$}\Stone{$b$}\Stone{$c$}\hspace{1cm}\Stone{$x$}\Stone{$b$}\Stone{$c$}\\[0.2pt]
 $\Tab(w)$= \Stone{$x$}\Stone{$x$} \hspace{0.5cm} $\Tab(w')$=  \Stone{$x$}\Stone{$x+1$}\\[0.2pt]
 \hspace{1.5cm}\Stone{$z$}\hspace{3.9cm}\Stone{$z$}\\
\end{center}
for $b\neq x$ and $c\neq x$.
\vspace{0.3cm}
\vspace{0.3cm}
\end{itemize}
\end{lemma}
\begin{remark}
For the tableaux in Lemma \ref{Lemma: tableaux where size of Omega+ decreases}(iv)(v), one may consider $x$ at  positions $3i$, $i=1,2,3$, obtaining the same property for $\Omega^+$.
\end{remark}

\begin{definition}\label{Definition: Type I and II tableaux}
We will say that a tableau $\Tab(w)$ is \emph{of type (I)} if it can be written in the form of one of the tableaux $\Tab(w)$ of parts (i), (ii) or (iii) of Lemma \ref{Lemma: tableaux where size of Omega+ decreases} for some $x,a,b,c,t,z$. We also say that the tableau is \emph{of type (II)$_{i}$} if can be written in the form of one of the tableaux $\Tab(w)$ of parts (iv) or (v) of Lemma \ref{Lemma: tableaux where size of Omega+ decreases} for some $x,b,c,t,z$ and $x$ appear in the top row in position $3i$.
\end{definition}

\begin{remark}\label{Remark: we can obtain tableaux with smaller Omega+}
With the notation of Lemma \ref{Lemma: tableaux where size of Omega+ decreases} for tableaux of type (I) we have $\Omega^{+}(\Tab(w'))=\Omega^{+}(\Tab(w))\setminus\{(2,1,1)\}$ and for tableaux of type (II)$_{i}$ we have $\Omega^{+}(\Tab(w'))=\Omega^{+}(\Tab(w))\setminus\{(3,i,2)\}$.
\end{remark}
\begin{definition}\label{Definition: N_(r',s',t') and N^(i)}
Let $\Tab(w)\in\mathcal{B}(T(\bar{v}))$ and $(r,s,t)\in \Omega^{+}(\Tab(w))$. Set 
$$
\mathcal{N}_{(r,s,p)}(\Tab(w)):=\mathcal{N}(\Tab(w))\cup \mathcal{N}(\Tab(w')),$$
where $\Tab(w')$ is any tableau such that $\Omega^{+}(\Tab(w))\setminus\{(r,s,p)\}=\Omega^{+}(\Tab(w'))$. Also, define:
\begin{equation}\label{Equation: N^(1)}
\mathcal{N}^{(1)}(\Tab(w)):=
\begin{cases}
\mathcal{N}_{(2,1,1)}(\Tab(w)),& \text{ if } \Tab(w) \text{ is of type } (I)\\
\mathcal{N}(\Tab(w)), & \text{otherwise},
\end{cases}
\end{equation}
 
\begin{equation}\label{Equation: N^(2)}
\mathcal{N}^{(2)}(\Tab(w)):=
\begin{cases}
\mathcal{N}_{(3,i,2)}(\Tab(w)),& \text{ if } \Tab(w) \text{ is of type } (II)_{i}\\
\mathcal{N}(\Tab(w)), & \text{otherwise}.
\end{cases}.
\end{equation}

\begin{equation}\label{Equation: Ahat}
\widehat{\mathcal{A}}(\Tab(w)):=\mathcal{N}^{(1)}(\Tab(w))\cup\mathcal{N}^{(2)}(\Tab(w))
\end{equation}
\end{definition}

\begin{lemma}\label{Lemma: N^ is inside the generated module}
Let $\Tab(w)$ be any Gelfand-Tsetlin tableau and $\Tab(w')\in\mathcal{N}(\Tab(w))$. We have $\widehat{\mathcal{A}}(\Tab(w'))\subseteq U\cdot \Tab(w)$. 
\end{lemma}
\begin{proof}
As $\Tab(w')\in\mathcal{N}(\Tab(w))$, by Lemma \ref{Lemma: We can generate certain set of tableaux as in generic case}, Lemma \ref{Lemma: tableaux where size of Omega+ decreases}, and Remark \ref{Remark: we can obtain tableaux with smaller Omega+} we have $\mathcal{N}^{(i)}(\Tab(w'))\subseteq U\cdot \Tab(w)$ for $i=1,2$.
\end{proof}

The following theorem summarize the results of this section.

\begin{theorem}\label{Theorem: Module generated by singular tableau}
Let $\Tab(w)\in\mathcal{B}(T(\bar{v}))$. The submodule $U\cdot \Tab(w)$ has the following basis of tableaux:

$$\widehat{\mathcal{N}}(\Tab(w)):=\bigcup\limits_{\Tab(w')\in\mathcal{N}(\Tab(w))}\widehat{\mathcal{A}}(\Tab(w'))$$

\end{theorem}
\begin{proof}
The statement follows from Lemmas \ref{Lemma: tableaux where size of Omega+ decreases} and \ref{Lemma: N^ is inside the generated module}.
\end{proof}
\begin{definition}\label{Definition: Omega+ maximal}
Let $M$ be a Gelfand-Tsetlin module with basis $\mathcal{B}_{M}\subseteq \mathcal{B}(T(\bar{v}))$ for some $1$-singular vector $\bar{v}$. We say that $\Tab(w)\in\mathcal{B}_{M}$ is \emph{$\Omega^{+}$-maximal in $M$} if $|\Omega^{+}(\Tab(w))|$ is maximal for all $\Tab(w)$ in $\mathcal{B}_{M}$. Also, denote by $U\cdot_{M} \Tab(w)$ the submodule of $M$ generated by $\Tab(w)$.
\end{definition}

The next two corollaries follow from Theorem \ref{Theorem: Module generated by singular tableau} and will be useful when describing the simple subquotients of $V(T(\bar{v}))$.\\

\begin{corollary}\label{Corollary: regular tableau with maximal Omega+ generates simple submodule}
Let $M$ be a Gelfand-Tsetlin module with basis $\mathcal{B}_{M}\subseteq \mathcal{B}(T(\bar{v}))$ for some $1$-singular vector $\bar{v}$. If $\Tab(w)\in\mathcal{B}_{M}$ is a regular tableau that is $\Omega^{+}$-maximal in $M$, then $U\cdot_{M} \Tab(w)$ is a simple submodule of $M$.
\end{corollary}
\begin{proof}
It is enough to proof that $\Tab(w)$ belongs to $U\cdot_{M} \Tab(w')$ for any $\Tab(w')$ in $U\cdot_{M} \Tab(w)$. As $\Tab(w')$ in $U\cdot_{M} \Tab(w)$ and $\Tab(w)$ is a regular tableau, we have $\Omega^{+}(\Tab(w))\subseteq \Omega^{+}(\Tab(w'))\cup \{(r,s,t)\}$ for some $(r,s,t)$. As $\Tab(w)$ is $\Omega^{+}$-maximal, we should have $\Omega^{+}(\Tab(w))=\Omega^{+}(\Tab(w'))\cup \{(r,s,t)\}$ for some $(r,s,t)$. Therefore, $U\cdot_{M} \Tab(w)\subseteq U\cdot_{M} \Tab(w')$ and, then we have $\Tab(w)\in U\cdot_{M} \Tab(w')$.
\end{proof}
\begin{corollary}\label{Corollary: derivative tableau with maximal Omega+ generates simple submodule}
Let $M$ be a Gelfand-Tsetlin module with basis $\mathcal{B}_{M}\subseteq \mathcal{B}(T(\bar{v}))$ for some $1$-singular vector $\bar{v}$. If $\{\Tab(w)\in\mathcal{B}_{M}\ |\ \Tab(w) \text{ is } \Omega^{+}\text{-maximal}\}$ does not contain regular tableaux, then for any $\Omega^{+}$-maximal tableau $\Tab(w)$ the submodule $U\cdot_{M} \Tab(w)$ is a simple submodule of $M$.
\end{corollary}
\begin{proof}
The proof is analogous to the proof of Corollary \ref{Corollary: regular tableau with maximal Omega+ generates simple submodule}.
\end{proof}

In order to describe the basis of the simple subquotients of $V(T(\bar{v}))$ we modify Definition \ref{Definition: M(B)} to  singular vectors.
\begin{definition}
Let $\bar{v}$ be any $1$-singular vector and  $\mathcal{B}$ be a subset of ${\mathbb Z}^3$. By $\Tab(\mathcal{B})$ we will denote the set of tableaux $\{ \Tab(m,n,k) \; | \; (m,n,k)\in \mathcal{B} \}\subseteq\mathcal{B}(T(\bar{v}))$. Assume that $M$ is a Gelfand-Tsetlin module with basis $\Tab(\mathcal{B})$. Then we will denote $M$ by $M(\mathcal{B}, \bar{v})$, or simply by $M(\mathcal{B})$ if $\bar{v}$ is fixed. If  $M(\mathcal{B})$ is simple, we will write $L(\mathcal{B})$ for $M(\mathcal{B})$.
\end{definition}

\begin{example}\label{Example: computing the module generates by some tableau}
Let  $\bar{v}=(a,b,c|a,a|z)$. Below we  give  a basis for the submodule of $V(T(\bar{v}))$ generated by $\Tab(0,0,0)$. In this case $\mathcal{B}(T(\bar{v}))$ does not contain tableaux of type $(I)$, $(II)_{2}$ or $(II)_{3}$. However, the set of all tableaux of type $(II)_{1}$ is 
$\{\Tab(0,n,k)\ |\ n\in\mathbb{Z}_{\leq 0}\}$. By definition we have 
$$\widehat{\mathcal{N}}(\Tab(0,0,0))=\widehat{\mathcal{A}}(\Tab(0,0,0))=\mathcal{A}(\Tab(0,0,0))=\{\Tab(m,n,k)\ |\ m\leq 0,\ n\leq 0\},$$
$$\widehat{\mathcal{N}}(\Tab(0,-1,0))=\widehat{\mathcal{A}}(\Tab(0,-1,0))=\mathcal{N}^{(2)}(\Tab(0,-1,0))=\{\Tab(m,n,k)\ |\ m\leq 0\}.$$
Therefore, by Theorem \ref{Theorem: Module generated by singular tableau}, the submodule of $V(T(\bar{v}))$ generated by $\Tab(0,0,0)$ has basis:
$$\bigcup\limits_{\mathcal{N}(\Tab(0,0,0))}\widehat{\mathcal{A}}(\Tab(w'))=\widehat{\mathcal{A}}(\Tab(0,0,0))\cup \widehat{\mathcal{A}}(\Tab(0,-1,0))=\Tab(m\leq 0).$$

\end{example}

\subsection{The singular block containing $L(-\rho)$}\label{Example: The singular block containing L(-rho)} In this subsection, we   describe  all simple subquotients of the module $V(T(\bar{v}))$.

Next we give an algorithm, which based on Theorem \ref{Theorem: Module generated by singular tableau} and Corollaries \ref{Corollary: regular tableau with maximal Omega+ generates simple submodule} and \ref{Corollary: derivative tableau with maximal Omega+ generates simple submodule}, provides an explicit basis of all simple subquotients of a module $M$ with basis $\mathcal{B}_{M}\subseteq \mathcal{B}(T(\bar{v}))$.
\begin{enumerate}[{\emph Step} $1$.]
\item If there is an $\Omega^{+}$-maximal regular tableau in $\mathcal{B}_{M}$,  choose any such tableau $\Tab(w)$. By Corollary \ref{Corollary: regular tableau with maximal Omega+ generates simple submodule}, $U\cdot_{M} \Tab(w)$ is  a simple submodule of $M$.
\item If there are no $\Omega^{+}$-maximal regular tableaux in $\mathcal{B}_{M}$, consider any $\Omega^{+}$-maximal (derivative) tableau $\Tab(w)$. By Corollary \ref{Corollary: derivative tableau with maximal Omega+ generates simple submodule} the  module $U\cdot_{M} \Tab(w)$ will be a simple submodule of $M$.
\item Using the bases of $M$ and $U\cdot \Tab(w)$ (see Theorem \ref{Theorem: Module generated by singular tableau}), we find a basis of  $M/(U\cdot_{M} \Tab(w))$.
\item Start over the procedure with the module $M':=M/(U\cdot_{M} \Tab(w))$.

\end{enumerate}

\begin{example}\label{Example: decomposition of the most singular block}
Let  $\bar{v}=(a,a,a|a,a|a)$. Below we define explicit bases of all simple subquotient of $V(T(\bar{v}))$. 

Note that none of the tableaux in $\mathcal{B}(T(\bar{v}))$ can be of type $(II)_{i}$, $i=1,2,3$. Therefore $\mathcal{N}^{(2)}(\Tab(w))=\mathcal{N}(\Tab(w))$ for any $\Tab(w)\in \mathcal{B}(T(\bar{v}))$. Moreover, the set of all tableaux of type $(I)$ is $\{\Tab(m,n,k)\ |\ n\leq m=k\}$. Now we apply Steps 1--4 described above to the module $V(T(\bar{V}))$. 

\begin{itemize}
\item[(1)] The tableau $\Tab(0,0,0)$ is $\Omega^{+}$-maximal on $V(T(\bar{v}))$. By Corollary \ref{Corollary: regular tableau with maximal Omega+ generates simple submodule}, $U\cdot \Tab(0,0,0)$ is a simple submodule of $V(T(\bar{v}))$ and by Theorem \ref{Theorem: Module generated by singular tableau}, the submodule $U\cdot \Tab(0,0,0)$ has a basis:
\begin{align*}
\widehat{\mathcal{N}}(\Tab(0,0,0))=& \widehat{\mathcal{A}}(\Tab(0,0,0))\\
=&\mathcal{N}^{(1)}(\Tab(0,0,0))\\
=& \Tab\left(\begin{array}{c}
m\leq 0\\
n\leq 0\\
k\leq n
\end{array}\right).
\end{align*}
 Denote this module by $L_{1}$, and $M_{1}=V(T(\bar{v}))/L_{1}$.
\item[(2)] Now, the derivative tableau $\Tab(0,-2,-1)$ is $\Omega^{+}$-maximal in $M_{1}$. By Theorem \ref{Theorem: Module generated by singular tableau} , $U\cdot \Tab(0,-2,-1)$ has a basis:
\begin{align*}
\widehat{\mathcal{N}}(\Tab(0,-2,-1))=& \widehat{\mathcal{A}}(\Tab(0,0,1))\\
=& \mathcal{N}^{(1)}(\Tab(0,-2,-1))\\
=&\Tab\left(\begin{array}{c}
m\leq 0\\
n\leq 0
\end{array}\right).
\end{align*}
Moreover, by Corollary \ref{Corollary: derivative tableau with maximal Omega+ generates simple submodule}, $U\cdot_{M_{1}} \Tab(0,-2,-1)$ is a simple submodule of $M_{1}$ and has a basis
\begin{align*}
&\Tab\left(\begin{array}{c}
m\leq 0\\
n\leq 0
\end{array}\right)\setminus \Tab\left(\begin{array}{c}
m\leq 0\\
n\leq 0\\
k\leq n
\end{array}\right)\\
&= \Tab\left(\begin{array}{c}
m\leq 0\\
n\leq 0\\
n< k
\end{array}\right).\\
\end{align*}
Denote by $L_{3}$ this module and $M_{2}=M_{1}/L_{3}$.
\item[(3)] The tableau $\Tab(0,1,0)$ is $\Omega^{+}$-maximal in $M_{2}$ and $U\cdot \Tab(0,1,0)$ has a basis $\widehat{\mathcal{A}}(\Tab(0,1,0))\cup \widehat{\mathcal{A}}(\Tab(0,0,0))$ which is equal to
\begin{align*}
\Tab\left(\begin{array}{c}
m\leq n\\
m\leq 0\\
k\leq m\\
k\leq n
\end{array}\right)\bigcup \Tab\left(\begin{array}{c}
m\leq 0\\
n\leq 0\\
k\leq n
\end{array}\right).
\end{align*}
Therefore, $U\cdot_{M_{2}} \Tab(0,1,0)$ has basis
\begin{align*}
&\left(\Tab\left(\begin{array}{c}
m\leq n\\
m\leq 0\\
k\leq m\\
k\leq n
\end{array}\right)\bigcup \Tab\left(\begin{array}{c}
m\leq 0\\
n\leq 0\\
k\leq n
\end{array}\right)\right)\setminus \left(\Tab\left(\begin{array}{c}
m\leq 0\\
n\leq 0\\
k\leq n
\end{array}\right)\bigcup\Tab\left(\begin{array}{c}
m\leq 0\\
n\leq 0\\
n< k
\end{array}\right)\right)\\
&= \Tab\left(
k\leq m\leq 0< n\right),
\end{align*}
call this module $L_{2}$ and $M_{3}=M_{2}/L_{2}$.
\item[(4)] There are not $\Omega^{+}$-maximal regular tableaux in $M_{3}$, so we choose the derivative tableau $\Tab(1,0,0)$ which is $\Omega^{+}$-maximal in $M_{3}$. By Corollary \ref{Corollary: derivative tableau with maximal Omega+ generates simple submodule} the module $U\cdot_{M_{3}} \Tab(1,0,0)$ is a simple submodule of $M_3$ with basis  
$$\Tab\left(
\begin{split}
k\leq n\leq 0<m
\end{split}\right)$$
call this module $L_{5}$ and $M_{4}=M_{3}/L_{5}$.
\item[(5)] The tableau $\Tab(0,1,1)$ is $\Omega^{+}$-maximal in $M_{4}$ and $U\cdot_{M_{4}} \Tab(0,1,1)$ is a simple submodule of $M_4$ with basis 
$$\Tab\left(\begin{split}
m\leq 0<n\\
m<k\leq n
\end{split}
\right)$$
call this module $L_{4}$ and $M_{5}=M_{4}/L_{4}$.
\item[(6)] The derivative tableau $\Tab(1,0,1)$ is $\Omega^{+}$-maximal in $M_{5}$. Therefore, by Corollary \ref{Corollary: derivative tableau with maximal Omega+ generates simple submodule} $U\cdot_{M_{5}} \Tab(1,0,1)$ a simple submodule of $M_5$ with basis 
$$\Tab\left(\begin{split}
n\leq 0<m\\
n<k\leq m
\end{split}
\right)$$
call this module $L_{7}$ and $M_{6}=M_{5}/L_{7}$.
\item[(7)] The tableau $\Tab(0,1,2)$ is $\Omega^{+}$-maximal in $M_{6}$ so, $U\cdot_{M_{6}} \Tab(0,1,2)$ is a simple submodule  of $M_6$ and has a basis
$$\Tab\left(\begin{split}
m\leq 0<n<k\\
\end{split}\right)$$
call this module $L_{6}$ and $M_{7}=M_{6}/L_{6}$.

\item[(8)] The tableau $\Tab(1,0,2)$ is $\Omega^{+}$-maximal in $M_{7}$ and $U\cdot_{M_{7}} \Tab(1,0,2)$ is a simple submodule of $M_7$ with a basis 
$$\Tab\left(
\begin{split}
n\leq 0<m<k\\
\end{split}\right).$$
Call this module $L_{9}$ and $M_{8}=M_{7}/L_{9}$.

\item[(9)] The tableau $\Tab(1,1,0)$ is $\Omega^{+}$-maximal in $M_{8}$. The module $U\cdot_{M_{8}} \Tab(1,2,2)$ is a simple submodule of $M_8$ with a basis 
$$\Tab\left( \begin{array}{c}
0<m\\
0<n\\
k\leq n
\end{array}\right)$$
call this module $L_{8}$ and $M_{9}=M_{8}/L_{8}$.

\item[(10)] The tableau $\Tab(1,1,2)$ is $\Omega^{+}$-maximal in $M_{9}$ and $U\cdot \Tab(1,1,2)$ has a basis $\Tab(\mathbb{Z}^{3})$. Therefore, $U\cdot_{M_{9}} \Tab(1,1,2)$ is has a basis
$$\Tab\left( \begin{array}{c}
0<m\\
0<n\\
n<k
\end{array}\right)$$
call this module $L_{10}$.

\end{itemize}
\end{example}
\begin{remark}\label{Remark: Loewy series can be described using the method of the example}
The reasoning in Example \ref{Example: decomposition of the most singular block}  can be applied also when finding the Loewy series decomposition of $V(T(\bar{v}))$. More precisely, the Loewy series decomposition of $V(T(\bar{v}))$ for $\bar{v}=(a,a,a|a,a|a)$  is:
$$L_{1} ,\  L_{2}\oplus L_{3} ,\  L_{4} ,\  L_{5}\oplus L_{6} ,\  L_{7} ,\  L_{8}\oplus L_{9} ,\  L_{10}.$$
\end{remark}
\subsection{Realizations of all simple singular Gelfand-Tsetlin $\mathfrak{sl}(3)$-modules}\label{subsection: Realizations of all simple singular Gelfand-Tsetlin modules for sl(3)}

In this section we will describe all simple objects in every block $\mathcal{GT}_{T(v)}$ defined by a singular Gelfand-Tsetlin character $\chi_{v}$ (see Definition \ref{Definition: blocks} and Remark \ref{Remark: correspondence between characters and tableaux}). Such description will include an explicit tableaux basis of each simple subquotient $M$ in $\mathcal{GT}_{T(v)}$ and  the weight support of $M$. For the weight support we will use Proposition \ref{Proposition: Basis for weight spaces in singular blocks} and the explicit basis to give a description of the weight multiplicities, when the multiplicities are finite, a picture of the weight lattice is provided. We also present the  components of the Loewy series of the universal module $V(T(v))$. A rigorous proof based on Theorem \ref{Theorem: Module generated by singular tableau}, and Corollaries \ref{Corollary: regular tableau with maximal Omega+ generates simple submodule} and \ref{Corollary: derivative tableau with maximal Omega+ generates simple submodule} was given  for Case $(C13)$, see \S \ref{Example: The singular block containing L(-rho)}, Example \ref{Example: decomposition of the most singular block}. For all other cases  the  reasoning is the same.


The simple subquotients will be defined by their corresponding sets in $\mathbb{Z}^{3}$, equivalently, by their bases in $ \mathcal{B}(T(\bar{v}))$. We should note that all subsets of $\mathbb{Z}^{3}$ that define a simple subquotient are defined by a set of inequalities of the form $a\leq b$ or $a < b$ where $a,b$ are elements in the set $\{m,n,k,0,-t,-s\}$.

As we did for the description of generic blocks we will characterize  the weight spaces for subquotients of the singular module $V(T(\bar{v}))$.
\begin{proposition}\label{Proposition: Basis for weight spaces in singular blocks}
Let $M$ be a singular Gelfand-Tsetlin module with basis of tableaux $\mathcal{B}_{M}\subseteq\mathcal{B}(T(\bar{v}))$. If $\Tab(z)\in\mathcal{B}_{M}$ is a tableau of weight $\lambda$, then the weight space $M_{\lambda}$ is spanned by the set of tableaux $\{\Tab(z+(i,-i,0))\ |\ i\in \mathbb{Z}\}\cap\mathcal{B}_{M}$.
\end{proposition}
\begin{proof} The action of the generators of $\mathfrak{h}$ in $M$ is given by the same expressions as in the case of generic modules, therefore, we can use the same argument of the proof of Proposition \ref{Proposition: Basis for weight spaces in generic blocks}.
\end{proof}

We now describe the sets $\mathcal{B}\subseteq \mathbb{Z}^{3}$ that define all simple subquotients of $V(T(\bar{v}))$.  For convenience, the modules  listed in one row are isomorphic. Recall that $\tau$ denote the transposition $(\id,(1,2),\id)\in S_{3}\times S_{2}\times S_{1}$. It is worth noting that all isomorphisms between simple subquotients of $V(T(\bar{v}))$ are $\tau$-induced, that is all isomorphisms between simple subquotients are given by $L(\mathcal{B}) \mapsto L(\tau(\mathcal{B}))$.
\begin{remark}
In general it is not true that if $\mathcal{B}\subseteq \mathbb{Z}^{3}$ defines a subquotient of $V(T(\bar{v}))$ then $\tau(\mathcal{B})$ defines also a subquotient of $V(T(\bar{v}))$. 
\end{remark}

\begin{remark}
We should note that for singular $\mathfrak{sl}(3)$-modules we may have characters with unique simple extension or with two non-isomorphic simple extensions. In particular, the number of simple subquotients of the singular blocks in general will not coincide with the number of non-isomorphic modules in the block. 
\end{remark}

Until the end of this section we use the following convention. The entries of the Gelfand-Tsetlin tableaux we will use will be integer shifts of some of the complex numbers $a,b,c,x,z$. We also assume that if any two of $a,b,c,x,z$ appear in the same row or in consecutive rows of a given tableau, then their difference is not integer. 
For convenience, in the description of basis of simple subquotients, isomorphic modules are listed in the same row.

\begin{itemize}
\item[$(C1)$] Consider the Gelfand-Tsetlin tableau:

\begin{center}

 \hspace{1.5cm}\Stone{$a$}\Stone{$b$}\Stone{$c$}\\[0.2pt]
 $T(\bar{v})$=\hspace{0.3cm}   \Stone{$x$}\Stone{$x$}\\[0.2pt]
 \hspace{1.3cm}\Stone{$z$}\\
\end{center}
\noindent
In this case the module $V(T(\bar{v}))$ is simple and all its weight spaces are infinite dimensional.

\begin{center}
\begin{tabular}{|>{$}l<{$}|>{$}c<{$}|}\hline
\textsf{Module} &   \hspace{0.3cm}\text{Basis}\\ \hline
L_{1} &
L(\mathbb{Z}^{3})\\ \hline
\end{tabular}
\end{center}

\item[$(C2)$] Consider the following Gelfand-Tsetlin tableau:

\begin{center}

 \hspace{1.5cm}\Stone{$a$}\Stone{$b$}\Stone{$c$}\\[0.2pt]
 $T(\bar{v})$=\hspace{0.3cm}   \Stone{$x$}\Stone{$x$}\\[0.2pt]
 \hspace{1.3cm}\Stone{$x$}\\
\end{center}
\noindent
\begin{itemize}
\item[I.] \textbf{Number of simples in the block}: $2$\\

\item[II.] \textbf{simple subquotients}.\\
In this case the module $V(T(\bar{v}))$ has  $2$ simple subquotients  and they have infinite-dimensional weight multiplicities:

\begin{center}
\begin{tabular}{|>{$}l<{$}|>{$}c<{$}|}\hline
\textsf{Module} &   \hspace{0.3cm}\text{Basis}\\ \hline
L_{1}&L\left(
k\leq n
\right)\\\hline
L_{2}&L\left(
n<k
\right)\\\hline
\end{tabular}
\end{center} 

\vspace{0.3cm}
\item[III.] \textbf{Loewy series.}\\
$$L_{1} ,\  L_{2}.$$
\end{itemize}


\item[$(C3)$] Consider the following Gelfand-Tsetlin tableau:

\begin{center}

 \hspace{1.5cm}\Stone{$a$}\Stone{$b$}\Stone{$c$}\\[0.2pt]
 $T(\bar{v})$=\hspace{0.3cm}   \Stone{$a$}\Stone{$a$}\\[0.2pt]
 \hspace{1.3cm}\Stone{$z$}\\
\end{center}
\noindent
\begin{itemize}
\item[I.] \textbf{Number of simples in the block}: $2$\\

\item[II.] \textbf{simple subquotients}.\\
We have $2$ simple subquotients and they have infinite-dimensional weight spaces. 

\begin{center}
\begin{tabular}{|>{$}l<{$}|>{$}c<{$}|}\hline
\textsf{Module} &   \hspace{0.3cm}\text{Basis}\\ \hline
L_{1}&L\left(
m\leq 0
\right)\\\hline
L_{2}&L\left(
0<m
\right)\\\hline
\end{tabular}
\end{center} 

\noindent
\vspace{0.3cm}
\item[III.] \textbf{Loewy series.}\\
$$L_{1} ,\  L_{2}$$
\end{itemize}

\item[$(C4)$] Consider the Gelfand-Tsetlin tableau:

\begin{center}

 \hspace{1.5cm}\Stone{$a$}\Stone{$b$}\Stone{$c$}\\[0.2pt]
 $T(\bar{v})$=\hspace{0.3cm}   \Stone{$a$}\Stone{$a$}\\[0.2pt]
 \hspace{1.3cm}\Stone{$a$}\\
\end{center}
\begin{itemize}
\item[I.] \textbf{Number of simples in the block}: $5$\\
\item[II.] \textbf{simple subquotients}.\\
In this case we have $6$ simple subquotients. The origin of the weight lattice corresponds to the $\mathfrak{sl}(3)$-weight associated to the tableau $T(\bar{v}+\delta^{11}+\delta^{21})$.
\begin{itemize}
\item[(i)] Modules with unbounded finite weight multiplicities:
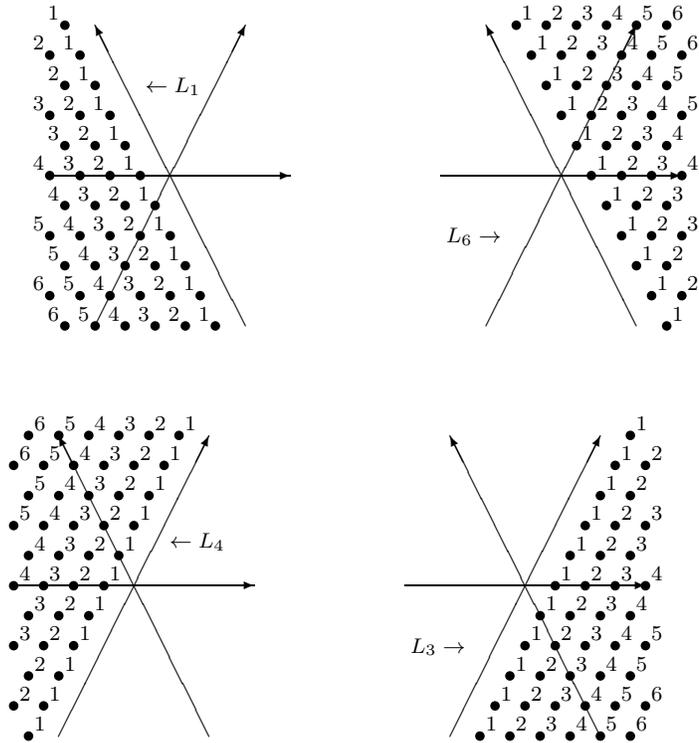
\vspace{0.3cm}
\begin{center}
\begin{tabular}{|>{$}l<{$}|>{$}c<{$}|}\hline
\textsf{Module} &   \hspace{0.3cm}\text{Basis}\\ \hline
L_{1}&L\left(\left\{\begin{array}{c}
m\leq 0\\
0<n\\
k\leq m
\end{array}\right\}\bigcup
\left\{\begin{array}{c}
m\leq 0\\
n\leq 0\\
k\leq n
\end{array}\right\}\right)\\\hline
L_{3}&L\left(\begin{array}{c}
m\leq 0\\
n<k
\end{array}\right)\\\hline
L_{4}&L\left(\begin{array}{c}
0<m\\
k\leq n
\end{array}\right)\\\hline
L_{6}&L\left(\left\{\begin{array}{c}
0<m\\
0<n\\
n<k
\end{array}\right\}\bigcup
 \left\{\begin{array}{c}
0<m\\
n\leq 0\\
m<k
\end{array}\right\}\right)\\\hline
\end{tabular}%
\end{center}

\vspace{0.3cm}

\item[(ii)] Two isomorphic modules with infinite-dimensional weight spaces:
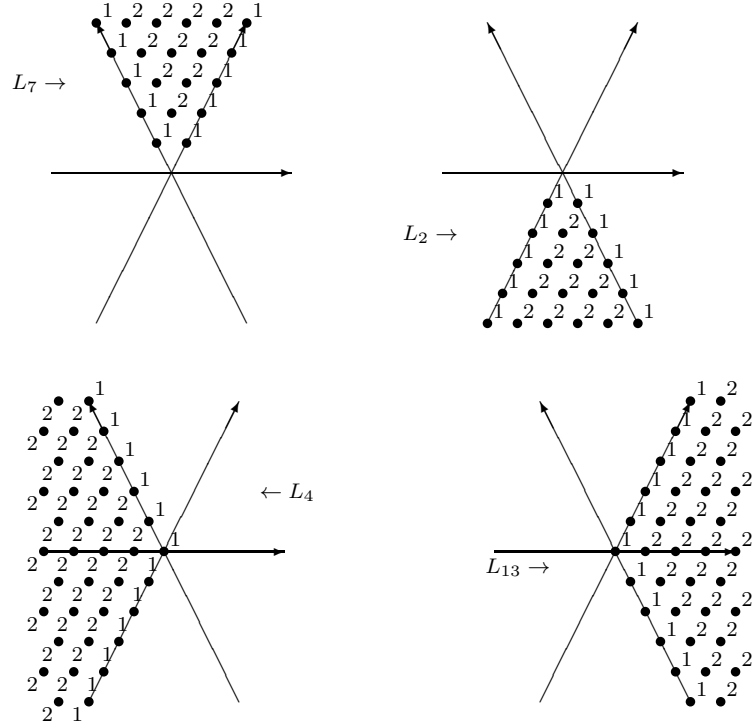
\vspace{0.3cm}
\begin{center}
\begin{tabular}{|>{$}l<{$}|>{$}c<{$}||>{$}l<{$}|>{$}c<{$}|}\hline
\textsf{Module} &   \hspace{0.3cm}\text{Basis}& \textsf{Module} &   \hspace{0.3cm}\text{Basis}\\ \hline
L_{2} &
L\left(\begin{array}{c}
m\leq 0<n\\
m<k\leq n
\end{array}\right)& L_{5}& L\left(\begin{array}{c}
n\leq 0<m\\
n<k\leq m
\end{array}\right)\\ \hline
\end{tabular}
\end{center}
\vspace{0.3cm}
\end{itemize}
\vspace{0.3cm}
\item[III.] \textbf{Loewy series.}\\
$$L_{1} ,\  L_{2} ,\  L_{3}\oplus L_{4} ,\  L_{5} ,\  L_{6}.$$

\end{itemize}


\item[$(C5)$] For any $t\in\mathbb{Z}_{> 0}$, consider the following Gelfand-Tsetlin tableau:

\begin{center}

 \hspace{1.5cm}\Stone{$a$}\Stone{$a-t$}\Stone{$c$}\\[0.2pt]
 $T(\bar{v})$=\hspace{0.3cm}   \Stone{$a$}\Stone{$a$}\\[0.2pt]
 \hspace{1.3cm}\Stone{$a$}\\
\end{center}
\begin{itemize}
\item[I.] \textbf{Number of simples in the block}: $11$\\

\item[II.] \textbf{simple subquotients}.\\

We have $16$ simple subquotients. In this case, the origin of the weight lattice corresponds to the $\mathfrak{sl}(3)$-weight associated to the tableau $T(\bar{v}-\delta^{21})$. The pictures correspond to the case $t=2$.
\noindent
\begin{itemize}
\item[(i)] Six modules with weight multiplicities bounded by $t$, two pairs of isomorphic modules and two more modules:

\begin{center}
\begin{tabular}{|>{$}l<{$}|>{$}c<{$}||>{$}l<{$}|>{$}c<{$}|}\hline
\textsf{Module} &   \hspace{0.3cm}\text{Basis}& \textsf{Module} &   \hspace{0.3cm}\text{Basis}\\ \hline
L_{2} &
L\left(\begin{array}{c}
m\leq -t<n\\
n\leq 0\\
m<k\leq n
\end{array}\right)& L_{8}& L\left(\begin{array}{c}
n\leq -t<m\\
m\leq 0\\
n<k\leq m
\end{array}\right)\\ \hline
L_{7} &
L\left(\begin{array}{c}
-t<m\leq 0\\
0<n\\
m<k\leq n
\end{array}\right)& L_{15}& L\left(\begin{array}{c}
-t<n\leq 0\\
0<m\\
n<k\leq m
\end{array}\right)\\ \hline
\end{tabular}
\end{center}

\vspace{0.5cm}

\begin{center}
\begin{tabular}{|>{$}l<{$}|>{$}c<{$}|}\hline
\textsf{Module} &   \hspace{0.3cm}\text{Basis}\\ \hline
L_{4}&L\left(\left\{\begin{array}{c}
-t<m\leq 0\\
\ \ \ 0<n\\
\ \ \ k\leq m
\end{array}\right\}\bigcup
\left\{\begin{array}{c}
-t<m\leq 0\\
\ \ \ n\leq 0\\
\ \ \ k\leq n
\end{array}\right\}\right)\\\hline
L_{13}&L\left(\left\{\begin{array}{c}
\ \ \ m\leq n\\
-t<m\leq 0\\
\ \ \ n<k
\end{array}\right\}\bigcup
 \left\{\begin{array}{c}
-t<m\leq 0\\
\ \ \ n\leq -t\\
\ \ \ m<k
\end{array}\right\}\bigcup
 \left\{\begin{array}{c}
m\leq 0\\
-t< n\\
\ \ n<k
\end{array}\right\}\right)\\\hline
\end{tabular}%
\end{center}
\vspace{0.5cm}

\item[(ii)] Modules with unbounded weight multiplicities:

\begin{center}
\begin{tabular}{|>{$}l<{$}|>{$}c<{$}||>{$}l<{$}|>{$}c<{$}|}\hline
\textsf{Module} &   \hspace{0.3cm}\text{Basis}& \textsf{Module} &   \hspace{0.3cm}\text{Basis}\\ \hline
L_{3} &
L\left(\begin{array}{c}
m\leq -t\\
0<n\\
k\leq m
\end{array}\right)& L_{10}& L\left(\begin{array}{c}
n\leq -t\\
0<m\\
k\leq n
\end{array}\right)\\ \hline
L_{9} &
L\left(\begin{array}{c}
m\leq -t\\
0<n\\
n<k
\end{array}\right)& L_{14}& L\left(\begin{array}{c}
n\leq -t\\
0<m\\
m<k
\end{array}\right)\\ \hline
\end{tabular}
\end{center}
\vspace{0.5cm}
\begin{center}
\begin{tabular}{|>{$}l<{$}|>{$}c<{$}|}\hline
\textsf{Module} &   \hspace{0.3cm}\text{Basis}\\ \hline
L_{1}&L\left( \left\{\begin{array}{c}
m\leq -t<n\leq 0\\
k\leq m
\end{array}\right\}\bigcup \left\{\begin{array}{c}
m\leq -t\\
n\leq -t\\
k\leq n
\end{array}\right\}\right)\\\hline
L_{6}&L\left(\left\{\begin{array}{c}
m\leq -t< n\leq 0\\
n<k
\end{array}\right\}\bigcup
\left\{\begin{array}{c}
m\leq -t\\
n\leq -t\\
n<k\\
\end{array}\right\}\right)\\\hline
L_{11}&L\left(\left\{\begin{array}{c}
m\leq n\\
0<m\\
k\leq n
\end{array}\right\}\bigcup
\left\{\begin{array}{c}
0<m\\
-t<n\leq 0\\
k\leq n
\end{array}\right\}\right)\\\hline
L_{16}& L\left(\left\{\begin{array}{c}
m\leq n\\
0<m\\
n<k
\end{array}\right\}\bigcup
\left\{\begin{array}{c}
0<m\\
-t<n\leq 0\\
m<k\\
\end{array}\right\}\bigcup
\left\{\begin{array}{c}
n<m\\
0<n\\
n<k
\end{array}\right\}\right)\\\hline
\end{tabular}%
\end{center}

\vspace{0.4cm}

 \item[(iii)] Two isomorphic modules with infinite-dimensional weight spaces:
 \vspace{0.3cm}
\begin{center}
\begin{tabular}{|>{$}l<{$}|>{$}c<{$}||>{$}l<{$}|>{$}c<{$}|}\hline
\textsf{Module} &   \hspace{0.3cm}\text{Basis}& \textsf{Module} &   \hspace{0.3cm}\text{Basis}\\ \hline
L_{5} &
L\left(\begin{array}{c}
m\leq -t\\
0<n\\
m<k\leq n
\end{array}\right)& L_{12}& L\left(\begin{array}{c}
n\leq -t\\
0<m\\
n<k\leq m
\end{array}\right)\\ \hline
\end{tabular}
\end{center}
\end{itemize}
\vspace{0.3cm}
\item[III.] \textbf{Loewy series.}\\

$L_{1} ,\  L_{2}\oplus L_{3} ,\  L_{4}\oplus L_{5}\oplus L_{6} ,\  L_{7}\oplus L_{8}\oplus L_{9}\oplus L_{10} ,\  L_{11}\oplus L_{12}\oplus L_{13} ,\  L_{14}\oplus L_{15} ,\  L_{16}.$\\

\end{itemize}


\item[$(C6)$] Set $t\in \mathbb{Z}_{>0}$ and consider the following Gelfand-Tsetlin tableau:

\begin{center}

 \hspace{1.5cm}\Stone{$a$}\Stone{$a-t$}\Stone{$c$}\\[0.2pt]
 $T(\bar{v})$=\hspace{0.3cm}   \Stone{$a$}\Stone{$a$}\\[0.2pt]
 \hspace{1.3cm}\Stone{$z$}\\
\end{center}
\begin{itemize}
\item[I.] \textbf{Number of simples in the block}: $4$\\
\item[II.] \textbf{simple subquotients}.\\
In this case we have $5$ simple subquotients. The origin of the weight lattice corresponds to the $\mathfrak{sl}(3)$-weight associated to the tableau $T(\bar{v}-\delta^{21})$. We provide the pictures corresponding to the case $t=2$.

\begin{itemize}

\item[(i)] Two modules with unbounded weight multiplicities:
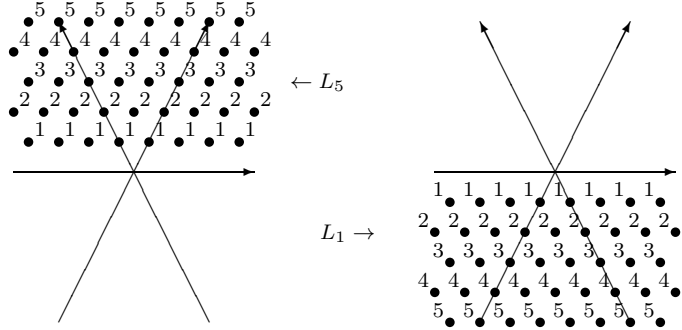
\vspace{0.3cm}
\begin{center}
\begin{tabular}{|>{$}l<{$}|>{$}c<{$}|}\hline
\textsf{Module} &   \hspace{0.3cm}\text{Basis}\\ \hline
L_{1}&L\left(\left\{\begin{array}{c}
m\leq n\\
m\leq -t\\
n\leq 0
\end{array}\right\}\bigcup
\left\{\begin{array}{c}
n<m\\
m\leq -t
\end{array}\right\}\right)\\\hline
L_{5}&L\left(\left\{\begin{array}{c}
m\leq n\\
0<m
\end{array}\right\}\bigcup
\left\{\begin{array}{c}
n<m\\
0<m\\
-t<n
\end{array}\right\}\right)\\\hline
\end{tabular}%
\end{center}

  \vspace{0.5cm}

\item[(ii)] A cuspidal module with $t$-dimensional weight multiplicities:
\vspace{0.5cm}
\begin{center}
\begin{tabular}{|>{$}l<{$}|>{$}c<{$}|}\hline
\textsf{Module} &   \hspace{0.3cm}\text{Basis}\\ \hline
L_{3} &
L\left(\begin{array}{c}
-t<m\leq 0
\end{array}\right)\\ \hline
\end{tabular}
\end{center}
\vspace{0.5cm}
\item[(iii)] Two isomorphic modules with infinite-dimensional weight spaces:
\vspace{0.5cm}
\begin{center}
\begin{tabular}{|>{$}l<{$}|>{$}c<{$}||>{$}l<{$}|>{$}c<{$}|}\hline
\textsf{Module} &   \hspace{0.3cm}\text{Basis}& \textsf{Module} &   \hspace{0.3cm}\text{Basis}\\ \hline
L_{2} &
L\left(\begin{array}{c}
m\leq -t\\
0<n
\end{array}\right)& L_{4}& L\left(\begin{array}{c}
n\leq -t\\
0<m
\end{array}\right)\\ \hline
\end{tabular}
\end{center}

\end{itemize}
\vspace{0.3cm}
\item[III.] \textbf{Loewy series.}\\
$$L_{1} ,\  L_{2} ,\  L_{3} ,\  L_{4} ,\  L_{5}$$
\end{itemize}


\item[$(C7)$] Consider the following Gelfand-Tsetlin tableau:

\begin{center}

 \hspace{1.5cm}\Stone{$a$}\Stone{$a$}\Stone{$c$}\\[0.2pt]
 $T(\bar{v})$=\hspace{0.3cm}   \Stone{$a$}\Stone{$a$}\\[0.2pt]
 \hspace{1.3cm}\Stone{$a$}\\
\end{center}
\noindent
\begin{itemize}
\item[I.] \textbf{Number of simples in the block}: $7$\\
\item[II.] \textbf{simple subquotients}.\\
The module $V(T(\bar{v}))$ has $10$ simple subquotients. In this case, the origin of the weight lattice corresponds to the $\mathfrak{sl}(3)$-weight associated to the tableau $T(\bar{v}+\delta^{21}+\delta^{11})$.

\begin{itemize}

\item[(i)] Eight modules with unbounded weight multiplicities:

\begin{center}
\begin{tabular}{|>{$}l<{$}|>{$}c<{$}|}\hline
\textsf{Module} &   \hspace{0.3cm}\text{Basis}\\ \hline
L_{1}&L\left( \begin{array}{c}
m\leq 0\\
n\leq 0\\
k\leq n
\end{array}\right)\\\hline
L_{2}&L\left( \begin{array}{c}
m\leq 0\\
n\leq 0\\
n<k
\end{array}\right)\\\hline
L_{9}&L\left( \begin{array}{c}
0<m\\
0<n\\
k\leq n
\end{array}\right)\\\hline
L_{10}&L\left( \begin{array}{c}
0<m\\
0<n\\
n<k
\end{array}\right)\\\hline
\end{tabular}
\end{center}
\vspace{0.5cm}
\begin{center}
\begin{tabular}{|>{$}l<{$}|>{$}c<{$}||>{$}l<{$}|>{$}c<{$}|}\hline
\textsf{Module} &   \hspace{0.3cm}\text{Basis}& \textsf{Module} &   \hspace{0.3cm}\text{Basis}\\ \hline
L_{3} &
L\left(\begin{array}{c}
k\leq m\leq 0<n
\end{array}\right)& L_{5}& L\left(\begin{array}{c}
k\leq n\leq 0<m
\end{array}\right)\\ \hline
L_{6} &
L\left(\begin{array}{c}
m\leq 0<n<k
\end{array}\right)& L_{8}& L\left(\begin{array}{c}
n\leq 0<m<k
\end{array}\right)\\ \hline
\end{tabular}%
\end{center}

\vspace{0.4cm}

\item[(ii)] Two isomorphic modules with infinite-dimensional weight spaces:
\vspace{0.3cm}
\begin{center}
\begin{tabular}{|>{$}l<{$}|>{$}c<{$}||>{$}l<{$}|>{$}c<{$}|}\hline
\textsf{Module} &   \hspace{0.3cm}\text{Basis}& \textsf{Module} &   \hspace{0.3cm}\text{Basis}\\ \hline
L_{4} &
L\left(\begin{array}{c}
m\leq 0<n\\
m<k\leq n
\end{array}\right)& L_{7}& L\left(\begin{array}{c}
n\leq 0<m\\
n<k\leq m
\end{array}\right)\\ \hline
\end{tabular}
\end{center}
\end{itemize}
\vspace{0.3cm}
\item[III.] \textbf{Loewy series.}\\
$$L_{1} ,\  L_{2}\oplus L_{3} ,\  L_{4} ,\  L_{5}\oplus L_{6} ,\  L_{7} ,\  L_{8}\oplus L_{9} ,\  L_{10}.$$

\end{itemize}


\item[$(C8)$] Consider the following Gelfand-Tsetlin tableau:

\begin{center}

 \hspace{1.5cm}\Stone{$a$}\Stone{$a$}\Stone{$c$}\\[0.2pt]
 $T(\bar{v})$=\hspace{0.3cm}   \Stone{$a$}\Stone{$a$}\\[0.2pt]
 \hspace{1.3cm}\Stone{$z$}\\
\end{center}
\noindent
\begin{itemize}
\item[I.] \textbf{Number of simples in the block}: $3$\\
\item[II.] \textbf{simple subquotients}.\\
The module $V(T(\bar{v}))$ has $4$ simple subquotients. In this case, the origin of the weight lattice corresponds to the $\mathfrak{sl}(3)$-weight associated to the tableau $T(\bar{v}+\delta^{21}+\delta^{11})$.
\begin{itemize}

\item[(i)] Modules with unbounded weight multiplicities:
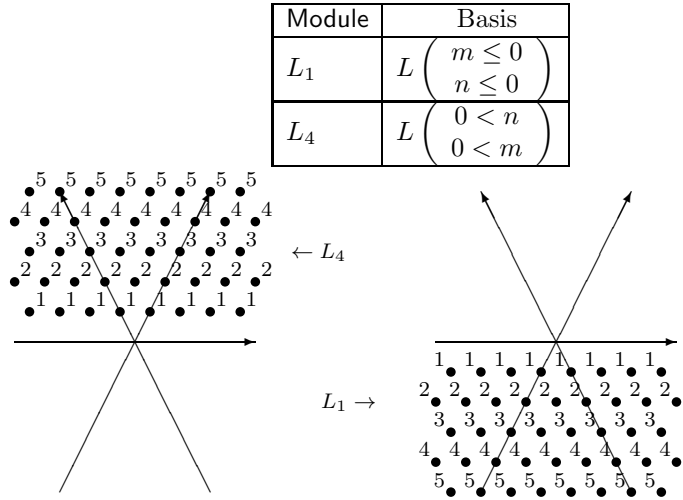
\vspace{0.3cm}
\begin{center}
\begin{tabular}{|>{$}l<{$}|>{$}c<{$}|}\hline
\textsf{Module} &   \hspace{0.3cm}\text{Basis}\\ \hline
L_{1}&L\left( \begin{array}{c}
m\leq 0\\
n\leq 0
\end{array}\right)\\\hline
L_{4}&L\left( \begin{array}{c}
0<n\\
0<m
\end{array}\right)\\\hline
\end{tabular}%
\end{center}

\vspace{0.3cm}

\item[(ii)] Two isomorphic modules with infinite-dimensional weight multiplicities:
\vspace{0.3cm}
\begin{center}
\begin{tabular}{|>{$}l<{$}|>{$}c<{$}||>{$}l<{$}|>{$}c<{$}|}\hline
\textsf{Module} &   \hspace{0.3cm}\text{Basis}& \textsf{Module} &   \hspace{0.3cm}\text{Basis}\\ \hline
L_{2} &
L\left(\begin{array}{c}
m\leq 0<n
\end{array}\right)& L_{3}& L\left(\begin{array}{c}
n\leq 0<m
\end{array}\right)\\ \hline
\end{tabular}
\end{center}
\end{itemize}
\vspace{0.3cm}
\item[III.] \textbf{Loewy series.}\\
$$L_{1} ,\  L_{2}\oplus L_{3} ,\  L_{4}.$$

\end{itemize}


\item[$(C9)$] Let $s,t\in\mathbb{Z}_{> 0}$ be such that $t<s$. Consider the following Gelfand-Tsetlin tableau:

\begin{center}

 \hspace{1.5cm}\Stone{$a$}\Stone{$a-t$}\Stone{$a-s$}\\[0.2pt]
 $T(\bar{v})$=\hspace{0.3cm}   \Stone{$a$}\Stone{$a$}\\[0.2pt]
 \hspace{1.3cm}\Stone{$z$}\\
\end{center}
\noindent
\begin{itemize}
\item[I.] \textbf{Number of simples in the block}: $7$\\
\item[II.] \textbf{simple subquotients}.\\
The module $V(T(\bar{v}))$ has $10$ simple subquotients. In this case, the origin of the weight lattice corresponds to the $\mathfrak{sl}(3)$-weight associated to the tableau $T(\bar{v}-\delta^{22})$. The pictures correspond to the case $t=1$, $s=2$.

\begin{itemize}

\item[(i)] Two modules with unbounded weight multiplicities:
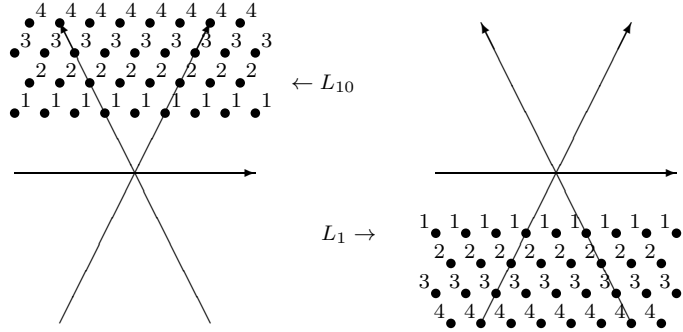
\vspace{0.3cm}
\begin{center}
\begin{tabular}{|>{$}l<{$}|>{$}c<{$}|}\hline
\textsf{Module} &   \hspace{0.3cm}\text{Basis}\\ \hline
L_{1}&L\left( \left\{\begin{array}{c}
m\leq n\\
m\leq -s\\
n\leq -t
\end{array}\right\}\bigcup
\left\{\begin{array}{c}
n<m\\
m\leq -s\\
n\leq -t
\end{array}\right\}\right)\\\hline
L_{10}&L\left(\left\{ \begin{array}{c}
m\leq n\\
0<m
\end{array}\right\}\bigcup
\left\{\begin{array}{c}
n<m\\
0<m\\
-t<n
\end{array}\right\}\right)\\\hline
\end{tabular}%
\end{center}

\vspace{0.5cm}

\item[(ii)] Three modules with weight multiplicities bounded by $t$:
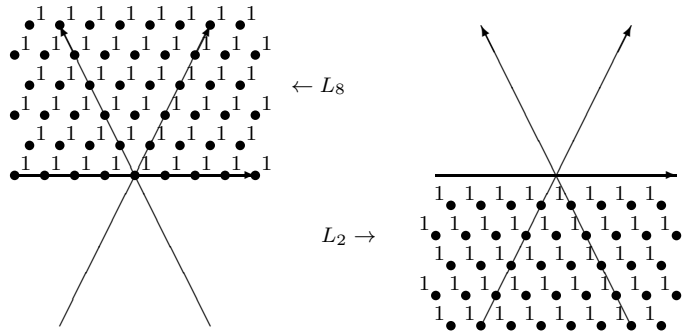
\vspace{0.3cm}
\begin{center}
\begin{tabular}{|>{$}l<{$}|>{$}c<{$}||>{$}l<{$}|>{$}c<{$}|}\hline
\textsf{Module} &   \hspace{0.3cm}\text{Basis}& \textsf{Module} &   \hspace{0.3cm}\text{Basis}\\ \hline
L_{2} &
L\left(\begin{array}{c}
m\leq -s\\
-t<n\leq 0
\end{array}\right)& L_{6}& L\left(\begin{array}{c}
n\leq -s\\
-t<m\leq 0
\end{array}\right)\\ \hline
\end{tabular}
\end{center}
\vspace{0.5cm}
\begin{center}
\begin{tabular}{|>{$}l<{$}|>{$}c<{$}|}\hline
\textsf{Module} &   \hspace{0.3cm}\text{Basis}\\ \hline
L_{8} &
L\left( \left\{\begin{array}{c}
\ \ \ m\leq n\\
-t<m\leq 0\\
\ \ \ -t<n
\end{array}\right\}\bigcup
\left\{\begin{array}{c}
-t<m\leq 0\\
-s<n\leq -t
\end{array}\right\}\right)\\ \hline
\end{tabular}%
\end{center}

  \vspace{0.5cm}

\item[(iii)] Three modules with weight multiplicities bounded by $s-t$:
\vspace{0.3cm}
\begin{center}
\begin{tabular}{|>{$}l<{$}|>{$}c<{$}||>{$}l<{$}|>{$}c<{$}|}\hline
\textsf{Module} &   \hspace{0.3cm}\text{Basis}& \textsf{Module} &   \hspace{0.3cm}\text{Basis}\\ \hline
L_{5} &
L\left(\begin{array}{c}
-s<m\leq -t\\
0<n
\end{array}\right)& L_{9}& L\left(\begin{array}{c}
-s<n\leq -t\\
0<m
\end{array}\right)\\ \hline
\end{tabular}
\end{center}

\vspace{0.5cm}

\begin{center}
\begin{tabular}{|>{$}l<{$}|>{$}c<{$}|}\hline
\textsf{Module} &   \hspace{0.3cm}\text{Basis}\\ \hline
L_{3} &
L\left( \left\{\begin{array}{c}
\ \ \ m\leq n\\
-s<m\leq -t\\
\ \ \ n\leq 0
\end{array}\right\}\bigcup
\left\{\begin{array}{c}
-s<m\leq -t\\
\ \ \ n\leq -s
\end{array}\right\}\right)\\ \hline
\end{tabular}%
\end{center}

  \vspace{0.5cm}

\item[(iv)] Two isomorphic modules with infinite-dimensional weight spaces
\vspace{0.3cm}
\begin{center}
\begin{tabular}{|>{$}l<{$}|>{$}c<{$}||>{$}l<{$}|>{$}c<{$}|}\hline
\textsf{Module} &   \hspace{0.3cm}\text{Basis}& \textsf{Module} &   \hspace{0.3cm}\text{Basis}\\ \hline
L_{4} &
L\left(\begin{array}{c}
m\leq -s\\
0<n
\end{array}\right)& L_{7}& L\left(\begin{array}{c}
n\leq -s\\
0<m
\end{array}\right)\\ \hline
\end{tabular}
\end{center}
\end{itemize}
\vspace{0.3cm}
\item[III.] \textbf{Loewy series.}\\
$$L_{1} ,\  L_{2} ,\  L_{3}\oplus L_{4} ,\  L_{5}\oplus L_{6} ,\  L_{7}\oplus L_{8} ,\  L_{9} ,\  L_{10}.$$

\end{itemize}


\item[$(C10)$] Let $t,s\in\mathbb{Z}_{>0}$ be such that $t<s$. Consider the  following Gelfand-Tsetlin tableau:

\begin{center}

 \hspace{1.5cm}\Stone{$a$}\Stone{$a-t$}\Stone{$a-s$}\\[0.2pt]
 $T(\bar{v})$=\hspace{0.3cm}   \Stone{$a$}\Stone{$a$}\\[0.2pt]
 \hspace{1.3cm}\Stone{$a$}\\
\end{center}
\noindent
\begin{itemize}
\item[I.] \textbf{Number of simples in the block}: $20$\\
\item[II.] \textbf{simple subquotients}.\\

The module $V(T(\bar{v}))$ has $32$ simple subquotients, two of them are isomorphic to the simple finite dimensional module with highest weight $\lambda=(t-1,s-t-1)$. Also, there are two isomorphic modules with infinite-dimensional weight spaces. In this case, the origin of the weight lattice corresponds to the $\mathfrak{sl}(3)$-weight associated to the tableau $T(\bar{v}-\delta^{22})$. We provide the pictures corresponding to the case $t=1$, $s=2$ (i.e. the principal block).

\noindent
\begin{itemize}
\item[(i)] Two isomorphic finite dimensional modules with weight multiplicities of degree $\min\{t,s-t\}$ and highest weight $\lambda=(t-1, s-t-1)$.
\vspace{0.3cm}
\begin{center}
\begin{tabular}{|>{$}l<{$}|>{$}c<{$}||>{$}l<{$}|>{$}c<{$}|}\hline
\textsf{Module} &   \hspace{0.3cm}\text{Basis}& \textsf{Module} &   \hspace{0.3cm}\text{Basis}\\ \hline
L_{8} &
L\left(\begin{array}{c}
-s<m\leq -t\\
-t<n\leq 0\\
m<k\leq n
\end{array}\right)& L_{20}& L\left(\begin{array}{c}
-s<n\leq -t\\
-t<m\leq 0\\
n<k\leq m
\end{array}\right)\\ \hline
\end{tabular}
\end{center}

\medskip
\item[(ii)] Twenty weight modules with weight multiplicities bounded by $t$ or $s-t$. There are eight pairs of isomorphic modules and four more modules in the list:

\medskip
\begin{center}
\begin{tabular}{|>{$}l<{$}|>{$}c<{$}||>{$}l<{$}|>{$}c<{$}|}\hline
\textsf{Module} &   \hspace{0.3cm}\text{Basis}& \textsf{Module} &   \hspace{0.3cm}\text{Basis}\\ \hline
L_{2}&L\left(\begin{array}{c}
m\leq -s\ \ \\
-t<n\leq 0\\
k\leq m\ \ \ 
\end{array}\right)&L_{9}& L\left(\begin{array}{c}
n\leq -s\ \ \\
-t<m\leq 0\\
k\leq n\ \ \ 
\end{array}\right)\\\hline
L_{3}&L\left(\begin{array}{c}
m\leq -s<n\\
n\leq -t\\
m<k\leq n\ \ 
\end{array}\right)& L_{12} & L\left(\begin{array}{c}
n\leq -s<m\\
m\leq -t\\
n<k\leq m\ \  
\end{array}\right)\\\hline
L_{6}&L\left(\begin{array}{c}
m\leq -s\ \ \\
-t<n\leq 0\\
m<k\leq n
\end{array}\right)&
L_{15}&L\left(\begin{array}{c}
n\leq -s\ \ \\
-t<m\leq 0\\
n<k\leq m
\end{array}\right)\\\hline
L_{10}&L\left(\begin{array}{c}
m\leq -s\ \ \\
-t<n\leq 0\\
n<k\ \ \ 
\end{array}\right)&
 L_{22}& 
L\left(\begin{array}{c}
n\leq -s\ \ \\
-t<m\leq 0\\
m<k\ \ \ 
\end{array}\right)\\\hline
L_{11}&L\left(\begin{array}{c}
-s<m\leq -t\\
0<n\ \ \ \\
k\leq m\ \ \ 
\end{array}\right)& L_{23} & L\left(\begin{array}{c}
-s<n\leq -t\\
0<m\ \ \ \\
k\leq n\ \ \ 
\end{array}\right)\\\hline 
L_{16}&L\left(\begin{array}{c}
-s<m\leq -t\\
0<n\ \ \ \\
m<k\leq n
 \end{array}\right)& L_{27}& L\left(\begin{array}{c}
-s<n\leq -t\\
0<m\ \ \ \\
n<k\leq m
\end{array}\right)\\\hline
L_{21}&L\left(\begin{array}{c}
-t<m\leq 0\\
0<n\ \ \ \\
m<k\leq n
\end{array}\right)&L_{30}&L\left(\begin{array}{c}
-t<n\leq 0\\
0<m\ \ \ \\
n<k\leq m
\end{array}\right)\\\hline
L_{24}&L\left(\begin{array}{c}
-s<m\leq -t\\
0<n\ \ \ \\
n<k\ \ \ 
\end{array}\right)& L_{31} & L\left(\begin{array}{c}
-s<n\leq -t\\
0<m\ \ \ \\
m<k\ \ \ 
\end{array}\right)\\\hline
\end{tabular}
\end{center}

\vspace{0.5cm}

\begin{center}
\begin{tabular}{|>{$}l<{$}|>{$}c<{$}|}\hline
\textsf{Mod} &   \hspace{0.3cm}\text{Basis}\\ \hline
L_{7}& L\left(\left\{\begin{array}{c}
\ \ \ m\leq n\\
-s<m\leq -t\\
\ \ \ n\leq 0\\
\ \ \ k\leq m
\end{array}\right\}\bigcup\left\{\begin{array}{c}
\ \ \ n<m\\
-s<m\leq -t\\
\ \ \ n\leq -t\\
\ \ \ k\leq n
\end{array}\right\}\right)\\\hline
L_{14} & L\left(\left\{\begin{array}{c}
-t<m\leq 0\\
0<n\\
k\leq m\\
\end{array}\right\}\bigcup\left\{\begin{array}{c}
\ \ \ m\leq n\\
-t<m\leq 0\\
\ \ k\leq n\leq 0
\end{array}\right\}\bigcup\left\{\begin{array}{c}
\ \ \ n<m\\
-t<m\leq 0\\
-s<n\leq 0\\
\ \ \ k\leq n
\end{array}\right\}\right)\\\hline
L_{17}&L\left(\left\{\begin{array}{c}
m\leq n\\
-s<m\leq -t\\
n\leq 0\\
n<k
\end{array}\right\}\bigcup\left\{\begin{array}{c}
n\leq -s<m\\
m\leq -t\\
m<k\\
\end{array}\right\}\bigcup\left\{\begin{array}{c}
-s<n<m\\
m\leq -t\\
m<k\\
\end{array}\right\}\right)\\\hline
L_{26}&L\left(\left\{\begin{array}{c}
m\leq n\\
-t<m\leq 0\\
n<k
\end{array}\right\}\bigcup\left\{\begin{array}{c}
n<m\\
-t<m\leq 0\\
-s<n\leq 0\\
m<k
\end{array}\right\}\right)\\\hline
\end{tabular}
\end{center}

\medskip
The following picture describes the weight support of the above listed modules. Recall that for the modules listed in the left picture, the weight multiplicities are bounded by $t$, while  for the modules in the right picture, the multiplicities are bounded by $s-t$.

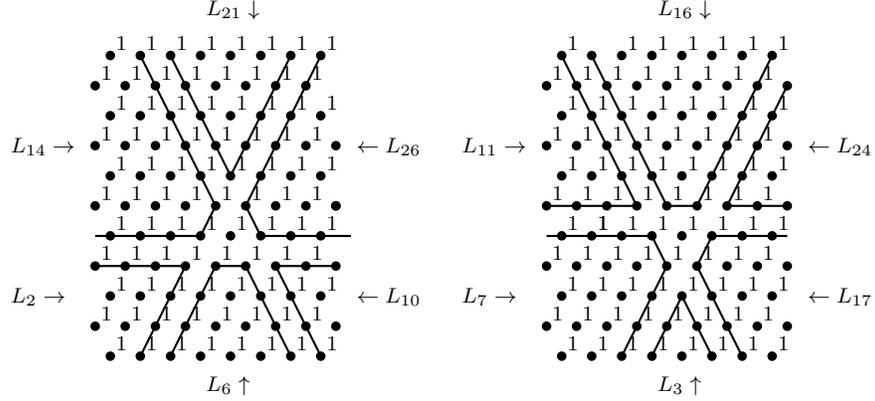
\medskip
\vspace{1cm}
\begin{center}
\psset{unit=0.2cm}
%
\end{center}

\vspace{0.4cm}
 In the pictures above, the point in the middle is the $\mathfrak{sl}(3)$-weight associated to the tableau $T(v-\delta^{22})$ and corresponds to the trivial (i.e. the finite-dimensional) module.

\item[(iii)] Eight simple Verma modules (with unbounded set of weight multiplicities). There are two pairs of isomorphic modules and four more modules in the list. 
\vspace{0.3cm}
\begin{center}
\begin{tabular}{|>{$}l<{$}|>{$}c<{$}||>{$}l<{$}|>{$}c<{$}|}\hline
\textsf{Module} &   \hspace{0.3cm}\text{Basis}& \textsf{Module} &   \hspace{0.3cm}\text{Basis}\\ \hline
L_{5}& L\left(\begin{array}{c}
k\leq m\leq -s\\
0<n
\end{array}\right)&
L_{18}& L\left(
\begin{array}{c}
k\leq n\leq -s\\
0<m
\end{array}\right)\\\hline
L_{19}&L\left(\begin{array}{c}
m\leq -s\\
0<n<k
\end{array}\right)&
L_{29}&L\left(
\begin{array}{c}
n\leq -s\\
0<m<k
\end{array}\right)\\\hline
\end{tabular}
\end{center}
\vspace{0.3cm}

\begin{center}
\begin{tabular}{|>{$}l<{$}|>{$}c<{$}|}\hline
\textsf{Module} &   \hspace{0.3cm}\text{Basis}\\ \hline
L_{1}&L\left(\left\{\begin{array}{c}
m\leq -s<n\\
n\leq -t\\
k\leq m
\end{array}\right\}\bigcup
\left\{\begin{array}{c}
m\leq n\\
n\leq -s\\
k\leq n
\end{array}\right\}\bigcup
\left\{\begin{array}{c}
n<m\\
m\leq -s\\
k\leq n
\end{array}\right\}\right)\\\hline
L_{4}&L\left(\left\{\begin{array}{c}
m\leq n\\
m\leq -s\\
n\leq -t\\
n<k
\end{array}\right\}\bigcup\left\{\begin{array}{c}
n<m\\
m\leq -s\\
n\leq -s\\
n<k
\end{array}\right\}\right)\\\hline
L_{28}&L\left(\left\{\begin{array}{c}
m\leq n\\
0<m\\
k\leq n
\end{array}\right\}\bigcup\left\{\begin{array}{c}
n<m\\
0<m\\
-t<n\\
k\leq n
\end{array}\right\}\right)\\\hline
L_{32}&L\left(\left\{\begin{array}{c}
m\leq n\\
0<m\\
n<k
\end{array}\right\}\bigcup\left\{
\begin{array}{c}
n<m\\
0<n\\
n<k
\end{array}\right\}\bigcup\left\{\begin{array}{c}
\ \ \ 0<m\\
-t<n\leq 0\\
\ \ \ m<k
\end{array}\right\}\right)\\\hline
\end{tabular}
\end{center}

\vspace{0.2cm}

\medskip
The weight supports are listed below. 
\medskip

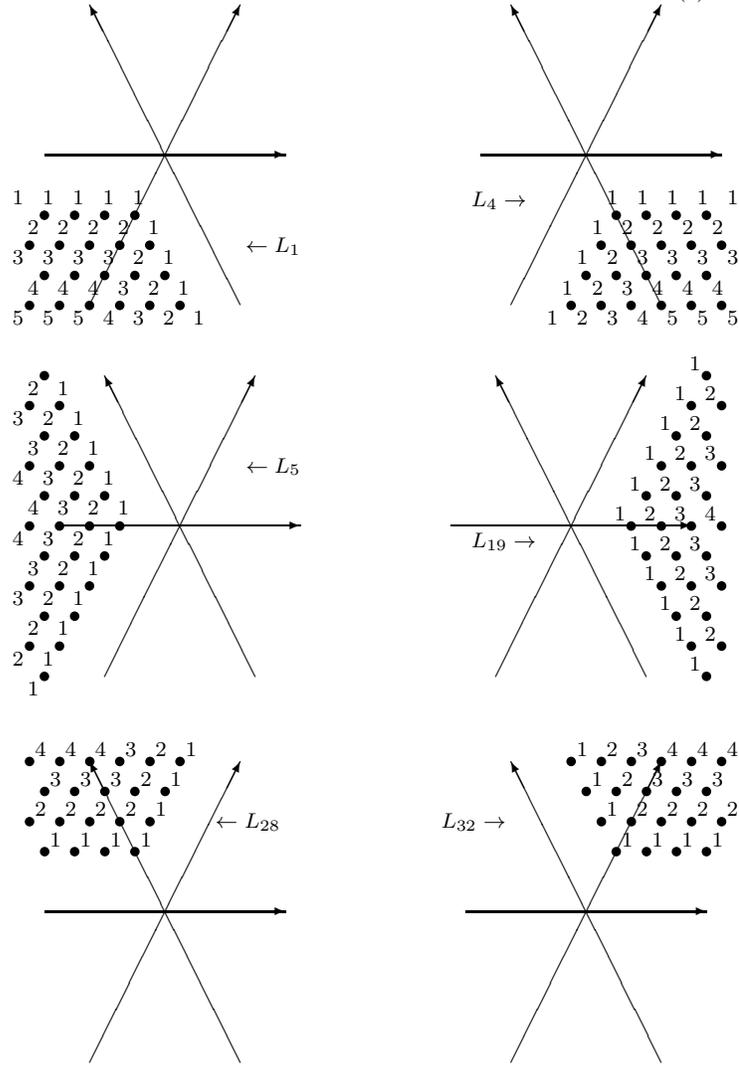
\begin{center}
\psset{unit=0.2cm}
%
\end{center}

\vspace{0.4cm}

\item[(iv)] Two weight modules with infinite weight multiplicities. In this case we have one pair of isomorphic modules:
\vspace{0.3cm}
\begin{center}
\begin{tabular}{|>{$}l<{$}|>{$}c<{$}||>{$}l<{$}|>{$}c<{$}|}\hline
\textsf{Module} &   \hspace{0.3cm}\text{Basis}& \textsf{Module} &   \hspace{0.3cm}\text{Basis}\\ \hline
L_{13}&L\left(\left\{\begin{array}{c}
m\leq -s\\
0<n\\
m<k\leq n
\end{array}\right\}
\right)& L_{25} & L\left(\left\{\begin{array}{c}
n\leq -s\\
0<m\\
n<k\leq m
\end{array}\right\}
\right)\\\hline
\end{tabular}
\end{center}

\end{itemize}
\vspace{0.3cm}
\item[III.] \textbf{Loewy series.}\\

$L_{1} ,\  L_{2}\oplus L_{3} ,\  L_{4}\oplus L_{5}\oplus L_{6}\oplus L_{7} ,\  L_{8}\oplus L_{9}\oplus L_{10}\oplus L_{11}\oplus L_{12}\oplus L_{13},\\
\  L_{14}\oplus L_{15}\oplus L_{16}\oplus L_{17}\oplus L_{18}\oplus L_{19} ,\  L_{20}\oplus L_{21}\oplus L_{22}\oplus L_{23}\oplus L_{24}\oplus L_{25} ,\\
L_{26}\oplus L_{27}\oplus L_{28}\oplus L_{29} ,\ L_{30}\oplus L_{31} ,\  L_{32}.$
\vspace{0.3cm}

\end{itemize}


\item[$(C11)$] Set $t\in\mathbb{Z}_{> 0}$ and consider the following Gelfand-Tsetlin tableau:

\begin{center}

 \hspace{1.5cm}\Stone{$a$}\Stone{$a$}\Stone{$a-t$}\\[0.2pt]
 $T(\bar{v})$=\hspace{0.3cm}   \Stone{$a$}\Stone{$a$}\\[0.2pt]
 \hspace{1.3cm}\Stone{$a$}\\
\end{center}

\begin{itemize}
\item[I.] \textbf{Number of simples in the block}: $13$\\
\item[II.] \textbf{simple subquotients}.\\
The module $V(T(\bar{v}))$ has $20$ simple subquotients. In this case, the origin of the weight lattice corresponds to the $\mathfrak{sl}(3)$-weight associated to the tableau $T(\bar{v})$. The pictures correspond to the case $t=1$.

\noindent
\begin{itemize}
\item[(i)] Six modules with finite dimensional weight spaces of dimension at most $t$, given by:
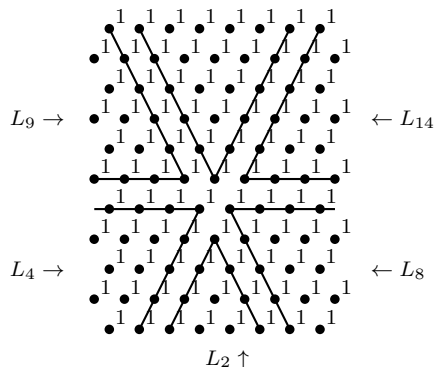
\vspace{0.3cm}
\begin{center}
\begin{tabular}{|>{$}l<{$}|>{$}c<{$}|}\hline
\textsf{Module} &   \hspace{0.3cm}\text{Basis}\\ \hline
L_{4}&L\left(\left\{\begin{array}{c}
m\leq n\\
-t<m\leq 0\\
-t<n\leq 0\\
k\leq m
\end{array}\right\}\bigcup\left\{\begin{array}{c}
-t<m\leq 0\\
k\leq n\leq -t
\end{array}\right\}\right)\\\hline
L_{8}&L\left(\left\{\begin{array}{c}
m\leq n\\
-t<m\leq 0\\
-t<n\leq 0\\
m<k
\end{array}\right\}\bigcup\left\{\begin{array}{c}
-t<m\leq 0\\
n\leq -t\\
m<k
\end{array}\right\}\right)\\\hline
\end{tabular}
\end{center}

\vspace{0.5cm}
\begin{center}
\begin{tabular}{|>{$}l<{$}|>{$}c<{$}||>{$}l<{$}|>{$}c<{$}|}\hline
\textsf{Module} &   \hspace{0.3cm}\text{Basis}& \textsf{Module} &   \hspace{0.3cm}\text{Basis}\\ \hline
L_{2}&L\left(\begin{array}{c}
m\leq -t<n\\
n\leq 0\\
m<k\leq n
\end{array}\right)&
L_{7}& L\left(\begin{array}{c}
n\leq -t<m\\
m\leq 0\\
n<k\leq m
\end{array}\right)\\\hline
L_{9}&L\left(\begin{array}{c}
-t<m\leq 0\\
 0<n\\
k\leq m
\end{array}\right)& L_{12} &L\left(\begin{array}{c}
-t<n\leq 0\\
0<m\\
k\leq n
\end{array}\right)\\\hline
L_{10}&L\left(\begin{array}{c}
-t<m\leq 0\\
0<n\\
m<k\leq n
\end{array}\right)& L_{19}&L\left(\begin{array}{c}
-t<n\leq 0\\
0<m\\
n<k\leq m
\end{array}\right)\\\hline
 L_{14}&L\left(\begin{array}{c}
-t<m\leq 0\\
0<n<k
\end{array}\right)& L_{17}&L\left(\begin{array}{c}
-t<n\leq 0\\
0<m<k
\end{array}\right)\\\hline
\end{tabular}%
\end{center}
\vspace{0.5cm}
In the pictures above, the highest weight of $L_{4}$ corresponds to the $\mathfrak{sl}(3)$-weight associated to the tableau $T(\bar{v})$.
\vspace{0.5cm}
\item[(ii)] Six modules with unbounded finite weight multiplicities:
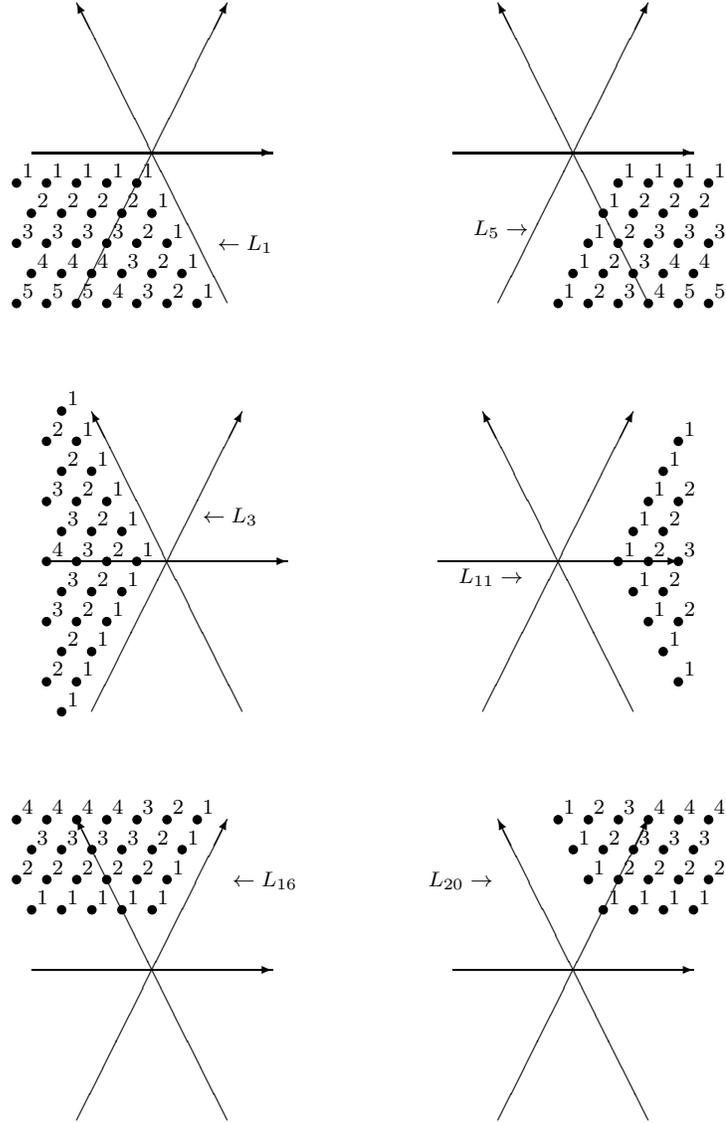
\vspace{0.3cm}

\begin{center}
\begin{tabular}{|>{$}l<{$}|>{$}c<{$}||>{$}l<{$}|>{$}c<{$}|}\hline
\textsf{Module} &   \hspace{0.3cm}\text{Basis}& \textsf{Module} &   \hspace{0.3cm}\text{Basis}\\ \hline
L_{3}&L\left( \begin{array}{c}
k\leq m\leq -t\\
0<n\\
\end{array}\right)& L_{13} & L\left( \begin{array}{c}
k\leq n\leq -t\\
0<m\\
\end{array}\right)\\\hline
L_{11}&L\left( \begin{array}{c}
m\leq -t\\
0<n<k
\end{array}\right)& L_{18}& L\left( \begin{array}{c}
n\leq -t\\
0<m<k
\end{array}\right)\\\hline
\end{tabular}
\end{center}
\hspace{0.5cm}
\begin{center}
\begin{tabular}{|>{$}l<{$}|>{$}c<{$}|}\hline
\textsf{Module} &   \hspace{0.3cm}\text{Basis}\\ \hline
L_{1}&L\left(\left\{\begin{array}{c}
m\leq n\\
m\leq -t\\
-t<n\leq 0\\
k\leq m
\end{array}\right\}\bigcup \left\{\begin{array}{c}
m\leq n\\
n\leq -t\\
k\leq n
\end{array}\right\}\bigcup \left\{\begin{array}{c}
n<m\\
m\leq -t\\
k\leq n
\end{array}\right\}\right)\\\hline
L_{5}&L\left(\left\{\begin{array}{c}
m\leq n\\
m\leq -t\\
n\leq 0\\
n<k
\end{array}\right\}\bigcup
\left\{\begin{array}{c}
n<m\\
m\leq -t\\
n<k\\
\end{array}\right\}\right)\\\hline
L_{16}&L\left(\begin{array}{c}
0<m\\
0<n\\
k\leq n
\end{array}\right)\\\hline
L_{20}&L\left(\begin{array}{c}
0<m\\
0<n\\
n<k
\end{array}\right)\\\hline
\end{tabular}%
\end{center}

\vspace{0.5cm}

 \item[(iii)] Two isomorphic modules with infinite-dimensional weight multiplicities:
\vspace{0.5cm}
\begin{center}
\begin{tabular}{|>{$}l<{$}|>{$}c<{$}||>{$}l<{$}|>{$}c<{$}|}\hline
\textsf{Module} &   \hspace{0.3cm}\text{Basis}& \textsf{Module} &   \hspace{0.3cm}\text{Basis}\\ \hline
L_{6}&L\left(\begin{array}{c}
m\leq -t\\
0<n\\
m<k\leq n
\end{array}\right)&L_{15}& L\left(\begin{array}{c}
n\leq -t\\
0<m\\
n<k\leq m
\end{array}\right)\\\hline
\end{tabular}
\end{center}
\end{itemize}
\vspace{0.3cm}
\item[III.] \textbf{Loewy series.}\\
 $L_{1} ,\ L_{2}\oplus L_{3}\oplus L_{4} ,\ L_{5}\oplus L_{6}\oplus L_{7} ,\ L_{8}\oplus L_{9}\oplus L_{10}\oplus L_{11}\oplus L_{12}\oplus L_{13} ,\\ L_{14}\oplus L_{15}\oplus L_{16} ,\  L_{17}\oplus L_{18}\oplus L_{19} ,\  L_{20}.$
 \vspace{0.3cm}
 \end{itemize} 


\item[(C12)] For any $t\in\mathbb{Z}_{>0}$  consider the tableau:

\begin{center}

 \hspace{1.5cm}\Stone{$a$}\Stone{$a$}\Stone{$a-t$}\\[0.2pt]
 $T(\bar{v})$=\hspace{0.3cm}   \Stone{$a$}\Stone{$a$}\\[0.2pt]
 \hspace{1.3cm}\Stone{$z$}\\
\end{center}
 \vspace{0.3cm}
\begin{itemize}
\item[I.] \textbf{Number of simples in the block}: $4$\\
\item[II.] \textbf{simple subquotients}.\\
We have $5$ simple subquotients. In this case, the origin of the weight lattice corresponds to the $\mathfrak{sl}(3)$-weight associated to the tableau $T(\bar{v})$. We provide pictures of the weight lattice  corresponding to the case $t=1$.

\begin{itemize}

\item[(i)] Two modules with unbounded weight multiplicities:
\vspace{0.5cm}
\begin{center}
\begin{tabular}{|>{$}l<{$}|>{$}c<{$}|}\hline
\textsf{Module} &   \hspace{0.3cm}\text{Basis}\\ \hline
L_{1}&L\left(\left\{\begin{array}{c}
m\leq n\\
m\leq -t\\
n\leq 0
\end{array}\right\}\bigcup
\left\{\begin{array}{c}
n<m\\
m\leq -t
\end{array}\right\}\right)\\\hline
L_{5}&L\left(\left\{\begin{array}{c}
m\leq n\\
0<m
\end{array}\right\}\bigcup
\left\{\begin{array}{c}
n<m\\
0<m\\
-t<n
\end{array}\right\}\right)\\\hline
\end{tabular}%
\end{center}

  \vspace{0.5cm}

\item[(ii)] A cuspidal module with $t$-dimensional weight multiplicities:
\vspace{0.3cm}

\begin{center}
\begin{tabular}{|>{$}l<{$}|>{$}c<{$}|}\hline
\textsf{Module} &   \hspace{0.3cm}\text{Basis}\\ \hline
L_{3} &
L\left( \begin{array}{c}
-t<m\leq 0
\end{array}\right)\\ \hline
\end{tabular}
\end{center}
\vspace{0.3cm}

\item[(iii)] Two isomorphic modules with infinite-dimensional weight spaces:
\vspace{0.3cm}
\begin{center}
\begin{tabular}{|>{$}l<{$}|>{$}c<{$}||>{$}l<{$}|>{$}c<{$}|}\hline
\textsf{Module} &   \hspace{0.3cm}\text{Basis}& \textsf{Module} &   \hspace{0.3cm}\text{Basis}\\ \hline
L_{2} &
L\left(\begin{array}{c}
m\leq -t\\
0<n
\end{array}\right)& L_{4}& L\left(\begin{array}{c}
n\leq -t\\
0<m 
\end{array}\right)\\ \hline
\end{tabular}
\end{center}

\end{itemize}
\vspace{0.3cm}
\item[III.] \textbf{Loewy series.}\\
$$L_{1} ,\  L_{2} ,\  L_{3} ,\  L_{4} ,\  L_{5}.$$

\end{itemize}


\item[$(C13)$] Consider the following Gelfand-Tsetlin tableau:

\begin{center}

 \hspace{1.5cm}\Stone{$a$}\Stone{$a$}\Stone{$a$}\\[0.2pt]
 $T(\bar{v})$=\hspace{0.3cm}   \Stone{$a$}\Stone{$a$}\\[0.2pt]
 \hspace{1.3cm}\Stone{$a$}\\
\end{center}
\begin{itemize}
\item[I.] \textbf{Number of simples in the block}: $7$\\
\item[II.] \textbf{simple subquotients}.\\

The module $V(T(\bar{v}))$ has $10$ simple subquotients. In this case, the origin of the weight lattice corresponds to the $\mathfrak{sl}(3)$-weight associated to the tableau $T(\bar{v}+\delta^{21}+\delta^{11})$.

\begin{itemize}

\item[(i)] Modules with unbounded weight multiplicities:

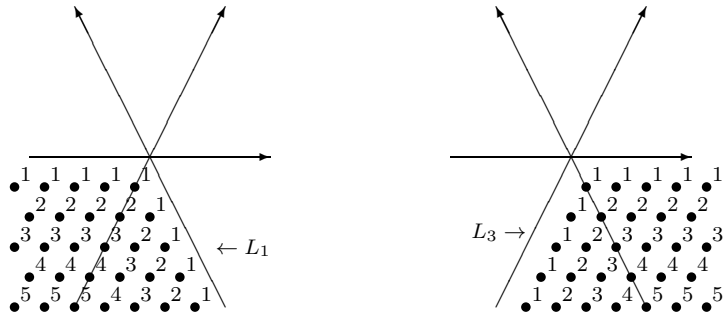
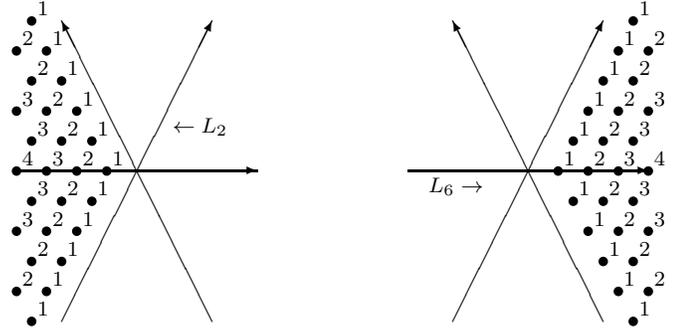
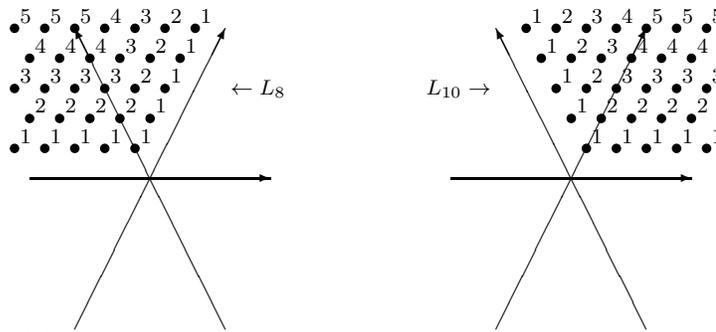
\vspace{0.3cm}
\begin{center}
\begin{tabular}{|>{$}l<{$}|>{$}c<{$}|}\hline
\textsf{Module} &   \hspace{0.3cm}\text{Basis}\\ \hline
L_{1}&L\left( \begin{array}{c}
m\leq 0\\
n\leq 0\\
k\leq n
\end{array}
\right)\\\hline
L_{3}&L\left( \begin{array}{c}
m\leq 0\\
n\leq 0\\
n<k
\end{array}\right)\\\hline
L_{8}&L\left( \begin{array}{c}
0<m\\
0<n\\
k\leq n
\end{array}\right)\\\hline
L_{10}&L\left( \begin{array}{c}
0<m\\
0<n\\
n<k
\end{array}\right)\\\hline
\end{tabular}
\end{center}

\vspace{0.5cm}

\begin{center}
\begin{tabular}{|>{$}l<{$}|>{$}c<{$}||>{$}l<{$}|>{$}c<{$}|}\hline
\textsf{Module} &   \hspace{0.3cm}\text{Basis}& \textsf{Module} &   \hspace{0.3cm}\text{Basis}\\ \hline
L_{2} &
L\left(\begin{array}{c}
k\leq m\leq 0\\
0<n
\end{array}\right)& L_{5}& L\left(\begin{array}{c}
k\leq n\leq 0\\
0<m
\end{array}\right)\\ \hline
L_{6} &
L\left(\begin{array}{c}
m\leq 0<n\\
n<k\\
\end{array}\right)& L_{9}& L\left(\begin{array}{c}
n\leq 0<m\\
m<k\\
\end{array}\right)\\ \hline
\end{tabular}%
\end{center}

\vspace{0.3cm}

\item[(ii)] Two isomorphic modules with infinite-dimensional weight spaces:
\vspace{0.3cm}
\begin{center}
\begin{tabular}{|>{$}l<{$}|>{$}c<{$}||>{$}l<{$}|>{$}c<{$}|}\hline
\textsf{Module} &   \hspace{0.3cm}\text{Basis}& \textsf{Module} &   \hspace{0.3cm}\text{Basis}\\ \hline
L_{4} &
L\left(\begin{array}{c}
m\leq 0<n\\
m<k\leq n
\end{array}\right)& L_{7}& L\left(\begin{array}{c}
n\leq 0<m\\
n<k\leq m
\end{array}\right)\\ \hline
\end{tabular}
\end{center}

\end{itemize}
\vspace{0.3cm}
\item[III.] \textbf{Loewy series.}\\
$$L_{1} ,\  L_{2}\oplus L_{3} ,\  L_{4} ,\  L_{5}\oplus L_{6} ,\  L_{7} ,\  L_{8}\oplus L_{9} ,\  L_{10}.$$
\end{itemize}


\item[$(C14)$] Consider the following Gelfand-Tsetlin tableau:

\begin{center}

 \hspace{1.5cm}\Stone{$a$}\Stone{$a$}\Stone{$a$}\\[0.2pt]
 $T(\bar{v})$=\hspace{0.3cm}   \Stone{$a$}\Stone{$a$}\\[0.2pt]
 \hspace{1.3cm}\Stone{$z$}\\
\end{center}
\begin{itemize}
\item[I.] \textbf{Number of simples in the block}: $3$\\
\item[II.] \textbf{simple subquotients}.\\
The module $V(T(\bar{v}))$ has $4$ simple subquotients. In this case, the origin of the weight lattice corresponds to the $\mathfrak{sl}(3)$-weight associated to the tableau $T(\bar{v}+\delta^{21}+\delta^{11})$.

\begin{itemize}
\item[(i)] Modules with unbounded weight multiplicities:
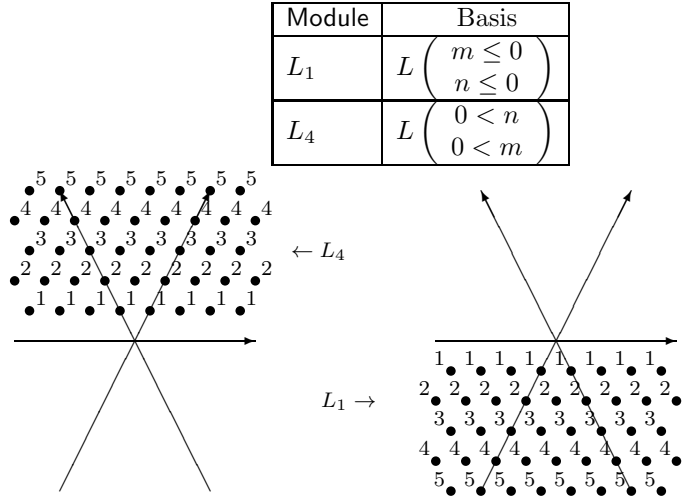
\vspace{0.3cm}
\begin{center}
\begin{tabular}{|>{$}l<{$}|>{$}c<{$}|}\hline
\textsf{Module} &   \hspace{0.3cm}\text{Basis}\\ \hline
L_{1}&L\left( \begin{array}{c}
m\leq 0\\
n\leq 0
\end{array}\right)\\\hline
L_{4}&L\left( \begin{array}{c}
0<n\\
0<m
\end{array}\right)\\\hline
\end{tabular}%
\end{center}

\vspace{0.3cm}

\item[(ii)] Two isomorphic modules with infinite-dimensional weight multiplicities:
\vspace{0.3cm}
\begin{center}
\begin{tabular}{|>{$}l<{$}|>{$}c<{$}||>{$}l<{$}|>{$}c<{$}|}\hline
\textsf{Module} &   \hspace{0.3cm}\text{Basis}& \textsf{Module} &   \hspace{0.3cm}\text{Basis}\\ \hline
L_{2} &
L\left(
m\leq 0<n
\right)& L_{3}& L\left(
n\leq 0<m
\right)\\ \hline
\end{tabular}
\end{center}
\end{itemize}
\vspace{0.3cm}
\item[III.] \textbf{Loewy series.}\\
$$L_{1} ,\  L_{2}\oplus L_{3} ,\  L_{4}.$$

\end{itemize}
\end{itemize}


\section{Localization on Gelfand-Tsetlin modules}\label{Section: Localization functors for Gelfand-Tsetlin modules}

\subsection{Localization and twisted localization functors}\label{subsection: Localization and twisted localization functors}
 We first recall the definition of the localization functor of $U$-modules.
For details we refer the reader to \cite{De} and \cite{M}.

For every  root $\alpha \in \Delta$ the multiplicative set
${\bf F}_{\alpha}:=\{ f_\alpha^n \, | \, n \in {\mathbb Z}_{\geq 0} \} \subset U$ satisfies Ore's localizability conditions
because $\mbox{ad} \, f_\alpha$ acts locally nilpotent on $U$. By
$\cD_\alpha U$ we denote the localization of $U$ relative to ${\bf F}_{\alpha}$.
For every weight module $M$, $\cD_\alpha M =
\cD_\alpha U \otimes_U M$ is the {\it
$\alpha$--localization of $M$}.  If $f_\alpha$ is injective on $M$, then $M$ can be naturally viewed as
a submodule of $\cD_\alpha M$. Furthermore, if
$f_\alpha$ is injective on $M$, then it is bijective on $M$ if and
only if $\cD_\alpha M = M$. 

 For $x \in \C$ and $u \in \cD_\alpha U$  we set
\begin{equation} \label{Equation: twisted action theta}
\Theta_x(u):= \sum_{i \geq 0} \binom{x}{i}\,
( \mbox{ad}\, f_\alpha)^i (u) \, f_\alpha^{-i},
\end{equation}
where $\binom{x}{i}= \frac{x(x-1)...(x-i+1)}{i!}$. Since $\mbox{ad}\, f_\alpha$ is locally nilpotent on
$U_\alpha$, the sum above is
actually finite. Note that for $x \in {\mathbb Z}$ we have $\Theta_x(u) =
f_\alpha^x u f_\alpha^{-x}$.  For a $\cD_\alpha U$-module $M$ by
$\Phi_\alpha^x M$ we denote the $\cD_\alpha U$-module $M$ twisted by
the action
$$
u \cdot v^x := ( \Theta_x (u)\cdot v)^x,
$$
where $u \in \cD_\alpha U$, $v \in M$, and $v^x$ stands for the
element $v$ considered as an element of $\Phi_\alpha^x M$. Since $v^n = f_{\alpha}^{-n} \cdot v$ whenever $n \in
{\mathbb Z}$ it is convenient to set $f_{\alpha}^x \cdot v :=v^{-x}$ in $\Phi_\alpha^{-x} M$ for $x
\in \C$.

In what follows we set $\cD_\alpha^x M:=\Phi_\alpha^x (\cD_\alpha M)$ and refer to it as a
{\it twisted localization of $M$}.  One easily check that if $M$ is a weight $\mathfrak g$-module, then 
$\cD_\alpha^x M$ iis a weight module as well, in particular, 
$v^x \in M_{\lambda + x \alpha}$ whenever $v \in
M_\lambda$. Furthermore, one easily verifies the following proposition.

\begin{proposition}
Let $\alpha$ be a  root and $x\in\mathbb{C}$.
\begin{itemize}
\item[(i)] ${\mathcal D}_{\alpha}^{x}$ is an exact functor from the category of $U$-modules to the category of ${\mathcal D}_{\alpha} U$-modules

\item[(ii)] If $M\subset N$ are $U$-modules such that $M$ is $f_{\alpha}$-injective and $N$ is $f_{\alpha}$-bijective, then $N = \cD_{\alpha}M$.
\end{itemize}
\end{proposition}

In the case when $f_{\alpha}$ acts injectively on $M$, set $\mathcal{QD}_{\alpha}M:={\mathcal D}_{\alpha}M/M$. Also, if $\alpha = \varepsilon_i - \varepsilon_j$ we will write $\cD_{ij}$, $\cD^x_{ij}$, and $\mathcal{QD}_{ij}$ for  $\cD_{\alpha}$, $\cD^x_{\alpha}$, and $\mathcal{QD}_{\alpha}$, respectively.


\subsection{Localization functors in the case of $\mathfrak{sl}(3)$}\label{subsection: Localization functors on sl(3)-case}  From now on we consider $U = U(\mathfrak{sl}(3))$. In this section we study the relation between the tableaux bases of a module and its localized module.

 Our goal is to apply localization functors to  Gelfand-Tsetlin $\mathfrak{sl}(3)$-modules and realize  all simple Gelfand-Tsetlin $\mathfrak{sl}(3)$-modules as subquotients of twisted localized modules.
 
 With this in mind, our first step is to  obtain conditions on the bases of the modules that guarantee injectivity or surjectivity of the operator $f_{\alpha}$. For simplicity, we will work with $\alpha = \varepsilon_1 - \varepsilon_2$, hence  $f_{\alpha} = E_{21}$.

\subsubsection{Injectivity and surjectivity of the operator $E_{21}$}\label{subsubsection: Injectivity and surjectivity of the operator $E_{21}$}

 In this subsection, we assume that  $V$ is the generic module $V(T(v))$, or the singular module $V(T(\bar{v}))$. By $\mathcal{B}$ we denote the lattice of tableaux $\mathcal{B}(T(v))$ (or $\mathcal{B}(T(\bar{v}))$). Also, the tableaux basis of  a Gelfand-Tsetlin module $M$ that is a subquotient of $V$ will be denoted by $\mathcal{B}_{M}\subseteq\mathcal{B}$.
 
\begin{remark}
Since $V$ is a weight module, every subquotient $M$ of $V$  is a weight module. Hence, in order to check injectivity or surjectivity of $E_{21}$ on $M$, it is enough to check those properties on weight spaces of $M$. Also, recall that for a weight $\lambda=(\lambda_{1},\lambda_{2})$ in the weight support of $M$,  $E_{21}(M_{\lambda})\subseteq M_{(\lambda_1-2,\lambda_{2}+1)}$.

\end{remark}

To unify the notation, in the case of a generic tableau $T(v)$ it will be convenient to write $\Tab(w):=T(v+w)$. Then, the action of $E_{21}$ on $\Tab(w)$ (generic or singular) is given by the formula:
\begin{equation}\label{Equation: action of E_{21}}
E_{21}(\Tab(w))=\Tab(w-\delta^{11})
\end{equation}

From Propositions \ref{Proposition: Basis for weight spaces in generic blocks} and \ref{Proposition: Basis for weight spaces in singular blocks}, if $\Tab(w)\in\mathcal{B}_{M}$ is a weight vector of  weight $\lambda$, then the weight space $M_{\lambda}$ is spanned by $\{\Tab(w+(i,-i,0))\ |\ i\in\mathbb{Z}\}\bigcap\mathcal{B}_{M}$. If $w_{i}$ denotes the vector $w+(i,-i,0)\in\mathbb{Z}^{3}$, any element of  $M_{\lambda}$ will be of the form $u=\sum_{i\in I}a_{i}T(w_{i})$ for some finite subset  $I$ of $\mathbb{Z}$.

\begin{lemma}\label{Lemma: E21 injective}
The operator $E_{21}$ acts injectively on $M$ if and only if $\Tab(w)\in\mathcal{B}_{M}$ implies $\Tab(w-\delta^{11})\in\mathcal{B}_{M}$.
\end{lemma}
\begin{proof}
Suppose first that there exists $\Tab(w)\in\mathcal{B}_{M}$ such that $\Tab(w-\delta^{11})\notin\mathcal{B}_{M}$, then $E_{21}(\Tab(w))=\Tab(w-\delta^{11})=0$ (on $M$), which implies that $E_{21}$ is not injective. On the other hand, suppose $u:=\sum_{i\in I}a_{i}\Tab(w_{i})\in M_{\lambda}$ with $\Tab(w_{i})\in \mathcal{B}_{M}$ for any $i$. If $u$ is such that $0=E_{21}(u)=\sum_{i\in I}a_{i}\Tab(w_{i}-\delta^{11})$, then $a_{i}T(w_{i}-\delta^{11})=0$ for any $i\in I$. Since,by hypothesis $T(w_{i}-\delta^{11})\in\mathcal{B}_{M}$, we  have $a_{i}=0$ for any $i\in I$. \end{proof}


\begin{lemma}\label{Lemma: E_21 surjective}
The operator $E_{21}$ acts surjectively on $M$ if and only if $T(w)\in\mathcal{B}_{M}$ implies $T(w+\delta^{11})\in\mathcal{B}_{M}$.
\end{lemma}
\begin{proof}
Any element of $M_{(\lambda_1-2,\lambda_{2}+1)}$ is of the form $u'=\sum\limits_{i\in I}a_{i}\Tab(w_{i}-\delta^{11})$ and a direct computation using  (\ref{Equation: action of E_{21}}) shows that $E_{21}\left(\sum\limits_{i\in I}a_{i}\Tab(w_{i})\right)=u'$. 
\end{proof}

\subsubsection{twisted localization with respect to $E_{21}$}\label{subsubsection: Localization with respect to E_21}

 Recall that for $x \in \C$ and $u \in \cD_{12} U$  we have
\begin{equation} 
\Theta_x(u):= \sum_{i \geq 0} \binom{x}{i}\,
( \mbox{ad}\, E_{21})^i (u) \, E_{21}^{-i}.
\end{equation}

\begin{lemma}\label{Lemma: twisted action of generators of Gamma}
Let $\{c_{ij}\}_{1\leq j\leq i\leq 3}$ be the generators of $\Gamma$ defined in (\ref{Equation: c_mk}). Then 
$$
\Theta_x(c_{ij})=\begin{cases} c_{ij}, & \text{ if } (i,j)\neq (1,1)\\ c_{11}+x, &  \text{ if } (i,j)=(1,1).
\end{cases}
$$
\end{lemma}
\begin{proof}
We first note that if $u$ commutes with $E_{21}$, then  $\Theta_x(u)=u$. Since  the generators $\{c_{ij}\}_{2\leq j\leq i\leq 3}$ commute with $E_{21}$, the first part of the lemma is proven. For the second part we use that $c_{11}=E_{11}$ and  $(\mbox{ad}\, E_{21})^2 (E_{11})=0$.
\end{proof}

As an immediate consequence of Lemma \ref{Lemma: twisted action of generators of Gamma} we have the following  corollary that will be frequently applied.
 \begin{corollary}\label{Corollary: Localization of E21 injective GT is GT} Let $M$ be any Gelfand-Tsetlin module on which $E_{21}$ acts injectively.
 \begin{itemize} 
 \item[(i)] The twisted localized module $\cD_{12}^x M$ is also a Gelfand-Tsetlin module. 
 \item[(ii)] If $v\in M$ has Gelfand-Tsetlin character $\chi=(a_{11}, a_{21}, a_{22}, a_{31}, a_{32}, a_{33})$, then $v^{x} \in \cD_{12}^x M$ has Gelfand-Tsetlin character $\tilde{\chi}=(a_{11}+x, a_{21}, a_{22}, a_{31}, a_{32}, a_{33})$.
 \end{itemize}
\end{corollary}

Next, for a Gelfand-Tsetlin module $M$ with tableaux basis $\mathcal{B}_{M}$ and injective action of $E_{21}$, we explicitly describe  the tableaux  basis of $\cD_{12}^x M$. For this, we introduce some notation.

For  $\mathcal{B}\subseteq\mathbb{Z}^{3}$, denote by $\mathcal{B} + \delta^{11}$ the region $\{ (m,n,k) \; | \; (m,n,k-1) \in \mathcal{B}\}$. Set  $\mathcal{B} + t \delta^{11} =  (\mathcal{B} + (t-1)\delta^{11}) + \delta^{11}$ for $t \in {\mathbb N}$ and $\mathcal{B} + {\mathbb N} \delta^{11} = \bigcup_{t=0}^{\infty} (\mathcal{B} + t \delta^{11})$.

Recall the Definition \ref{Definition: M(B)} for $L(\mathcal{B};\ v)$.
\begin{proposition}\label{Proposition: loc_regions}
Let ${\mathcal B} \subset {\mathbb Z}^3$ and $L(\mathcal{B}) = L(\mathcal{B};\ v)$ be a simple Gelfand-Tsetlin module.
Assume that $E_{21}$ acts injectively on $ L(\mathcal{B};\ v)$. Then $\cD_{12}^{x}L(\mathcal{B};\ v) \simeq M(\mathcal{B}+ {\mathbb N} \delta^{11};\ v+x\delta^{11} )$ and $\mathcal{QD}^{x}_{12}L(\mathcal{B};\ v) \simeq M((\mathcal{B} + {\mathbb N} \delta^{11})\setminus \mathcal{B};\ v+x\delta^{11} )$. In particular, if $x$ is an integer we have $\cD_{12}L(\mathcal{B}) \simeq M(\mathcal{B} + {\mathbb N} \delta^{11})$ and $\mathcal{QD}_{12}L(\mathcal{B}) \simeq M((\mathcal{B} + {\mathbb N} \delta^{11}) \setminus \mathcal{B})$.
\end{proposition}

\begin{corollary}
Let $M$ be a simple module in $\mathcal{GT}$, generated by a tableau $\Tab(w)$ and such that $E_{21}$ acts injectively on $M$. Then for any $x\in\mathbb{C}$, $\cD_{12}^x M$ has a subquotient isomorphic to a simple $\mathfrak{sl}(3)$-module in $\mathcal{GT}$ generated by the tableau $\Tab(w+x\delta^{11})$. 
\end{corollary}

\subsection{Simple Gelfand-Tsetlin modules and localization functors}\label{subsection: simple Gelfand-Tsetlin modules and localization functors}
In this section we will describe the simple Gelfand-Tsetlin $\mathfrak{sl}(3)$-modules via localization functors and subquotients starting with some simple $E_{21}$-injective Gelfand-Tsetlin module. In order to give such description we rely on Lemmas \ref{Lemma: E21 injective}, \ref{Lemma: E_21 surjective},  and Proposition \ref{Proposition: loc_regions}. In fact, in order to use Proposition \ref{Proposition: loc_regions} we have to check if the corresponding module defined in such region is $E_{21}$-bijective.
\noindent

For convenience, we denote by $L_{i}^{(Gj)}$ the simple module $L_{i}$ in the $j$th generic block from the list in \S \ref{subsection: Realization of simple generic Gelfand-Tsetlin modules for sl(3)},  and by $L_{i}^{(Cj)}$ the simple module $L_{i}$ in the $j$th  singular block in the list in \S \ref{subsection: Realizations of all simple singular Gelfand-Tsetlin modules for sl(3)}.  For example, $L_{1}^{(G2)}$ stands for the module \\ $L\left( \left\{k\leq m\right\}; T(a,b,c|x,y|x)
\right)$.

Below we list of all simple modules which are $E_{21}$-injective.
\begin{itemize}
\item[(i)] Simple $E_{21}$-injective generic modules:
\begin{center}
\begin{tabular}{|>{$}c<{$}|>{$}c<{$}|>{$}c<{$}|>{$}c<{$}|>{$}c<{$}|>{$}c<{$}|}\hline
L_{1}^{(G2)} & L_{1}^{(G4)} & L_{2}^{(G4)} &
L_{1}^{(G5)} & L_{3}^{(G5)} & L_{1}^{(G6)}\\ \hline
L_{3}^{(G6)} & L_{4}^{(G6)} & L_{7}^{(G6)} & L_{1}^{(G9)}
 & L_{3}^{(G9)}  & L_{5}^{(G9)}\\ \hline
L_{1}^{(G10)} & L_{2}^{(G10)}  & L_{4}^{(G10)} & L_{1}^{(G11)}  & L_{3}^{(G11)} & L_{4}^{(G11)}\\ \hline
L_{6}^{(G11)} & L_{8}^{(G11)}  & L_{11}^{(G11)} & L_{1}^{(G12)}  & L_{3}^{(G12)} & L_{4}^{(G12)}\\ \hline
L_{5}^{(G12)} & L_{8}^{(G12)}  & L_{11}^{(G12)} & L_{1}^{(G14)}  & L_{2}^{(G14)} & L_{4}^{(G14)}\\ \hline
L_{6}^{(G14)} & L_{1}^{(G15)}  & L_{2}^{(G15)} & L_{4}^{(G15)}  & L_{6}^{(G15)} & \\ \hline
\end{tabular}
\end{center}

\item[(ii)] Simple $E_{21}$-injective singular modules:
\begin{center}
\begin{tabular}{|>{$}c<{$}|>{$}c<{$}|>{$}c<{$}|>{$}c<{$}|>{$}c<{$}|>{$}c<{$}|>{$}c<{$}|>{$}c<{$}|}\hline
L _{1}^{(C2)} &  L_{1}^{(C4)} & L_{4}^{(C4)} & L_{1}^{(C5)} &  L_{3}^{(C5)} & L_{4}^{(C5)} &
L _{10}^{(C5)} &  L_{11}^{(C5)}\\ \hline
 L_{1}^{(C7)} & L_{3}^{(C7)} &  L_{5}^{(C7)} & L_{9}^{(C7)} &
L _{1}^{(C10)} &  L_{2}^{(C10)} & L_{5}^{(C10)} & L_{7}^{(C10)}\\ \hline  L_{9}^{(C10)} & L_{11}^{(C10)} &
L _{14}^{(C10)} &  L_{23}^{(C10)} & L_{28}^{(C10)} & L_{1}^{(C11)} &  L_{3}^{(C11)} & L_{4}^{(C11)}\\ \hline
L _{9}^{(C11)} &  L_{12}^{(C11)} & L_{13}^{(C11)} & L_{16}^{(C11)} &  L_{1}^{(C13)} & L_{2}^{(C13)} &
L_{5}^{(C13)} & L_{8}^{(C13)}\\ \hline
\end{tabular}
\end{center}
\end{itemize}

\medskip

Finally,  we apply (twisted) localization functors on the modules above and  obtain all  simple modules in the block as shown in the following tables.

\begin{center}
\begin{tabular}{|>{$}l<{$}|>{$}l<{$}|>{$}l<{$}|}\hline
\textsf{Block} &   \text{Module} &\text{Subquotient of $E_{21}$-localization}\\ \hline

G1 & L_{1}^{(G1)}&\mathcal{D}_{12}^{(z-x)}(L_{1}^{(G2)}) \\ \hline

G2 & L_{2}^{(G2)}&\mathcal{QD}_{12}(L_{1}^{(G2)})\\ \hline

G3 & L_{1}^{(G3)}& \mathcal{D}_{12}^{(y-z)}(L_{1}^{(G4)})\simeq \mathcal{D}_{12}^{(z-a)}(L_{1}^{(G5)})\\
& L_{2}^{(G3)}& \mathcal{D}_{12}^{(y-z)}(L_{2}^{(G4)})\simeq \mathcal{D}_{12}^{(z-a)}(L_{3}^{(G5)})\\ \hline

G4 & L_{3}^{(G4)}&\mathcal{QD}_{12}(L_{1}^{(G4)})\\
& L_{4}^{(G4)}&\mathcal{QD}_{12}(L_{2}^{(G4)})\\  \hline

G5 & L_{2}^{(G5)}&\mathcal{QD}_{12}(L_{1}^{(G5)})\\
& L_{4}^{(G5)}&\mathcal{QD}_{12}(L_{3}^{(G5)})\\ \hline

G6 &
 L_{2}^{(G6)}& \mathcal{QD}_{12}(L_{1}^{(G6)})\\
 & L_{5}^{(G6)}& \mathcal{QD}_{12}(L_{3}^{(G6)})\\
  & L_{6}^{(G6)}& \mathcal{QD}_{12}(L_{4}^{(G6)})\\
  & L_{8}^{(G6)}& \mathcal{QD}_{12}(L_{7}^{(G6)})\\ \hline

G7  & L_{1}^{(G7)}& \mathcal{D}_{12}^{(z-a)}(L_{1}^{(G6)}) \\
 & L_{2}^{(G7)}& \mathcal{D}_{12}^{(z-a)}(L_{4}^{(G6)})\\
 & L_{3}^{(G7)}& \mathcal{D}_{12}^{(z-a)}(L_{3}^{(G6)})\\
 & L_{4}^{(G7)}& \mathcal{D}_{12}^{(z-a)}(L_{7}^{(G6)})\\ \hline

\end{tabular}
\end{center}

\begin{center}
\begin{tabular}{|>{$}l<{$}|>{$}l<{$}|>{$}l<{$}|}\hline
\textsf{Block} &   \text{Module} &\text{Subquotient of $E_{21}$-localization}\\ \hline
G8 & L_{1}^{(G8)}& \mathcal{D}_{12}^{(z-a)}(L_{1}^{(G9)})\simeq \mathcal{D}_{12}^{(y-z)}(L_{1}^{(G10)}) \\
& L_{2}^{(G8)}&\mathcal{D}_{12}^{(z-a)}(L_{3}^{(G9)})\simeq \mathcal{D}_{12}^{(y-z)}(L_{2}^{(G10)})\\
& L_{3}^{(G8)}& \mathcal{D}_{12}^{(z-a)}(L_{5}^{(G9)})\simeq \mathcal{D}_{12}^{(y-z)}(L_{4}^{(G10)})\\ \hline

G9 & L_{2}^{(G9)}&\mathcal{QD}_{12}(L_{1}^{(G9)})\\
 & L_{4}^{(G9)}&\mathcal{QD}_{12}(L_{3}^{(G9)})\\
 & L_{6}^{(G9)}&\mathcal{QD}_{12}(L_{5}^{(G9)})\\ \hline

G10 & L_{3}^{(G10)}&\mathcal{QD}_{12}(L_{1}^{(G10)})\\
\hline

\end{tabular}
\end{center}


\begin{center}
\begin{tabular}{|>{$}l<{$}|>{$}l<{$}|>{$}l<{$}|}\hline
\textsf{Block} &   \text{Module} &\text{Subquotient of $E_{21}$-localization}\\ \hline

C1 & L_{1}^{(C1)}&\mathcal{QD}^{(z-x)}_{12}(L_{1}^{(C2)})\\ \hline

C2 & L_{2}^{(C2)}&\mathcal{QD}_{12}(L_{1}^{(C2)})\\ \hline

C3 & L_{1}^{(C3)}& \mathcal{QD}^{(z-a)}_{12}(L_{1}^{(C4)})\\
& L_{2}^{(C3)}& \mathcal{QD}^{(z-a)}_{12}(L_{4}^{(C4)})\\ \hline

C4 & L_{2}^{(C4)}\simeq L_{5}^{(C4)}& \soc(\mathcal{QD}_{12}(L_{1}^{(C4)})\simeq \soc(\mathcal{QD}_{12}(L_{4}^{(C4)}))\\
& L_{3}^{(C4)}&\mathcal{QD}_{12}(L_{1}^{(C4)})/L_{2}^{(C4)}\\
& L_{6}^{(C4)}&\mathcal{QD}_{12}(L_{4}^{(C4)})/L_{5}^{(C4)}\\
\hline

\end{tabular}
\end{center}

\begin{center}
\begin{tabular}{|>{$}l<{$}|>{$}l<{$}|>{$}l<{$}|}\hline
\textsf{Block} &   \text{Module} &\text{Subquotient of $E_{21}$-localization}\\ \hline

C10 & L_{3}^{(C10)}\simeq L_{12}^{(C10)}&\soc(\mathcal{QD}_{12}(L_{7}^{(C10)}/L_{8}^{(C10)})\simeq \soc(\mathcal{QD}_{12}(L_{1}^{(C10)}))\\
&L_{4}^{(C10)}&\mathcal{QD}_{12}(L_{1}^{(C10)})/L_{3}^{(C10)}\\
& L_{6}^{(C10)}\simeq L_{15}^{(C10)}&\soc(\mathcal{QD}_{12}(L_{2}^{(C10)})\simeq \soc(\mathcal{QD}_{12}(L_{9}^{(C10)}))\\
& L_{8}^{(C10)}\simeq L_{20}^{(C10)}&\soc(\mathcal{QD}_{12}(L_{7}^{(C10)})/L_{3}^{(C10)})\simeq \soc(\mathcal{QD}_{12}(L_{14}^{(C10)})/L_{21}^{(C10)})\\
 & L_{10}^{(C10)}&\mathcal{QD}_{12}(L_{2}^{(C10)})/L_{6}^{(C10)}\\
&
L_{13}^{(C10)}&\soc({QD}_{12}(L_{5}^{(C10)}))\\
& L_{16}^{(C10)}&\soc(\mathcal{QD}_{12}(L_{11}^{(C10)}))\\
\hline

\end{tabular}
\end{center}

\medskip

\begin{corollary}
Every simple Gelfand-Tsetlin module can be obtained via a composition of a twisted localization functor and taking a subquotient from a simple $E_{21}$-injective Gelfand-Tsetlin module.
\end{corollary}


\end{document}